\documentclass[a4paper,11pt,reqno]{amsart}

\usepackage{
	amsfonts,
	amsmath,
	amssymb,
	bbold,
	dsfont,
	mathrsfs,
	soul,
	subfig,
	xcolor
}

\usepackage{geometry}
\usepackage[hidelinks]{hyperref}
\usepackage[latin1]{inputenc}

\usepackage{tikz}
\usetikzlibrary{decorations.pathreplacing}
\usetikzlibrary{shapes.geometric,calc,decorations.markings,math}
\tikzstyle{nodo}=[circle,draw,fill,inner sep=0pt,minimum size=%
1.5mm]
\tikzstyle{infinito}=[circle,inner sep=0pt,minimum size=0mm]

\usetikzlibrary{shapes.geometric,calc,decorations.markings,math}
\tikzstyle{nodo}=[circle,draw,fill,inner sep=0pt,minimum size=0.5*width("k")]
\tikzstyle{infinito}=[circle,inner sep=0pt,minimum size=0mm]
\tikzset{every loop/.style={min distance=10mm,in=300,out=240,looseness=10}}
\tikzset{place/.style={circle,thick,draw=blue!75,fill=blue!20,minimum
		size=6mm}}
\tikzset{place2/.style={circle,thick,draw=red!75,fill=red!20,minimum
		size=6mm}}

\newtheorem{theorem}{Theorem}[section]
\newtheorem{proposition}[theorem]{Proposition}
\newtheorem{lemma}[theorem]{Lemma}

\theoremstyle{remark}
\newtheorem{remark}[theorem]{Remark}
\newtheorem*{remark*}{Remark}

\theoremstyle{definition}
\newtheorem{definition}[theorem]{Definition}

\DeclareMathOperator*{\argmin}{arg\,min}
\DeclareMathOperator*{\argmax}{arg\,max}

\newcommand\N{{\mathbb N}}
\newcommand\R{{\mathbb R}}
\newcommand\Z{{\mathbb Z}}

\newcommand\EE{{\mathcal E}}

\newcommand{\D}{\mathcal{D}}
\newcommand{\Q}{\mathcal{Q}}

\newcommand{\G}{\mathcal{G}}

\newcommand{\Sq}{\mathcal{S}}
\newcommand{\F}{\mathcal{F}}
\newcommand{\Bq}{\mathcal{B}}

\renewcommand\v{\textsc{v}}
\newcommand\w{\textsc{w}}
\newcommand\ttt{\texttt{t}}

\newcommand{\E}{\mathbf{E}}
\newcommand{\VQ}{\mathbf{V}_\Q}
\newcommand{\EQ}{\mathbf{E}_\Q}
\newcommand{\VG}{\mathbf{V}_\G}
\newcommand{\EG}{\mathbf{E}_\G}

\newcommand\f{\frac}
\newcommand\dx{{\,dx}}
\newcommand\dy{{\,dy}}
\newcommand\dt{{\,dt}}
\newcommand\ds{{\,ds}}

\newcommand\eps{\varepsilon}

\newcommand\Hmu{{H_\mu^1}}

\newcommand{\ov}[1]{\overline{#1}}

\title[Symmetry breaking in two--dimensional square grids]{Symmetry breaking in two--dimensional square grids: persistence and failure of the dimensional crossover}

\author[S. Dovetta]{Simone Dovetta}
\address[S. Dovetta]{Università degli Studi di Roma ``La Sapienza", Dipartimento di Scienze di Base ed Applicate per l'Ingegneria, via Antonio Scarpa 14, 00161 - Roma, Italy.}
\email{simone.dovetta@uniroma1.it}

\author[L. Tentarelli]{Lorenzo Tentarelli}
\address[L. Tentarelli]{Politecnico di Torino, Dipartimento di Scienze Matematiche ``G.L. Lagrange'', Corso Duca degli Abruzzi 24, 10129 - Torino, Italy.} 
\email{lorenzo.tentarelli@polito.it}

\date{\today}

\begin{document}



\begin{abstract}
We discuss the model robustness of the infinite two--dimensional square grid with respect to symmetry breakings due to the presence of defects, that is, lacks of finitely or infinitely many edges. Precisely, we study how these topological perturbations of the square grid affect the so--called \emph{dimensional crossover} identified in \cite{ADST19}. Such a phenomenon has two evidences: the coexistence of the one and the two--dimensional Sobolev inequalities and the appearence of a continuum of $L^2$--critical exponents for the ground states at fixed mass of the nonlinear Schr\"odinger equation. From this twofold perspective, we investigate which classes of defects do preserve the dimensional crossover and which classes do not.
\end{abstract}

\maketitle

\vspace{-.5cm}

\noindent {\footnotesize \textul{Keywords:} metric graphs, dimensional crossover, two--dimensional grid, defect, Sobolev inequality, nonlinear Schr\"odinger, ground states.}



    \section{Introduction}
    \label{sec:intro}  
    Since their first appearance in the Fifties of the last century \cite{RS}, \emph{metric graphs} have gained some popularity as effective models for the study of the dynamics of a wide range of physical systems with ramified geometries, from organic molecules to quantum physics. For instance, they have proved to approximate with a certain accuracy the behavior and the features of quantum wires, Bose--Einstein condensates in ramified traps and signal propagation in Kerr--type media (see, e.g., \cite{AST17,BK13,N14,P12} and references therein).
    
    As a consequence, plenty of works appeared through the years discussing several aspects and properties of differential equations and operators on graphs. From the very beginning, linear models have been considered extensively (we refer to the seminal paper \cite{KS}, as well as to the monographs \cite{BK13,Mug} for detailed discussions) and they continue to be a lively and rich research field (see for instance \cite{BKKM,BLS,EFK,FMN} and related works). Concurrently, over the last decades nonlinear problems have been gathering a constantly growing attention. In this setting, the main focus has been devoted first to the NonLinear Schr\"odinger Equation (NLSE) on \emph{noncompact} metric graphs with a \emph{finite number of edges} (see, e.g., Figure \ref{fig:grafo}), and specifically on the study of its standing waves (we refer, e.g., to \cite{AST19,BDL20,BDL21,BMP19,DST20bis,DT20,H19,NP20,PSV} for some of the latest results on the topic). Furthermore, due to the renewed interest on relativistic corrections of models from solid state and condensed matter physics (see, e.g., \cite{SBMK18}), some discussions have recently arised also on the standing waves of the nonlinear Dirac equation (see, e.g., \cite{BCT19,BCT19bis}).
    
    \begin{figure}[t]
    \centering
    \includegraphics[width=.65\columnwidth]{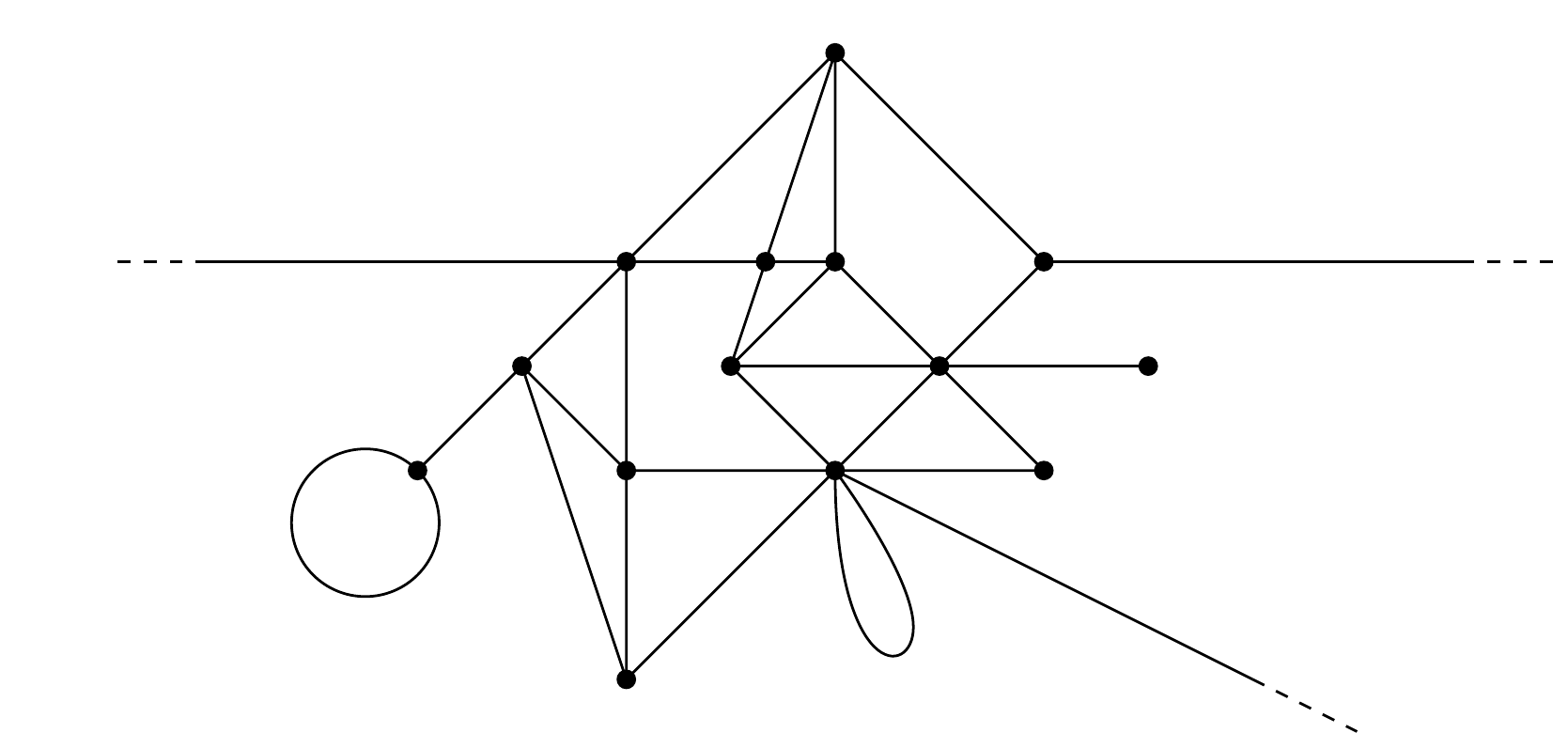}
    \caption{Example of a noncompact metric graph with a finite number of edges (three of which unbounded, denoted with $\infty$ at the end).}
    \label{fig:grafo}
    \end{figure}
    
    \medskip
    Nowadays, a major line of research concerns noncompact metric graphs with \emph{infinitely many bounded edges} and, in particular, {\em periodic graphs}, driven by their applications to the study of new materials such as for instance \emph{carbon nanutubes} and \emph{graphene} (see, e.g., \cite{BK13,KP07}). In the linear framework, the spectral analysis of Schr\"odinger operators on periodic structures is a focal topic both in the metric (see \cite{BB13,BK20,ET,KN} and references therein) and in the discrete setting \cite{FK}, also in connection with the wider subject of periodic elliptic operators (see \cite{K16} for an overview on this point).  
    
    As for nonlinear problems, several works are by now available on $\Z$--periodic graphs (see, e.g., Figure \ref{fig:periodico}), i.e. graphs built of an infinite number of copies of a fixed compact graph glued together along a given direction. The literature on the matter is getting larger and larger and we just mention a selection of the main works \cite{ABZ12,BSTU06,DPS09,D19,PS10,PS16}.
    
    \begin{figure}[t]
    \centering
    \includegraphics[width=.75\columnwidth]{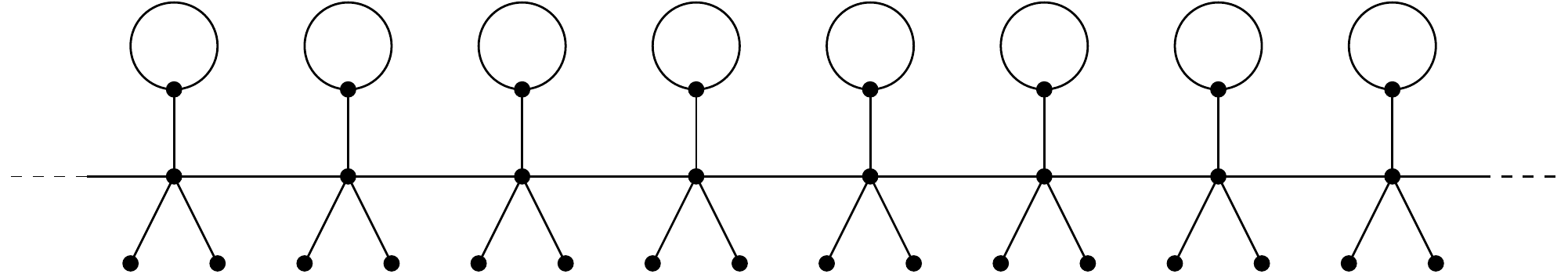}
    \caption{Example of a $\Z$--periodic graph.}
    \label{fig:periodico}
    \end{figure}
    Recently, the discussion of standing waves with prescribed mass for nonlinear Schr\"odinger equations on $\Z^2$--periodic graphs (also known as doubly periodic graphs) has been addressed by \cite{ADST19} in the case of the infinite two--dimensional square \emph{grid} (Figure \ref{fig:Q}). Precisely, \cite{ADST19} provides an extensive analysis of the phenomenon of the \emph{dimensional crossover} which arises both as the coexistence of the one and the two--dimensional Sobolev inequalities and as the presence of a continuum of $L^2$--critical exponents for the ground states of the NLSE.
    Further generalizations to \emph{honeycomb} and $\Z^3$--periodic graphs can be found in \cite{AD18,ADR19}, respectively (see also \cite{P18} for the discussion of a related problem on general periodic graphs, while a class of non--periodic graphs with infinitely many edges has been studied in \cite{DST20}). The rigorous derivation of the NLSE as an effective model for the dynamics of slow modulated wave packets on doubly periodic metric graphs such as the honeycomb grid can be found in \cite{GSU21}.
    
    \medskip
    The aim of our paper is that of investigating how the presence of \emph{defects}, that is a breaking of the symmetric structure of the two--dimensional square grid caused by the removal of finitely or infinitely many edges, may affect both the Sobolev inequalities supported by the resulting graph and the existence of the NLSE ground states. In particular, we aim at classifying which types of defects yield a persistence of the dimensional crossover and which ones, on the contrary, make the dimensional crossover fail. In other words, our goal is to detect under which perturbations the model provided by the undefected grid is robust.
    
    Although the analysis of topological perturbations of periodic structures has a long history in other settings, such as percolation theory (see for instance \cite{MR04,P08} and references therein) where related issues continue to be a hot topic (see the recent paper \cite{HH21}) , to the best of our knowledge this is the first insight on defected grids in the context of metric graphs. Our work will show that the problem displays an extremely wide phenomenology, of which we provide here a detailed though incomplete picture. Several new questions and interesting open problems are raised at the edge of combinatorics, geometry and analysis, whose future investigations we believe may move from the foundations laid in the present paper.  
    
    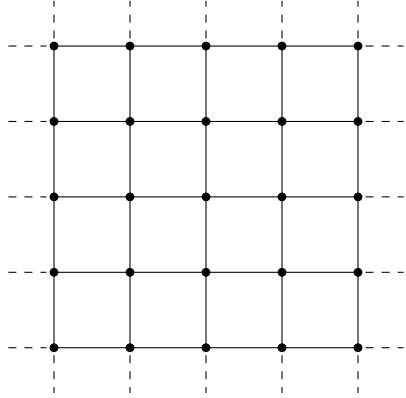
\begin{figure}[t]
		\centering
			\begin{tikzpicture}[xscale= 0.5,yscale=0.5]
			\draw[step=2,thin] (0,0) grid (8,8);
			\foreach \x in {0,2,...,8} \foreach \y in {0,2,...,8} \node at (\x,\y) [nodo] {};
			\foreach \x in {0,2,...,8}
			{\draw[dashed,thin] (\x,8.2)--(\x,9.2) (\x,-0.2)--(\x,-1.2) (-1.2,\x)--(-0.2,\x)  (8.2,\x)--(9.2,\x); }
			\end{tikzpicture}
		\caption{The infinite two--dimensional square grid $\Q$.}
		\label{fig:Q}
	\end{figure}
	


    \subsection*{Organization of the paper}
    
    The paper is organized as follows.

	\begin{itemize} 
	 \item In Section \ref{sec:main} we state the main results of the paper. Precisely:
	 \begin{itemize}
	  \item Section \ref{subsec:grid} recalls briefly the results obtained in \cite{ADST19} on the undefected two--dimensional grid;
	  \item Section \ref{subsec:sett} introduces the main tools and definitions for the setting of the problem;
	  \item Section \ref{subsec:main} states the main results of the paper concerning:
	  \begin{itemize}
	   \item[(i)] the coexistence of the one and the two--dimensional Sobolev inequalities (Theorem \ref{THM:iso=sob}, Theorem \ref{THM:s2d_bound}, Theorem \ref{THM:s2d_P} and Theorem \ref{THM:SobNoP}),
	   \item[(ii)] the NLSE ground states (Theorem \ref{THM:ex_com&Z} and Theorem \ref{thm:noground});
      \end{itemize}
      \item Section \ref{subsec:open} suggests some open problems.
	 \end{itemize}
	 \item In Section \ref{sec:s2d} we present the proofs of the results on Sobolev inequalities. In particular:
	 \begin{itemize}
	  \item Section \ref{subsec:isop} concerns the equivalence with the isoperimetric inequality (Theorem \ref{THM:iso=sob});
	  \item Section \ref{subsec:bound} addresses the case of uniformly bounded defects (Theorem \ref{THM:s2d_bound}); 
	  \item Section \ref{subsec:unbd_s2d} deals with the case of unbounded defects (Theorem \ref{THM:s2d_P} and Theorem \ref{THM:SobNoP}).
     \end{itemize}
	 \item In Section \ref{sec:gs} we give the proofs of the results on the NLSE ground states. In detail:
	 \begin{itemize}
	  \item Section \ref{subsec:proofNLS} investigates the case of finitely many bounded defects and the case of $\Z$--periodic defects (first part of Theorem \ref{THM:ex_com&Z});
	  \item Section \ref{sub:z2} studies the case of $\Z^2$--periodic defects (second part of Theorem \ref{THM:ex_com&Z});
	  \item Section \ref{subsec:nonex} discusses the case of infinitely many defects with no periodic structure (Theorem \ref{thm:noground}).
	 \end{itemize}
	 \item In the Appendix we prove Lemma \ref{lem_dDconn}.
	\end{itemize}
	
    

\subsection*{Acknowledgements}

The authors wish to thank Lorenzo Venturello for helpful discussions and suggestions in the construction of the counterexample in Theorem \ref{THM:SobNoP}. Moreover, the work has been partially supported by the MIUR project ``Dipartimenti di Eccellenza 2018--2022" (CUP E11G18000350001) and by the INdAM GNAMPA project 2020 ``Modelli differenziali alle derivate parziali per fenomeni di interazione". Finally, we wish to thank the anonymous referees for their useful comments concerning the present paper.


	
	\section{Setting and main results}
	\label{sec:main}
	
	In this Section we present all the results we obtained thus far on defected two--dimensional grids, both from the point of view of the Sobolev inequalities and from that of the NLSE ground states.
	
	Since the paper aims at discussing such results in comparison with those on the undefected two--dimensional grid, let us preliminarily revise the main achievements on this model attained in \cite{ADST19}.
    
    

    \subsection{Two--dimensional grid: the dimensional crossover}
    \label{subsec:grid}
    
	To better explain the unusual features of the two--dimensional grid, it is worth recalling some well known results on standard Euclidean spaces.
	
	A characteristic feature of these spaces is that, for every integer $d\geq1$, $\R^d$ supports a unique Sobolev inequality of the form
	\begin{equation}
		\label{intro:SdD}
		\|u\|_{L^{\f d{d-1}}(\R^d)}\lesssim\|\triangledown u\|_{L^1(\R^d)},\qquad   \forall u\in W^{1,1}(\R^d)\,,
	\end{equation}
    which depends on the sole dimension $d$ (see for instance \cite[Chapter 8]{LL01}). Such inequality has various consequences on the analysis of nonlinear PDEs. In particular, if we recall the NonLinear Schr\"odinger energy functional
    \begin{equation}
		\label{eq:ERd}
		E_p(u,\R^d):=\f12\|\nabla u\|_{L^2(\R^d)}^2-\f1p\|u\|_{L^p(\R^d)}^p,\qquad p\geq 2,
	\end{equation}
	then \eqref{intro:SdD} guarantees lower boundedness and coercivity of $E_p(\cdot,\R^d)$ on the space of functions in $H^1(\R^d)$ fulfilling the so--called \emph{mass constraint}
	\[
	\Hmu(\R^d):=\{u\in H^1(\R^d)\,:\,\|u\|_{L^2(\R^d)}^2=\mu\},
	\]
	for every $\mu>0$, provided that $p\leq p_d^*:=\f4d+2$. Indeed, classical arguments show that \eqref{intro:SdD} implies, for every $p\geq2$ if $d=1,2$ and every $p\in\left[2,\f{2d}{d-2}\right)$ if $d\geq3$, the $d$--dimensional Gagliardo--Nirenberg inequality
	\[
		\|u\|_{L^p(\R^d)}^p\lesssim \|u\|_{L^2(\R^d)}^{d+(2-d)\f{p}{2}}\|\triangledown u\|_{L^2(\R^d)}^{\left(\f{p}{2}-1\right)d},\qquad \forall u\in H^1(\R^d),
	\]
	which, combined with \eqref{eq:ERd} and $u\in\Hmu(\R^d)$, gives
	\[
	E_p(u,\R^d)\geq\f12\|\triangledown u\|_{L^2(\R^d)}^2-\f Cp\mu^{\f d2+(2-d)\f p4}\|\triangledown u\|_{L^2(\R^d)}^{\left(\f{p}{2}-1\right)d}.
	\]
	The exponent $p_d^*$ is usually said $L^2$\emph{--critical}, in the sense that, denoting
	\[
	\EE_{p,\R^d}(\mu):=\inf_{v\in\Hmu(\R^d)}E_p(v,\R^d),
	\]
	and defining \emph{ground state of mass} $\mu$ any function $u\in\Hmu(\R^d)$ such that
	\[
	E_p(u,\R^d)=\EE_{p,\R^d}(\mu),
	\]
	the following holds \cite{C03}:
	\begin{itemize}
		\item[(i)] if $p<p_d^*$, then $\EE_{p,\R^d}(\mu)$ is finite and strictly negative and ground states do exist for every $\mu>0$;
		\item[(ii)] if $p>p_d^*$, then $\EE_{p,\R^d}(\mu)=-\infty$ for every $\mu>0$;
		\item[(iii)] if $p=p_d^*$, there exists $\mu_d^*>0$, called \emph{critical mass}, such that:
		\begin{itemize}
		\item when $\mu<\mu_d^*$, $\EE_{p_d^*,\R^d}(\mu)=0$ and the infimum is never attained,
		\item when $\mu>\mu_d^*$, $\EE_{p_d^*,\R^d}(\mu)=-\infty$,
		\item when $\mu=\mu_d^*$, $\EE_{p_d^*,\R^d}(\mu)=0$ and ground states do exist.
		\end{itemize}
	\end{itemize}
	
	It has been shown in \cite{ADST19} that the portrait is sensibly different when one replaces $\R^d$ with the infinite two--dimensional square grid in Figure \ref{fig:Q}, throughout denoted by $\Q$ (for more details see Section \ref{subsubsec:grid} or \cite[Section 2]{ADST19}). A brand new phenomenon appears in this setting, the so--called \emph{dimensional crossover}, deep--seated in the interplay between the microscopic and the macroscopic scales of $\Q$.
	
	The first way to see it is from the point of view of Sobolev inequalities.
	As any metric graph, the two--dimensional grid is a locally one--dimensional metric space. Hence, it is easy to see that it inherits from the dimension one the following inequality
	\begin{equation}
	 \label{eq:1dGN}
	 \|u\|_{L^\infty(\Q)}\lesssim\|u'\|_{L^1(\Q)},\qquad\forall u\in W^{1,1}(\Q),
	\end{equation}
	which is the Sobolev inequality typical of $\R$. However, in this case it is no more a characteristic feature. Indeed, the grid also supports the two--dimensional Sobolev inequality \cite[Theorem 2.2]{ADST19}
	\begin{equation}
	\label{eq:2dGN}
	\|u\|_{L^2(\Q)}\lesssim\|u'\|_{L^1(\Q)},\qquad\forall u\in W^{1,1}(\Q).
	\end{equation}
	This is a very uncommon property, which possesses no analogue in other classes of metric graphs such as graphs with finitely many edges or $\Z$--periodic ones.
	
	For our purposes,  the most relevant consequence of the simultaneous validity of \eqref{eq:1dGN} and \eqref{eq:2dGN} is that one can prove both the one and the two--dimensional Gagliardo--Nirenberg inequalities for every $p\geq2$, i.e.
	\begin{equation}
	\label{eq:GNOld}
	\left\{
	\begin{array}{l}
	 \displaystyle\|u\|_{L^p(\Q)}^p\leq C_1(p,\Q)\|u\|_{L^2(\Q)}^{\f p2+1}\|u'\|_{L^2(\Q)}^{\f p2-1},\\[.4cm]
	 \displaystyle\|u\|_{L^p(\Q)}^{p}\leq C_2(p,\Q)\|u\|_{L^2(\Q)}^{2}\|u'\|_{L^2(\Q)}^{p-2},
	\end{array}
	\right.
	\qquad\forall u\in H^1(\Q),
	\end{equation}
	whose combination \cite[Corollary 2.4]{ADST19} yields, when $4\leq p\leq6$, a new family of Gagliardo--Nirenberg inequalities
	\begin{equation}
	\label{eq:GNNew}
	\|u\|_{L^p(\Q)}^p\leq K(p,\Q)\|u\|_{L^2(\Q)}^{p-2}\|u'\|_{L^2(\Q)}^2,\qquad\forall u\in H^1(\Q).
	\end{equation}
	Note that, in contrast to \eqref{eq:GNOld}, in \eqref{eq:GNNew} the parameter $p$ does not affect the power of $\|u'\|_{L^2(\Q)}$.
	
	This remark leads to the second evidence of the dimensional crossover. Indeed, defining again ground states of mass $\mu$ the functions $u\in\Hmu(\Q)$ such that
	\[
	E_p(u,\Q)=\EE_{p,\Q}(\mu):=\inf_{v\in\Hmu(\Q)}E_p(v,\Q),
	\]
	with
	\[
	 E_p(u,\Q):=\f12\|u'\|_{L^2(\Q)}^2-\f1p\|u\|_{L^p(\Q)}^p\quad\text{and}\quad\Hmu(\Q):=\{u\in H^1(\Q)\,:\,\|u\|_{L^2(\Q)}^2=\mu\},
	\]
	the validity of \eqref{eq:GNNew} allows to prove that \cite[Theorems 1.1--1.2]{ADST19}:
	\begin{itemize}
		\item[(i)] if $p\in(2,4)$, then $\EE_{p,\Q}(\mu)$ is finite and strictly negative and ground states do exist for every $\mu>0$;
		\item[(ii)] if $p\in[4,6]$, then there exists a critical mass $\mu_{p,\Q}>0$ such that
		\begin{equation}
		\label{intro:EQ_4p6}
		\EE_{p,\Q}(\mu)
		\left\{
		\begin{array}{ll}
		=0, & \text{if }\mu\leq\mu_{p,\Q},\\[.2cm]
		<0, & \text{if }\mu>\mu_{p,\Q},
		\end{array}
		\right.
		\end{equation}
		and
		\begin{itemize}
			\item when $p\in(4,6)$, ground states exist if and only if $\mu\geq\mu_{p,\Q}$,
			\item when $p=4$, ground states exist if $\mu>\mu_{4,\Q}$, whilst they do not exist if $\mu<\mu_{4,\Q}$,
			\item when $p=6$, $\EE_{p,\Q}(\mu)=-\infty$ for every $\mu>\mu_{6,\Q}$ and ground states do not exist for any value of $\mu>0$.
		\end{itemize}
	\end{itemize}
	The interplay between different dimensions thus appears as $p_1^*=6$ and $p_2^*=4$ are the critical powers in dimensions $d=1$ and $d=2$, respectively. When $p$ is subcritical both in dimensions one and two, i.e. $p\in(2,4)$, the problem exhibits the typical subcritical features and ground states always exist. On the contrary, when $p$ is supercritical in dimension two, i.e. $p>4$, but subcritical in dimension one, i.e. $p<6$, a critical mass (depending on $p$) arises. The two--dimensional macroscale of $\Q$ dominates for small masses, where ground states do not exist as all functions have positive energy, while for larger masses the one--dimensional microscale prevails and ground states do exist. Finally, when $p$ is supercritical for both the scales, i.e. $p>6$, the problem retraces the classical supercritical behavior as the energy is always unbounded from below.
	
    

\subsection{Setting and definitions}
    \label{subsec:sett}

    Before stating the main results of the paper it is necessary to clarify the setting of the problem. In particular, we must specify what we mean by \emph{defects} and \emph{defected grid}.
    


\subsubsection{Two--dimensional metric grids}
\label{subsubsec:grid}
    
    Let us first recall some basics on two--dimensional grids. For a more general overview on compact and noncompact metric graphs we redirect the reader to \cite{ADST19,AST17,BK13}.

	It is convenient to introduce the \emph{infinite two--dimensional square grid} $\Q=(\VQ,\EQ)$ with edges of length 1 as the graph isometrically embedded in $\R^2$ whose set of vertices $\VQ$ coincides with the integer lattice $\Z^2$ and with an edge between each couple of vertices at unitary distance in $\R^2$. Thus
	\[
	 \begin{array}{rcl}
	 \v\in\VQ & \Rightarrow &\v\cong(i,j)\in\R^2,\text{ for some }i,\,j\in\Z,
    \end{array}
	\]
    and $\EQ=\EQ^h\cup\EQ^v$, with
    \begin{align*}
     \EQ^h\cong & \big\{(i,i+1)\times\{j\}\subset\R^2: i,\,j\in\Z\big\}\subset\R^2\\[.2cm]
     \EQ^v\cong & \big\{\{i\}\times(j,j+1)\subset\R^2: i,\,j\in\Z\big\}\subset\R^2\,.
    \end{align*}
    In view of such embedding, one can write $\Q=\VQ\cup\EQ$, or also $\Q=\overline{\EQ}^{\R^2}$, with $\VQ$ and $\EQ$ meant as subsets of $\R^2$. Similarly, $\Q=\left(\bigcup_{j\in\Z}H_j\right)\cup\left(\bigcup_{i\in\Z}V_i\right)$, where $H_j$ is the horizontal line in $\R^2$ of ordinate $j\in\Z$ and $V_i$ is the vertical line in $\R^2$ of abscissa $i\in\Z$. 
    
    Clearly, one can think of the grid also independently of its embedding in $\R^2$, as the $\Z^2$--periodic planar graph with every vertex of degree 4 and every edge of length 1. Throughout the paper, we will use these different identifications of $\Q$ as interchangeable since they do not give rise to misunderstandings.
	
	Consistently, a function $u:\Q\to\R$ can be defined both as a family of pairs of functions $\{(h_k,v_k)\}_{k\in\Z}$ such that $h_k,v_k:\R\to\R$, where $h_k:=u_{|_{H_k}}$ and $v_k:=u_{|_{V_k}}$, and as a family of functions $(u_e)_{e\in\EQ}$ such that $u_e:[0,1]\to\R$, where either $u_e:=u_{|_{[i_e,i_e+1]\times\{j_e\}}}$ when $e\in\EQ^h$, or $u_e:=u_{|_{\{i_e\}\times[j_e,j_e+1]}}$ when $e\in\EQ^v$.
	
	As a consequence, $L^p(\Q)$ is the space of those functions $u$ such that
	\begin{align*}
	 \|u\|_{L^p(\Q)}^p:= & \sum_{j\in\Z}\|h_j\|_{L^p(\R)}^p+\sum_{i\in\Z}\|v_i\|_{L^p(\R)}^p=  \sum_{e\in\EQ}\|u_e\|_{L^p(0,1)}^p<\infty,\qquad\text{if }1\leq p<\infty,\\[.2cm]
	 \|u\|_{L^\infty(\Q)}:= & \sup_{i,j\in\Z}\big\{\|h_j\|_{L^\infty(\R)},\|v_i\|_{L^\infty(\R)}\big\}=\sup_{e\in\EQ}\|u_e\|_{L^\infty(0,1)}<\infty,
	\end{align*}
	while $H^1(\Q)$ and $W^{1,1}(\Q)$ are the spaces of functions $u$ which are continuous on $\Q$ and such that
	\begin{align*}
	 & \|u\|_{H^1(\Q)}^2:=\|u\|_{L^2(\Q)}^2+\|u'\|_{L^2(\Q)}^2<\infty\\[.2cm]
	 & \|u\|_{W^{1,1}(\Q)}:=\|u\|_{L^1(\Q)}+\|u'\|_{L^1(\Q)}<\infty.
	\end{align*}
	For the sake of simplicity, in the following we will always use the notation
	\[
	 \begin{array}{rcl}
	 \|u\|_{L^p(H_j)},\,\|u\|_{L^p(V_i)} & \text{in place of} & \|h_j\|_{L^p(\R)},\,\|v_i\|_{L^p(\R)},\\[.2cm]
	 \|u\|_{L^p(e)},\,\|u\|_{L^\infty(e)} & \text{in place of} & \|u_e\|_{L^p(0,1)},\,\|u_e\|_{L^\infty(0,1)},\\[.2cm]
	 \int_{H_j}|u(x)|^p\dx,\,\int_{V_i}|u(x)|^p\dx & \text{in place of} & \int_\R|h_j(x)|^p\dx,\,\int_\R|v_i(x)|^p\dx,\\[.2cm]
	 \int_{e}|u(x)|^p\dx & \text{in place of} & \int_0^1|u_e(x)|^p\dx.
	 \end{array}
	\]
	
	\medskip
	We also recall that, for any couple of points $q,\,r\in\Q$, we say that $\gamma\subset\Q$ is a \emph{path connecting $q$ and $r$ in $\Q$} if $\gamma=f([a,b])$, with $f:[a,b]\subseteq\R\to\Q$ such that
	\begin{itemize}
	 \item[(i)] $f$ is continuous;
	 \item[(ii)] $f(a)=q$ and $f(b)=r$;
	 \item[(iii)] for every $x\in\Q\setminus\big(\VQ\cup\{q,r\}\big)$, $\#\{f^{-1}(x)\}=1$
	\end{itemize}
	(see, e.g., Figure \ref{fig:path}). Here and throughout the symbol $\#$ denotes the cardinality of a set. Moreover, we call a path from a point to itself a \emph{cycle} and we say that a path is \emph{simple} if (iii) holds for every $x\in\Q\setminus\{q,r\}$.
	
	\begin{remark}
	 Observe that, by definition, a path possesses an orientation too. Roughly speaking, we could say that $x$ precedes $y$ in $\gamma$ whenever $f^{-1}(x)<f^{-1}(y)$. Clearly, this definition does not apply in general. For instance, if $x$ or $y$ are either vertices with two or more pre--images through $f$ or endpoints of a cycle, then the definition is meaningless. Nevertheless, it is meaningful up to a discrete set and then endows the path with a global orientation that reflects also on vertices with two or more pre--images in the fact that one can detect the order of the passages on the same vertex. In view of these remarks, we will always use expressions like ``the last vertex of the path such that...'', ``the first vertex of the path such that...''. They have to be understood as ``the vertex with the greatest pre--image of the path such that...'', ``the vertex with the smallest pre--image of the path such that...''
	\end{remark}
	\begin{figure}[t]
		\centering
		\includegraphics[width=.6\columnwidth]{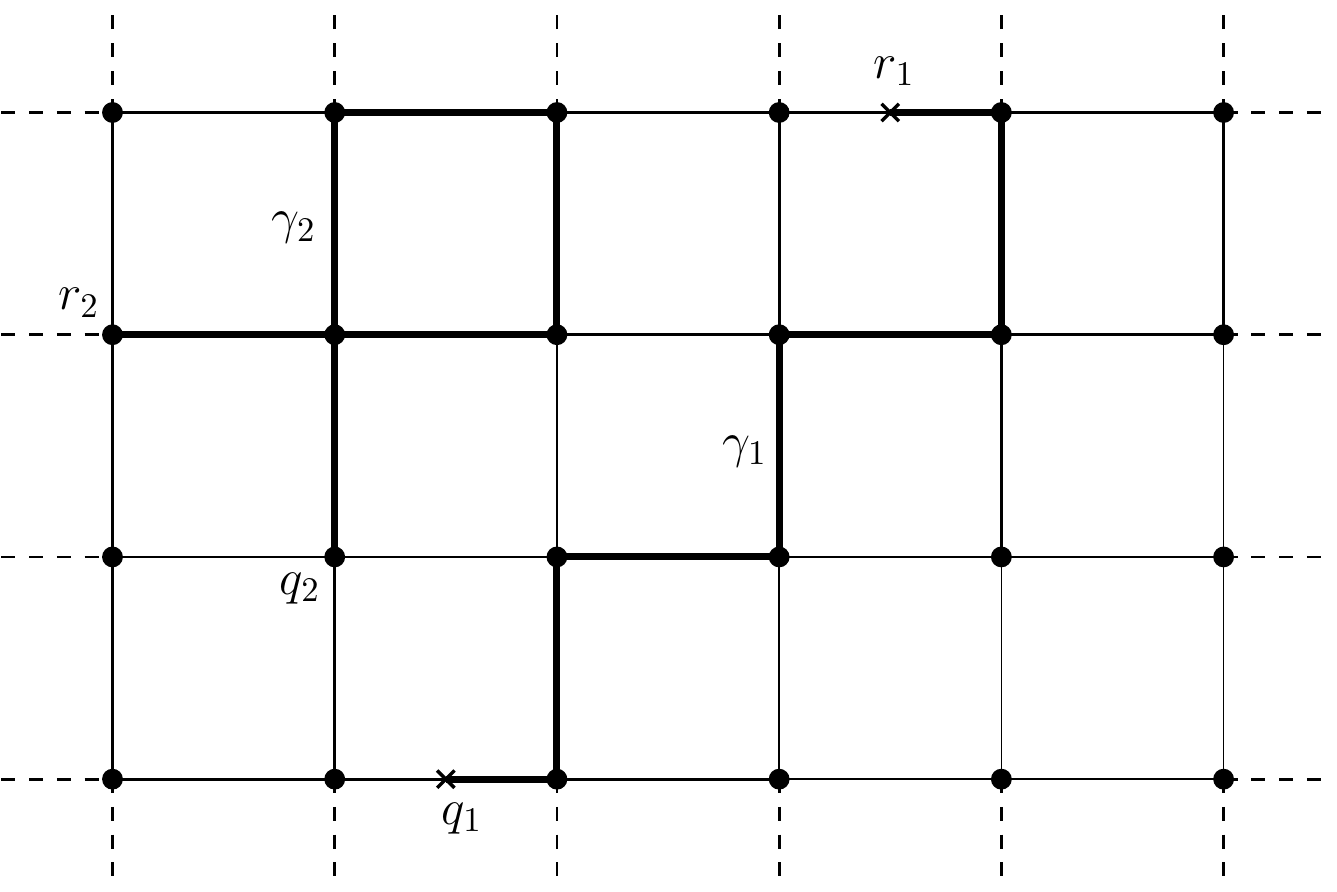}
		 \caption{Two different paths (in bold) on a grid.}
		\label{fig:path}
	\end{figure}
	Given the embedding of $\Q$ in $\R^2$, the \emph{length} of a path $\gamma$ is defined by its 1--dimensional Hausdorff measure and is denoted by $|\gamma|$. Therefore, the \emph{distance} in $\Q$ between two points $q$ and $r$ is given by
    \[
     d_\Q(q,r):=\min\{|\gamma|:\gamma \text{ is a path in }\Q\text{ connecting } q \text{ and } r\},
    \]
    whence the \emph{ball} of center $q$ and radius $R>0$ in $\Q$ is given by
    \[
     B_R(q,\Q):=\{r\in\Q:d_\Q(q,r)<R\}.
    \]
	We also say that a path $\gamma$ is \emph{infinite} whenever $|\gamma|=+\infty$.
	
	\begin{remark}
	 Throughout the paper we use the notation $|\Omega|$ to denote the  1--dimensional Hausdorff measure of any subset $\Omega$ of $\Q$.
	\end{remark}
	
	One thing we often do in the following is to look at the restriction of a function $u$ on the grid to a given path $\gamma$. Since by definition $\gamma$ can be identified with an interval $[a,b]$ on the real line, whenever we need to evaluate $u$ along $\gamma$ we will say for instance ``consider $u(x)$ for $x\in[a,b]$..." to indicate ``consider $u(f(x))$ for $x\in[a,b]$...", omitting to write explicitly the map $f$ for which $\gamma=f([a,b])$.

	Clearly, all the tools introduced in this section can be adapted to the case of a proper subgraph  $\Omega$ of $\Q$. In particular, one can define $\Omega$ simply by its set of edges, i.e. $\Omega=\overline{\E_\Omega}^{\R^2}$.



\subsubsection{Defected grids}
\label{subsubsec:defected}
We introduce the notion of defected grid (Figure \ref{fig:defected}).
	
	\begin{definition}[Defected grid]
		\label{def:G}
		A connected metric graph $\G=(\VG,\EG)$ is said a \emph{defected grid} if:
		\begin{itemize}
			\item[(i)] $\G\subset\Q$, i.e. $\VG\subset \VQ$ and $\EG\subset \EQ$;
			\item[(ii)] $\v\in \VQ\setminus \VG$ if and only if $e\in \EQ\setminus \EG$ \emph{for every} $e\succ\v$;
			\item[(iii)] for every $\v\in \VG$ it holds
			\[
			\liminf_{n\to+\infty}\f{|B_n(\v,\G)|}{|B_n(\v,\Q)|}>0.
			\]
		\end{itemize}
	\end{definition}
	
	\begin{figure}[t]
	\centering
       \subfloat[][three compact defects (with bold boundaries).]
       {\includegraphics[width=.48\columnwidth]{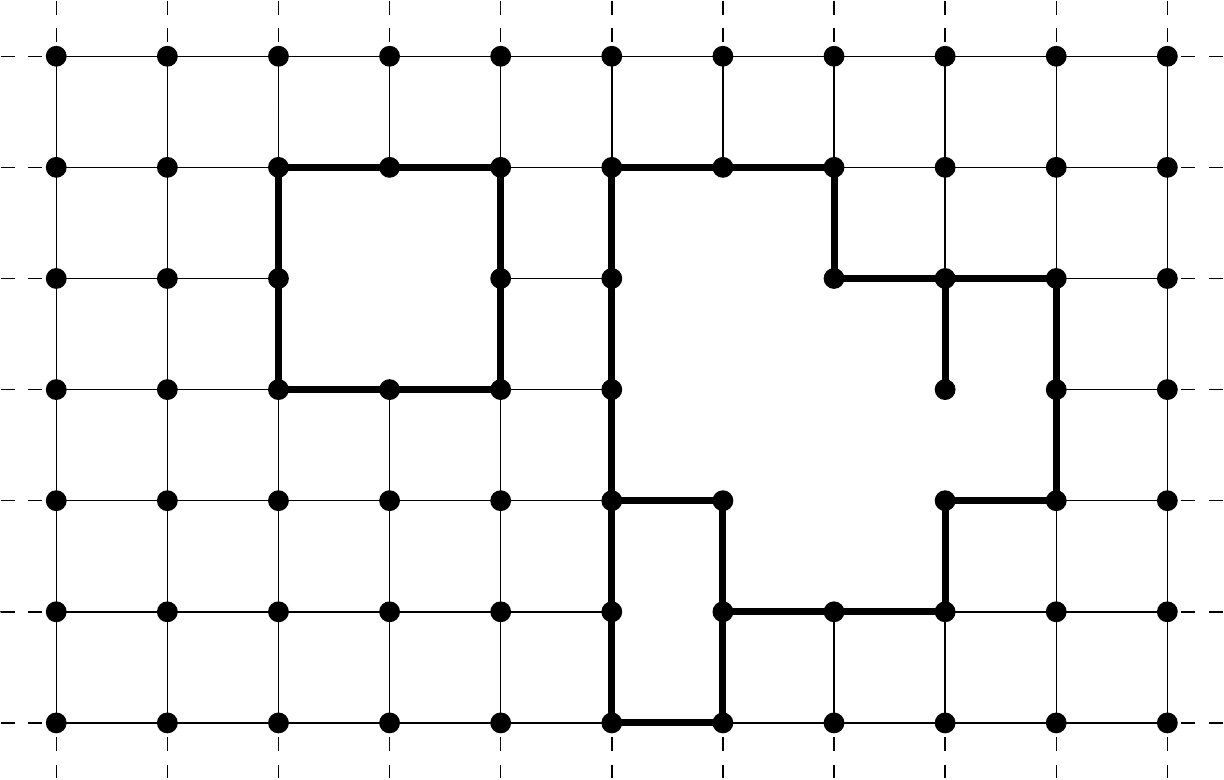}}
       \hspace{.4cm}
       \subfloat[][two noncompact defects (with bold boundaries).]
       {\includegraphics[width=.48\columnwidth]{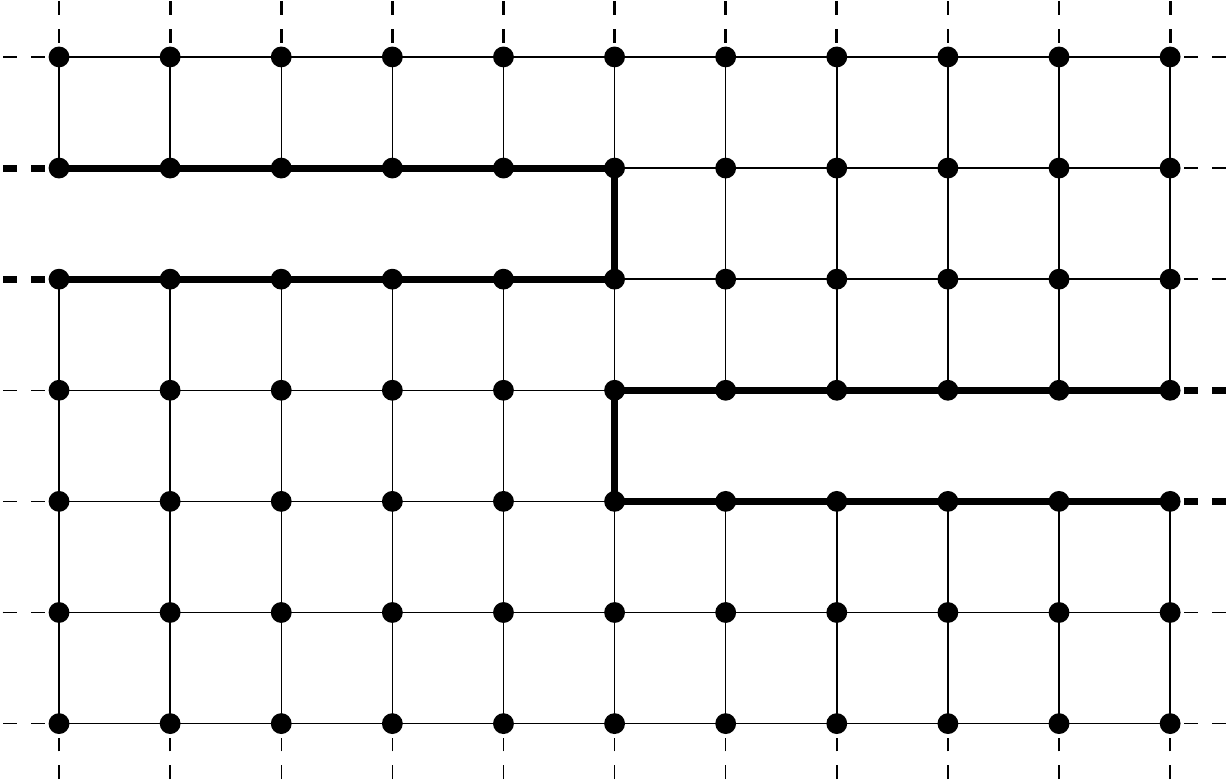}}
       \caption{Examples of defected grids.}
       \label{fig:defected}
    \end{figure}
	
	Some comments are in order. We first highlight that assuming $\G$ to be connected excludes from our discussion graphs as the one in Figure \ref{fig:no_defect}\textsc{(a)}, where most of the questions we are interested in become trivial. 
	
	Assumption (i) means that $\G$ is a subgraph of $\Q$, while assumption (ii) provides that the vertices at the endpoints of any edge of $\EG$ belong to $\VG$. In view of the previous section and of the embedding of $\Q$ in $\R^2$, one easily sees that condition (ii) could be dropped as redundant. However, we leave it for the sake of completeness, since the definition of defected grids does not actually require the embedding in $\R^2$.
	
	Assumption (iii) essentially says that, up to a factor, the area of metric balls with large radius in $\G$ grows as in the undefected grid $\Q$. Even if at first sight it may seem unnatural, it gets rid of several classes of graphs that would be meaningless to think of as perturbations of $\Q$. Indeed, such an assumption is not fulfilled by graphs with a finite number of vertices and edges (both compact ones as in Figure \ref{fig:no_defect}\textsc{(b)} and noncompact ones as in Figure \ref{fig:no_defect}\textsc{(c)}), as well as by $\Z$--periodic graphs (e.g., Figure \ref{fig:no_defect}\textsc{(d)}--\textsc{(e)}) and other graphs with infinitely many edges (e.g., Figure \ref{fig:no_defect}\textsc{(f)}). As pointed out in the Introduction, graphs like these have been already addressed over the last years with specific tools and techniques, and it is known that they do not share any of the two--dimensional features of the unperturbed grid discussed in Section \ref{subsec:grid}.
	
	However, we stress that it is an open question whether any graph that violates Definition \ref{def:G}(iii) does not exhibit any of the features in Section \ref{subsec:grid} (see Section \ref{subsec:open} below). 
	
	\begin{figure}[t]
	\centering
       \subfloat[][Unconnected graph.]
       {\includegraphics[width=.45\columnwidth]{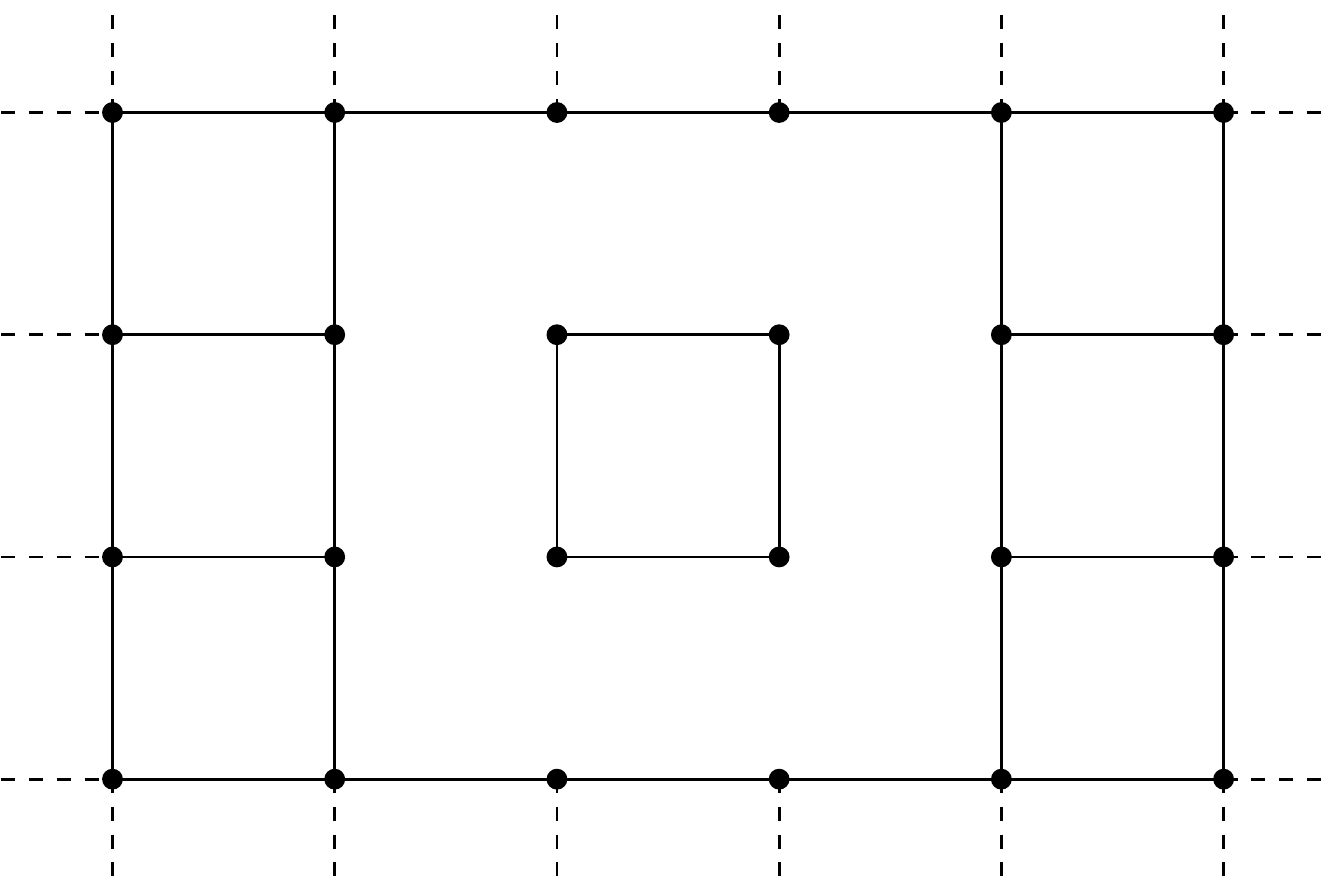}}
       \hspace{2cm}
       \subfloat[][Compact graph.]
       {\includegraphics[width=.28\columnwidth]{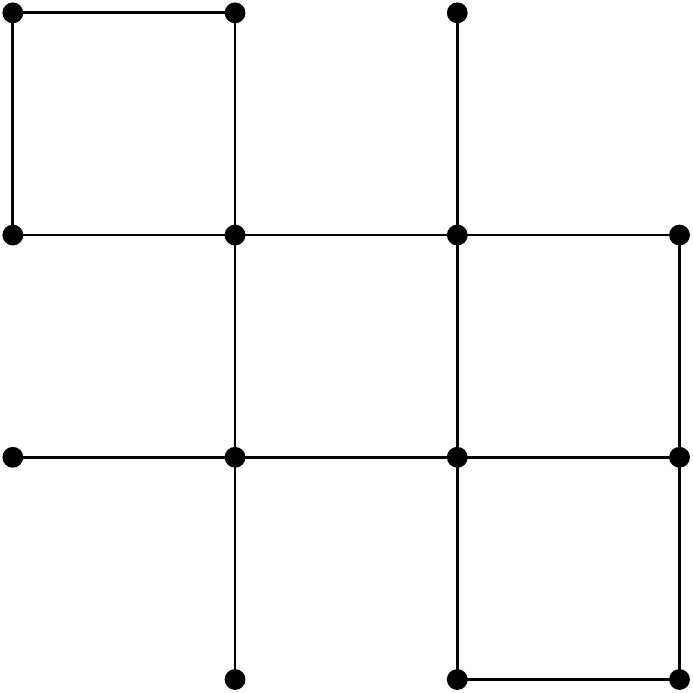}}
       \\
       \subfloat[][Noncompact graph with half--lines (denoted with $\infty$ at the end).]
       {\includegraphics[width=.5\columnwidth]{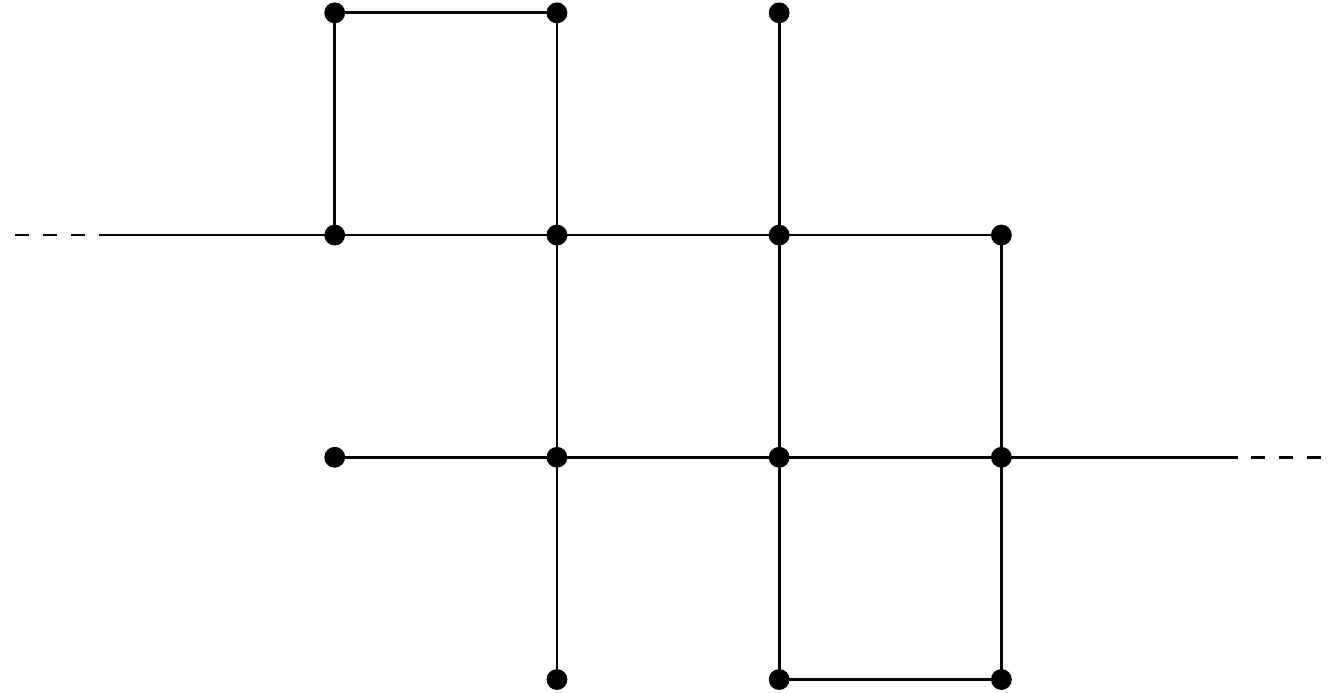}}
       \hspace{.6cm}
       \subfloat[][Ladder--type graph.]
       {\includegraphics[width=.45\columnwidth]{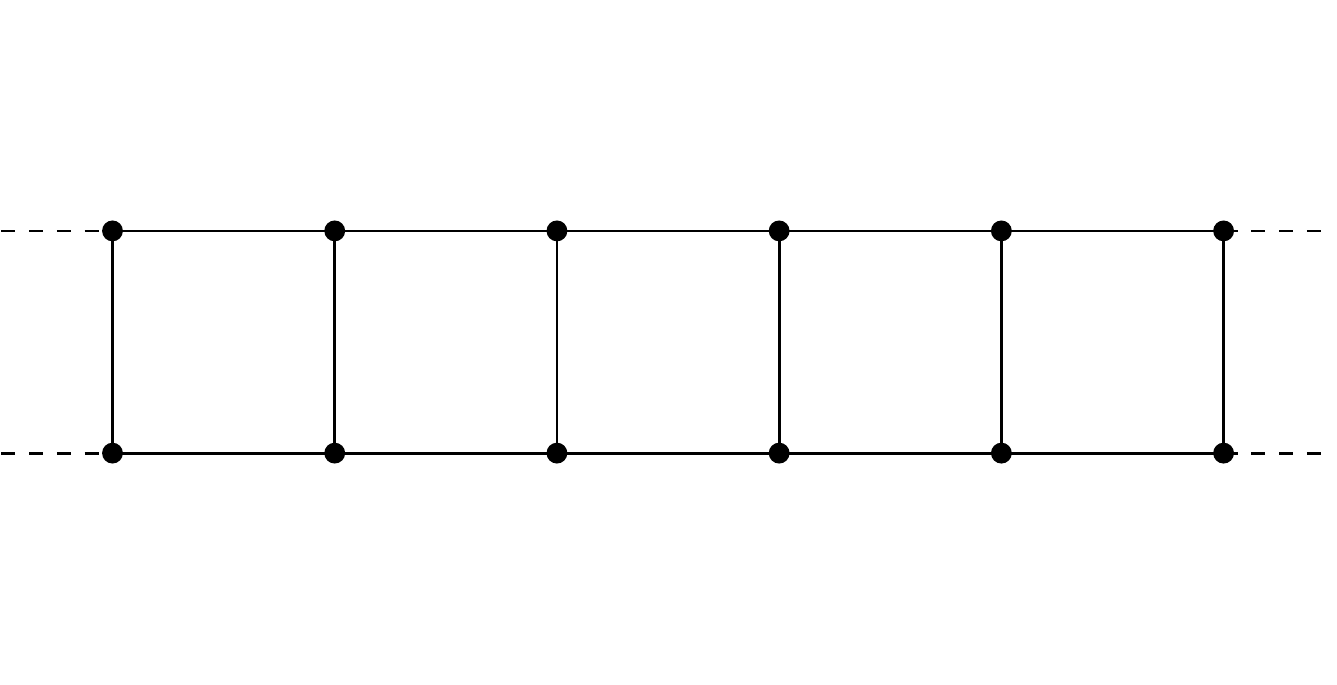}}
       \\
       \subfloat[][Multi--ladder--type graph.]
       {\includegraphics[width=.45\columnwidth]{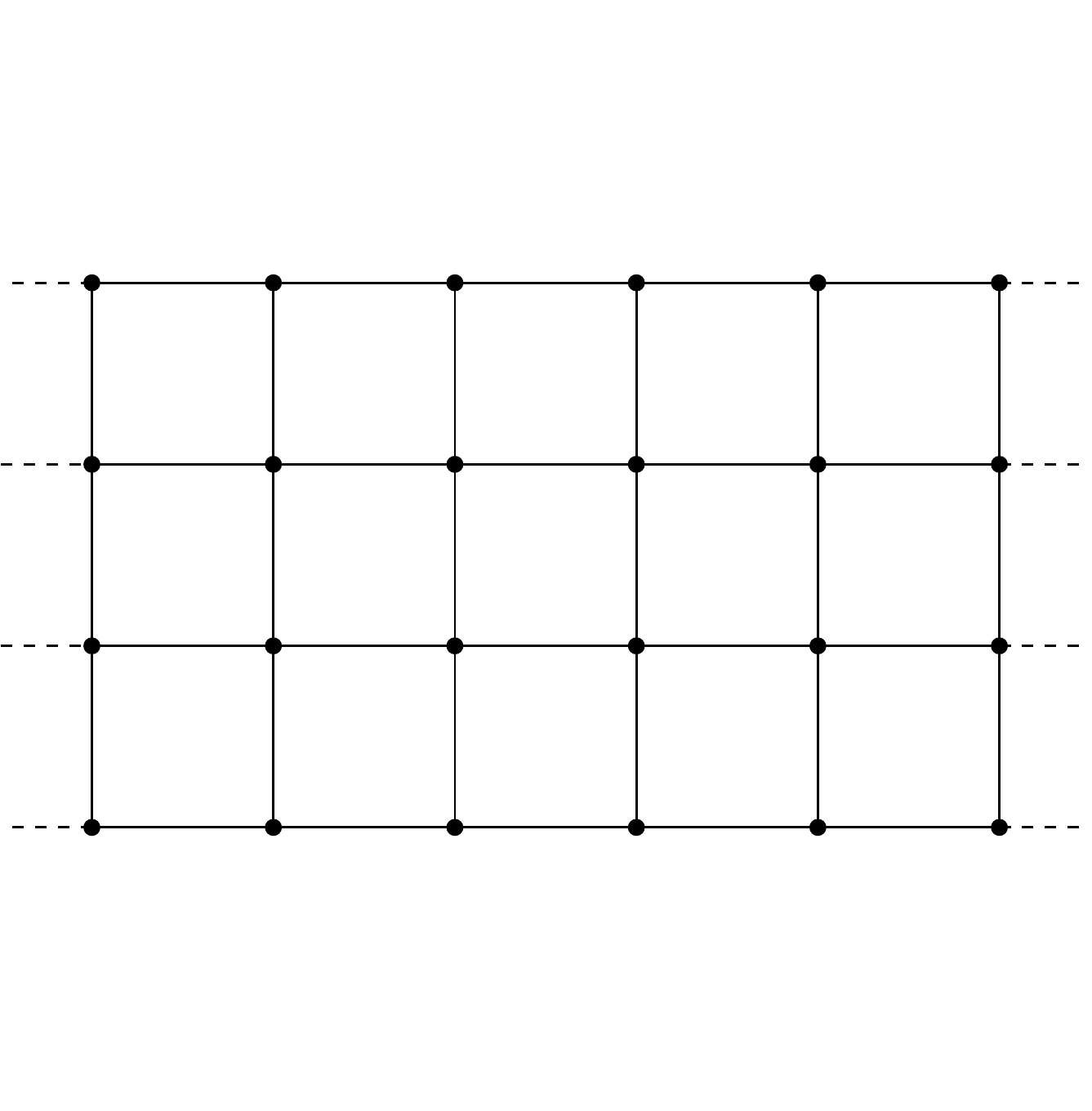}}
       \hspace{.6cm}
       \subfloat[][Cross--ladder--type graph.]
       {\includegraphics[width=.45\columnwidth]{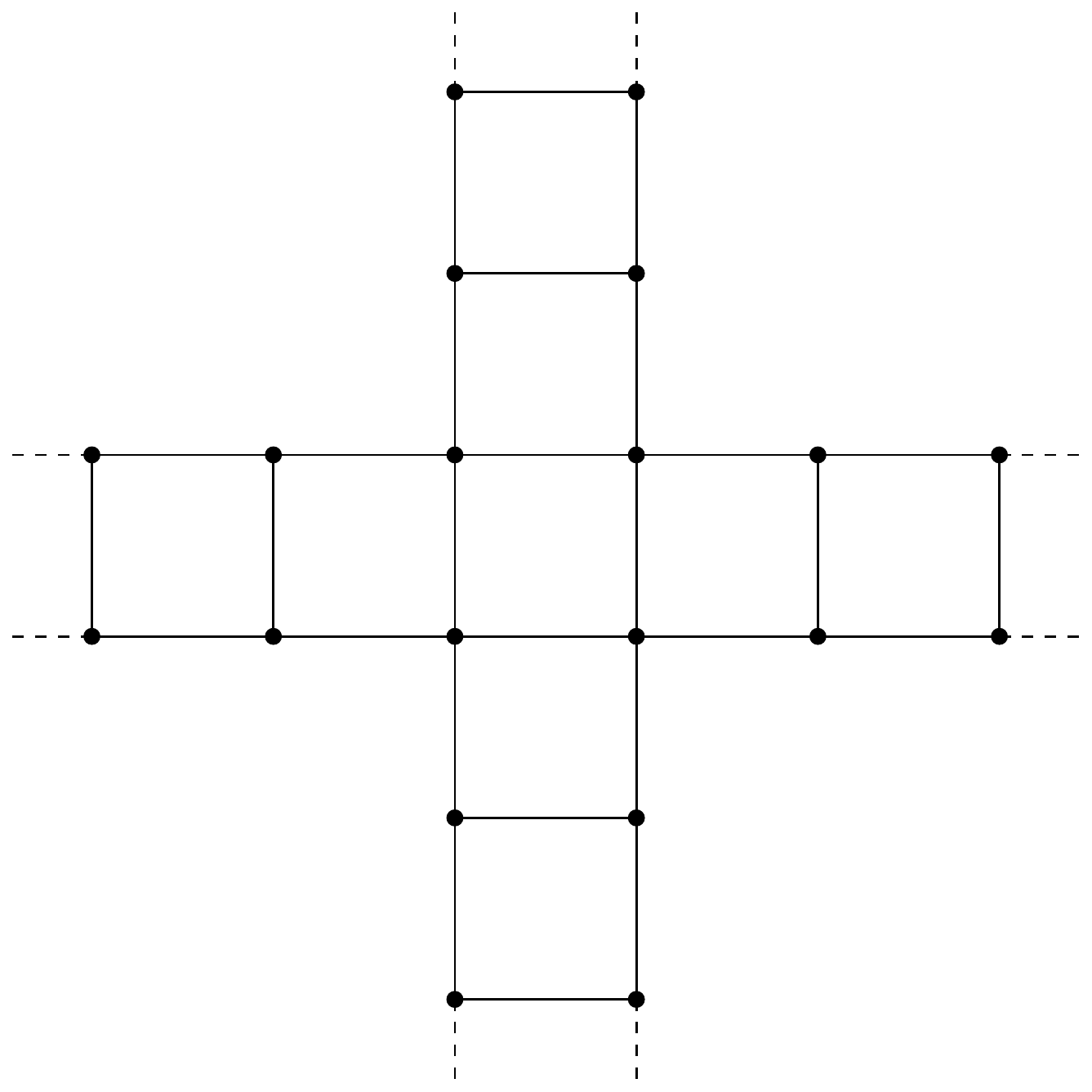}}
       \caption{Examples of graphs that do not satisfy Definition \ref{def:G}.}
       \label{fig:no_defect}
    \end{figure}
	
	\medskip
	We now need to specify what \emph{defects} and \emph{boundaries} of defects are. Preliminarily, we introduce the following notation.
	
	\begin{definition}[Cell]
	 We call a \emph{cell}, and denote it by $\Sq$, any cycle in $\Q$ of length 4. Moreover, given an edge $e\in \EQ$, we denote by $\Sq_e'$, $\Sq_e''\subset\Q$ the unique two cells that contain the edge $e$ (Figure \ref{fig:cells}).
	\end{definition}
	
	\begin{figure}[t]
	\centering
       \subfloat[][Horizontal edge.]
       {\includegraphics[width=.4\columnwidth]{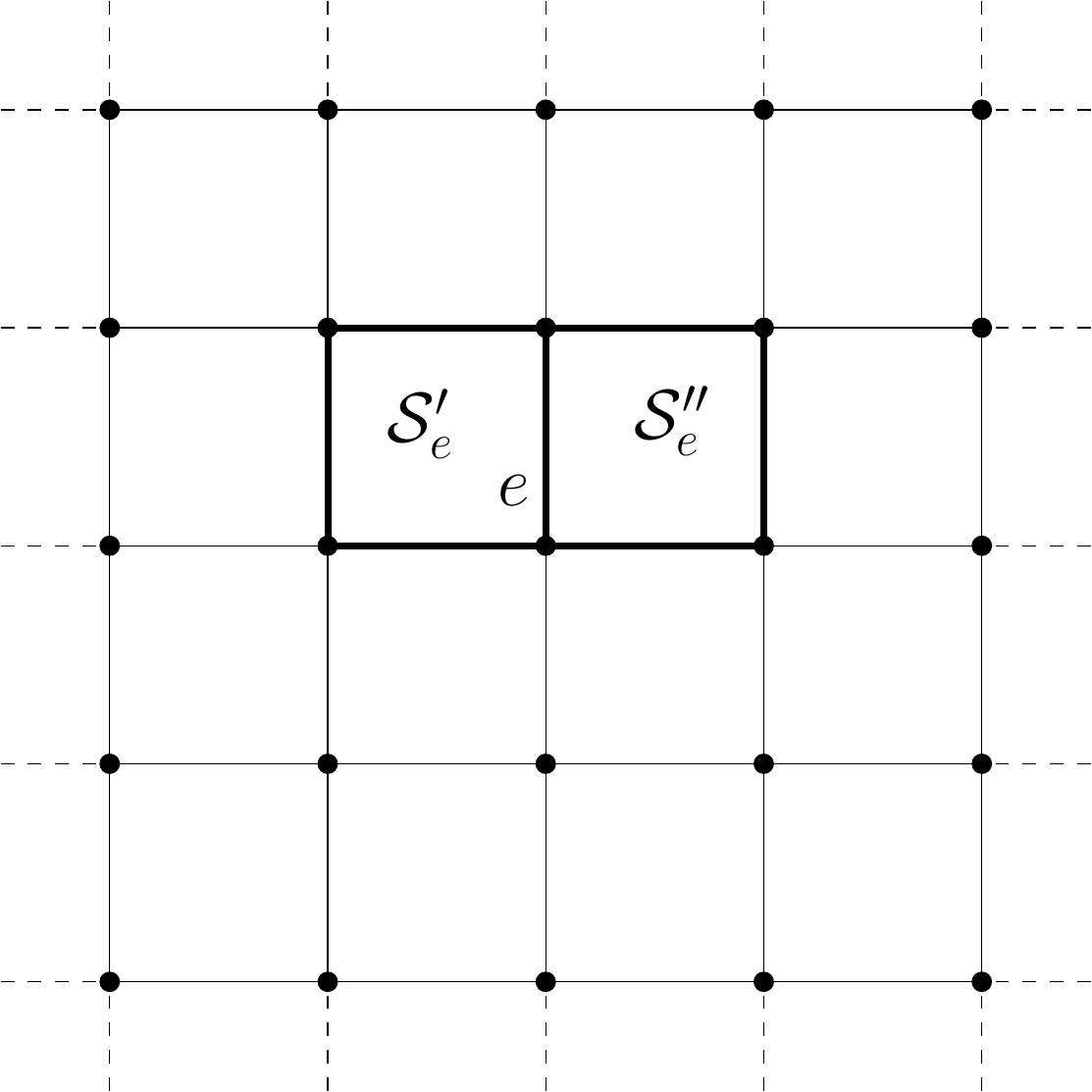}}
       \hspace{1cm}
       \subfloat[][Vertical edge.]
       {\includegraphics[width=.4\columnwidth]{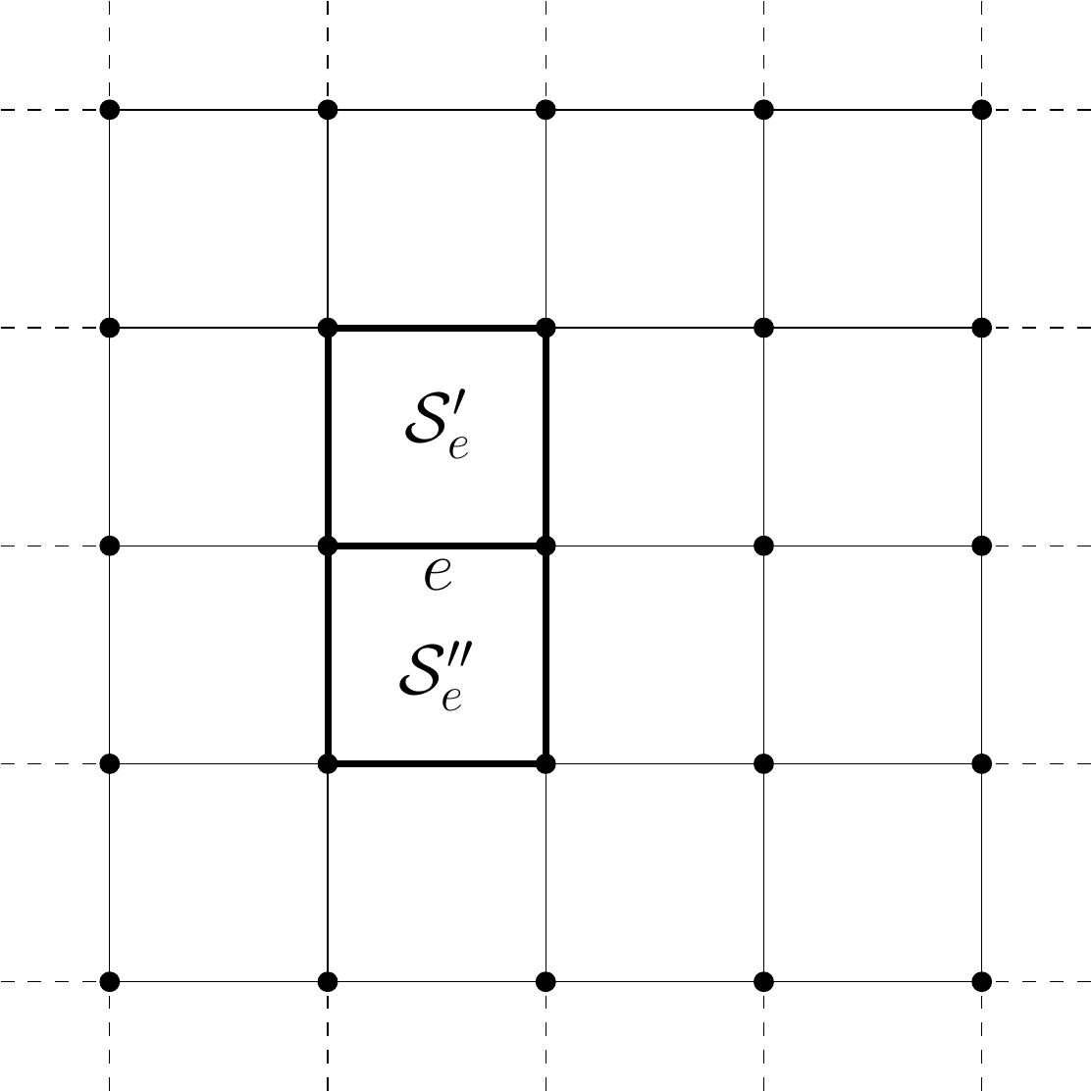}}
       \caption{Examples of cells $\Sq_e'$ and $\Sq_e''$.}
       \label{fig:cells}
    \end{figure}
	
	\begin{definition}(Defect)
		\label{def:D and dD}
		Let $\G$ be a defected grid. We say that a non empty set of edges $\D\subset \EQ\setminus\EG$ is a \emph{defect} of $\G$ if either
		\begin{enumerate}
		 \item[(a)] $\D=\{e\}$ and $(\Sq_e'\cup\Sq_e'')\cap(\EQ\setminus\EG)=\{e\}$; or
		 \item[(b)] for any pair of edges $e,f\in\D$, there exists a finite sequence of cells $\Sq_0,\dots,\Sq_N\subset\Q$ such that
		\begin{itemize}
			\item[\emph{i.}] $e\in\Sq_0$ and $f\in\Sq_N$,
			\item[\emph{ii.}] for every $j=0,\,\dots,\,N-1$ there exists an edge $e_j\in\D$ such that $\Sq_j\cap\Sq_{j+1}=\{e_j\}$.
		\end{itemize}
		\end{enumerate}
		We call \emph{boundary} of $\D$ the graph $\partial\D\subset\G$ given by 
		\[
		\partial\D:=\left(\bigcup_{e\in\D}\Sq_e'\cup\Sq_e''\right)\cap\G.
		\]
	\end{definition}
	
	\begin{figure}[th]
	\centering
       \subfloat[][Two defects consisting of a single edge each.]
       {\includegraphics[width=.47\columnwidth]{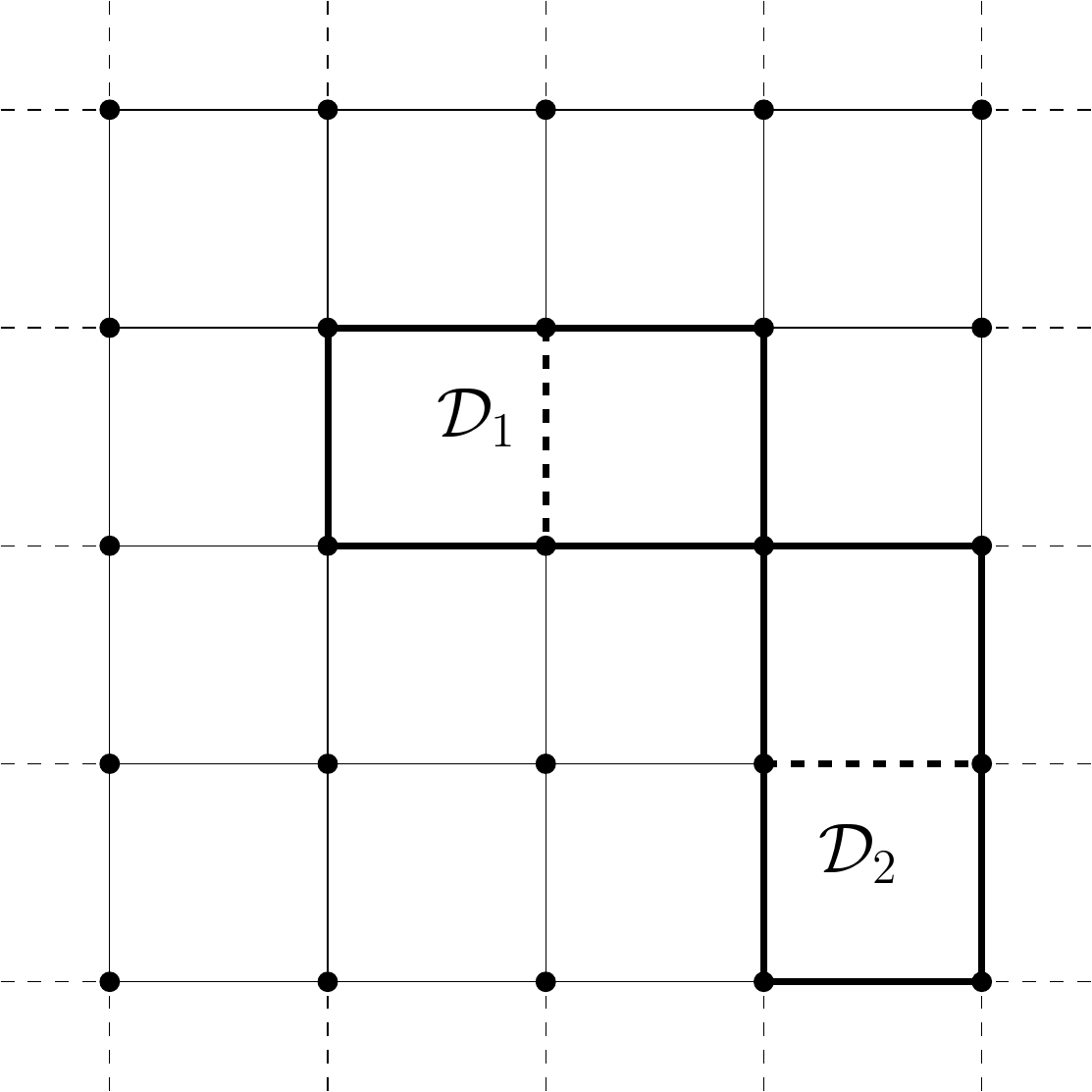}}
       \hspace{.3cm}
       \subfloat[][One defect consisting of four edges.]
       {\includegraphics[width=.47\columnwidth]{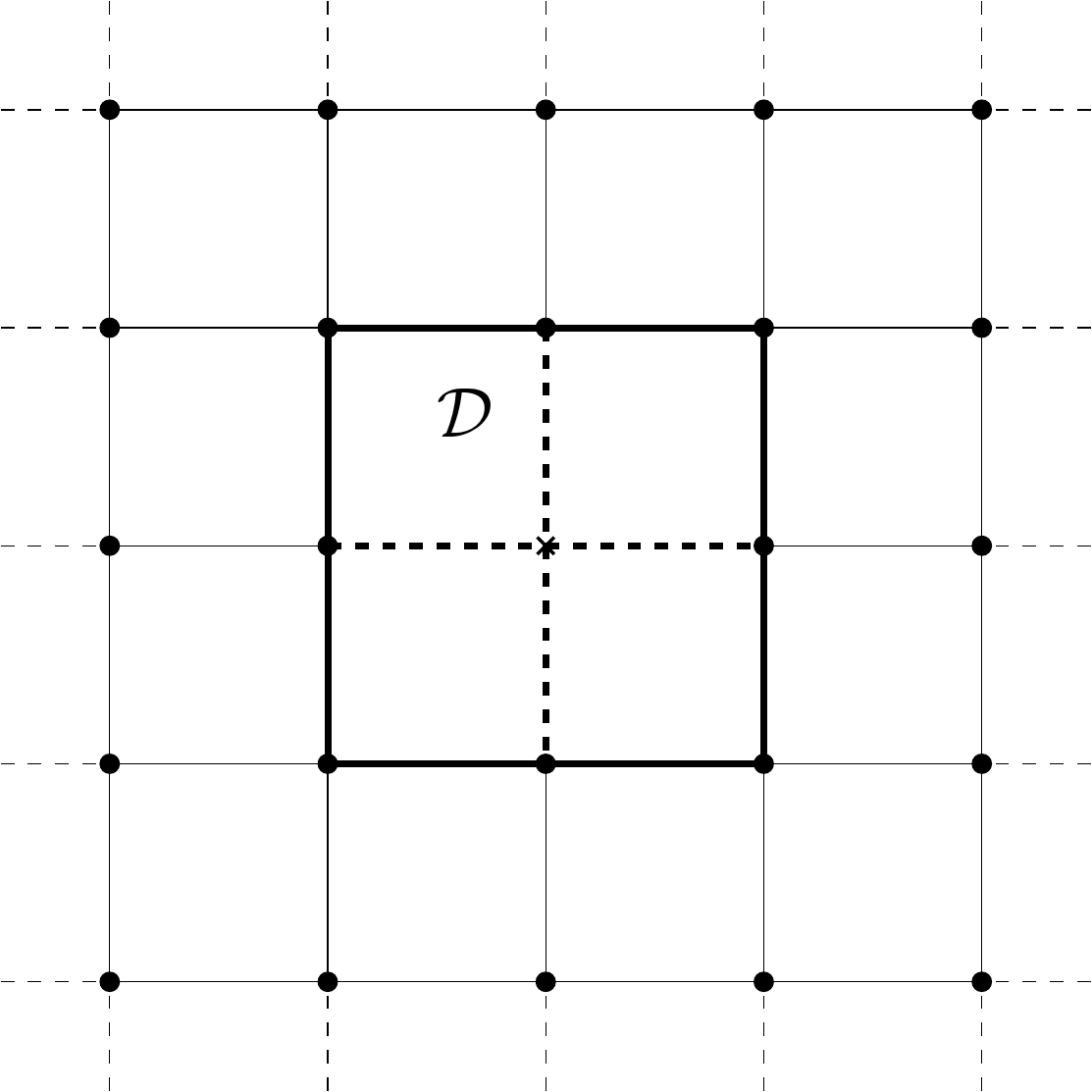}}
       \\
       \subfloat[][Two adjacent defects consisting of four edges each.]
       {\includegraphics[width=.67\columnwidth]{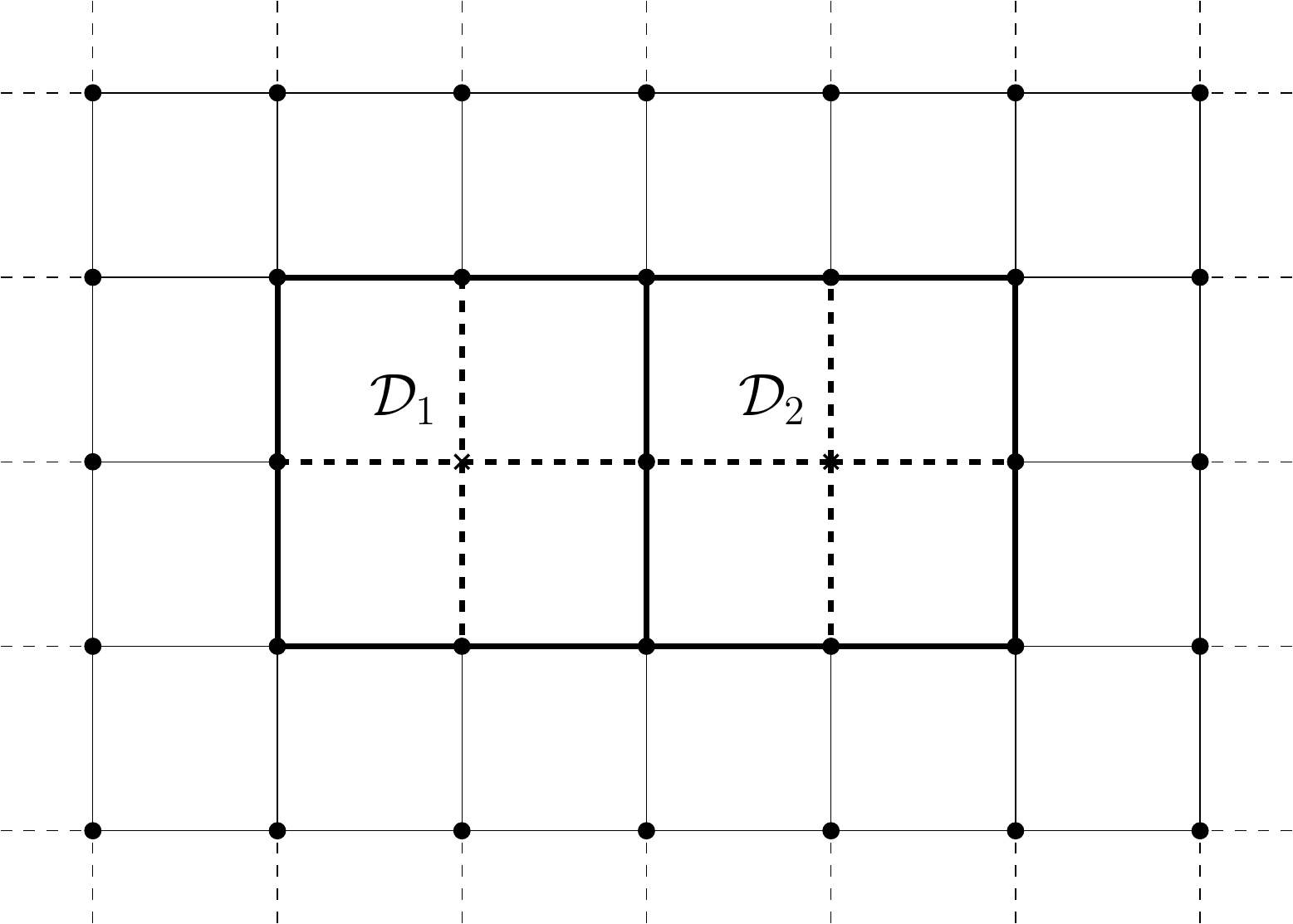}}
       \caption{Examples of defects: boundaries of defects are bold, ``removed'' edges are bold and dashed (crosses represent ``removed'' vertices).}
       \label{fig:defects}
    \end{figure}
    
    Again, some comments are in order. First, Definition \ref{def:D and dD}(a) refers to defects consisting of a single edge (e.g., Figure \ref{fig:defects}\textsc{(a)}), while Definition \ref{def:D and dD}(b) refers to defects with multiple edges (e.g., Figures \ref{fig:defects}\textsc{(b)}--\textsc{(c)}). Moreover, condition \emph{ii.} in Definition \ref{def:D and dD}(b) allows to distinguish between defects with adjacent boundaries as those depicted in Figure \ref{fig:defects}\textsc{(c)}. Indeed, if in the picture one considers any sequence of cells joining an edge of $\D_1$ with an edge of $\D_2$ such that each couple of consecutive cells shares one edge, then there exists at least a couple of consecutive cells whose intersection does not belong to $\D_1\cup\D_2$, thus implying that $\D_1\cup\D_2$ cannot be considered as a single defect. 
    
    Note also that boundaries of the defects belong to the defected grid, while defects do not.
    
	The following lemma states a rather intuitive topological property of the boundaries of the defects, whose proof is a bit technical and has been postponed to the Appendix.
	
	\begin{lemma}
		\label{lem_dDconn}
		Let $\G$ be a defected grid. If $\D$ is a defect of $\G$, then $\partial\D$ is connected.
	\end{lemma}
	
	\medskip
	Finally, we need to recall the notion of \emph{area} and \emph{perimeter} of a subset of $\G$ (see, e.g., Figure \ref{fig-area}). The corresponding definition in $\Q$ are analogous.
	
	\begin{definition}(Area and perimeter)
	 \label{def-AP}
	 Let $\G$ be a defected grid and $\Omega\subset\G$. Then:
	 \begin{itemize}
	  \item the \emph{area} of $\Omega$ is given by
	 \[
      A_\G(\Omega):=\int_\G\mathds{1}_{\Omega}(x)\dx,
      \qquad\text{where}\qquad
      \mathds{1}_\Omega(x):=\begin{cases}
	  1 & \text{if }x\in\Omega\\
	  0 & \text{if }x\in\G\setminus\Omega;
	  \end{cases}
	 \]
	 \item the \emph{perimeter} of $\Omega$ is given by
	 \begin{equation}
	  \label{eq:P}
	  P_\G(\Omega):=\sum_{x\in\partial\Omega}p(x),
    \end{equation}
	where:
	\begin{itemize}
	 \item $\partial\Omega$ is the \emph{topological boundary} of $\Omega$ in $\G$, i.e. the set of those points $x\in\G$ such that every open neighbourhood $U$ of $x$ satisfies $U\cap\left(\G\setminus\Omega\right)\neq\emptyset$,
	 \item $p(x)$ is the number of edges $e\in\G$ such that $x\in e$ and $e\cap\left(\G\setminus\Omega\right)\neq\emptyset$. 
	\end{itemize}
	
    \end{itemize}
 
	\end{definition}
	
	\begin{remark}
		\label{rem-haus}
	 One can easily check that $A_\G(\Omega)=|\Omega|$.
	\end{remark}

	\begin{remark}
		In the context of discrete graphs, i.e. when edges have no associated length, several notions of boundary of a set of vertices have been developed over the years (we redirect the interested reader to \cite{B94} and references therein for a detailed discussion of the topic). Among these, the \emph{edge boundary} of a given subset of vertices is defined as the set of all edges in the graph with one vertex in the considered subset and the other in its complement. We note that the perimeter $P_\G(\Omega)$ as given by \eqref{eq:P} coincides with the sum of the cardinality of the edge boundary of $\Omega\cap\VG$ with $\sum_{x\in\partial\Omega\setminus\VG}p(x)$. Furthermore, as $1\leq p(x)\leq 4$ for every $x\in\partial\Omega$, there results
		\begin{equation}
			\label{eq: P leq H}
			\mathcal{H}^0(\partial\Omega)\geq\f14 P_\G(\Omega),
		\end{equation}
		where $\mathcal{H}^0$ denotes the 0--dimensional Hausdorff measure.
	\end{remark}
	
	\begin{figure}[t]
	\centering
       \subfloat[][$A_\G(\Omega)=41/2$, as $\Omega$ contains 19 edges and 3 half--edges.]
       {\includegraphics[width=.49\columnwidth]{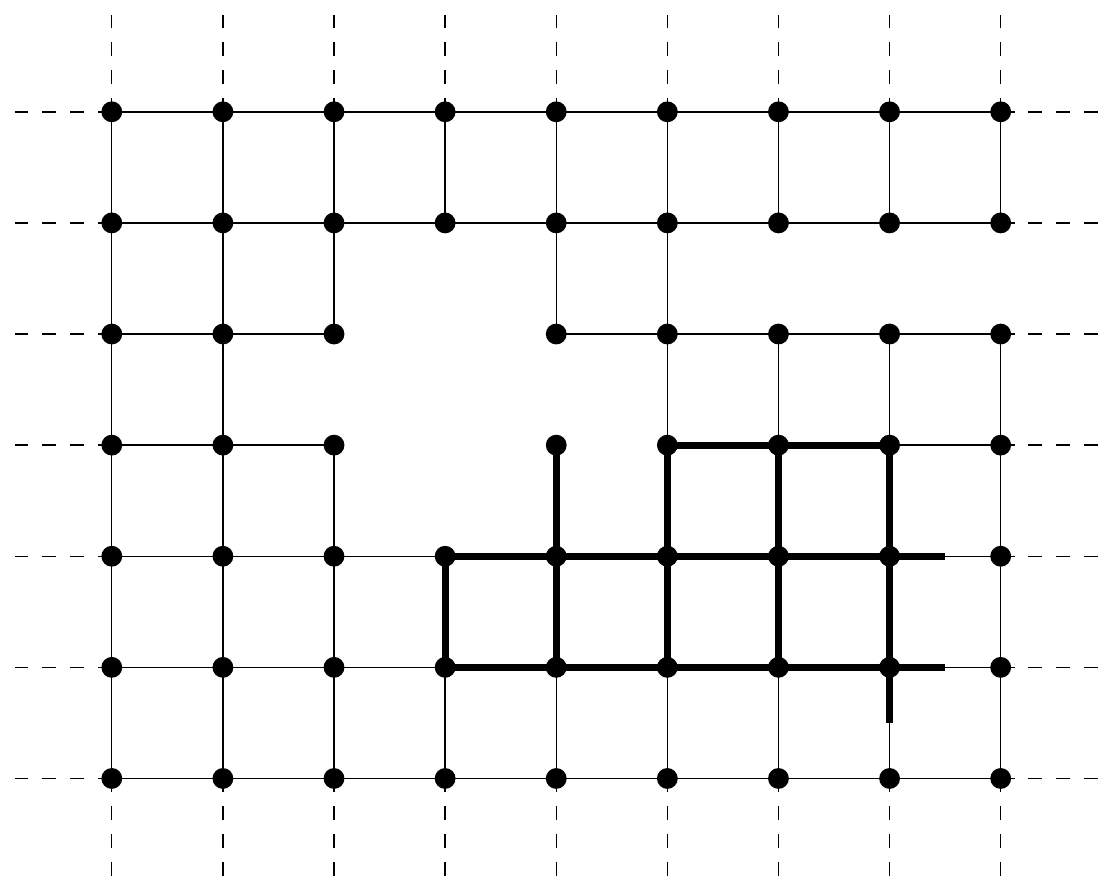}}
       \hspace{.1cm}
       \subfloat[][$P_\G(\Omega)=13$, i.e. the number of the crossed edges. Note that the points of the boundary (circled) are 11.]
       {\includegraphics[width=.49\columnwidth]{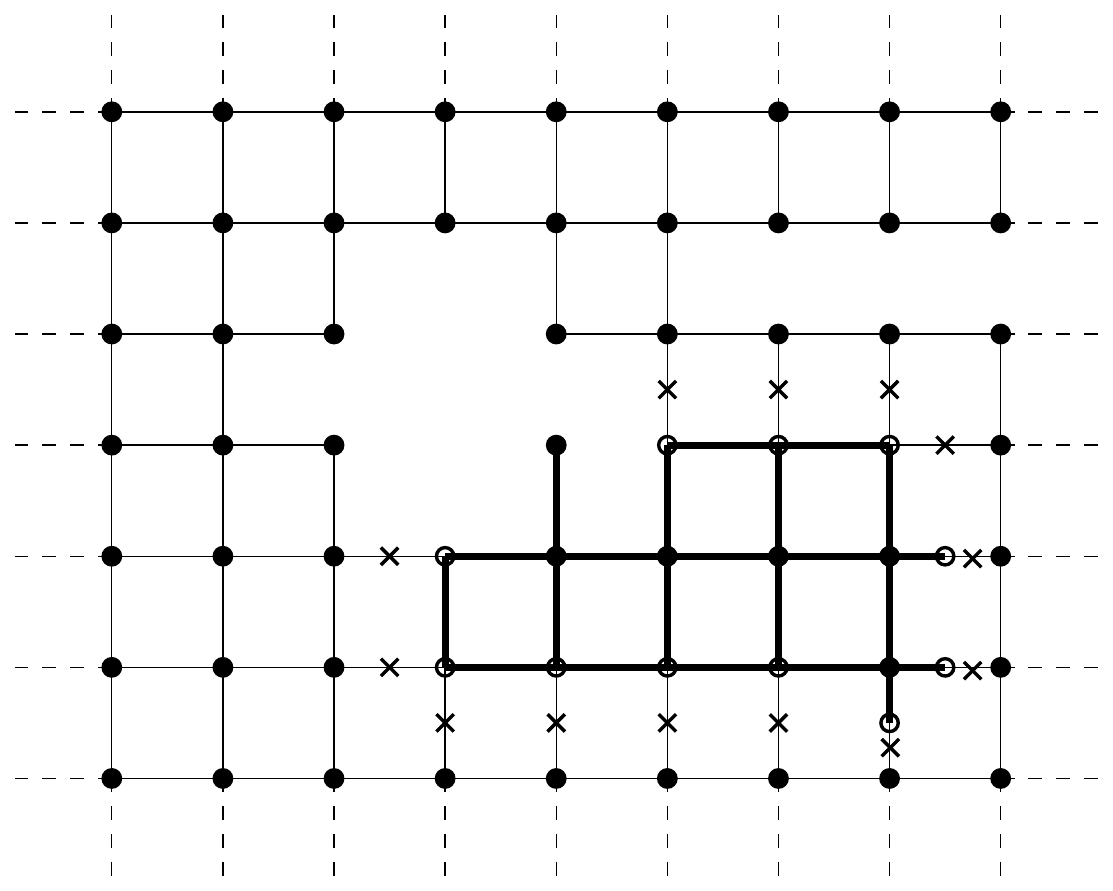}}
       \caption{Area and perimeter of a subset $\Omega$ (in bold) of a defected grid.}
       \label{fig-area}
    \end{figure}
	


\subsubsection{Some classes of defected grids}
\label{subsubsec:subclasses}
	
	Concluding this section, we introduce some special classes of defected grids that will be carefully investigated in the paper.
	In particular, we focus on two main classes and highlight some subclasses. The former class concerns bounded defects.
    
	\begin{definition}(Uniformly bounded defects)
		\label{def:unbdD}
		Let $\G$ be a defected grid and denote by
        \[
        (\D_k)_{k\in I},\qquad\text{for some}\: I\subseteq\N,
        \]
        the family of the defects of $\G$. We say that $\G$ has \emph{uniformly bounded defects} if 
		\begin{equation}
        \label{eq:bound_dD}
		\sup_{k\in I}\left|\D_k\right|<+\infty\,,
		\end{equation}
		where $|\D_k|$ denotes the 1--dimensional Hausdorff measure of $\D_k$.
	\end{definition}
	\begin{remark}
		\label{rem:D e dD finiti}
		Note that, by definition, any defect with finite measure is contained in a compact subset of $\Q$. As a consequence
		\[
		|\D_k|<+\infty\qquad\Leftrightarrow\qquad|\partial\D_k|<+\infty.
		\]
	\end{remark}
	
	\noindent For our purposes, the most interesting subclasses of defected grids given by Definition \ref{def:unbdD} are the following.
	
	\begin{definition}[Compactly defected -- Figure \ref{fig:ex_dam}\textsc{(a)}]
		\label{def:compD}
		Let $\G$ be a grid with uniformly bounded defects. We say that $\G$ is \emph{compactly defected} if, in addition, $I$ is finite.
	\end{definition}
	
	\begin{remark}
		\label{rem:compD}
		Definition \ref{def:compD} is equivalent to say that there exists a compact subset $K\subset\Q$ such that $\bigcup_{k\in I}\D_k\subset K$ (see again Figure \ref{fig:ex_dam}(A)).
	\end{remark}
	
	\begin{definition}[$\Z$--periodic defects -- Figure \ref{fig:ex_dam}\textsc{(b)}]
		\label{def:ZperD}
		Let $\G$ be a grid with uniformly bounded defects. We say that $\G$ has {\em$\Z$--periodic} defects if there exists a unique (up to sign--inversion) vector $\vec{v}$ of $\R^2$ such that 
		\begin{itemize}
			\item[(i)] $\G$ is invariant by integer translations along $\vec{v}$, i.e.
			\[
			\G\cong\G +k\vec{v}\quad\text{in}\:\R^2,\qquad\text{for every}\: k\in\Z;
			\]
			\item[(ii)] there exists $(x_0,y_0)\in\R^2$ and $r>0$ so that
			\[
			\bigcup_{k\in I}\D_k\subset\Q\cap\left\{(x,y)\in\R^2\,:\,\left|(x-x_0,y-y_0)\cdot \vec{v}^\perp\right|\leq r\right\}\,.
			\]
		\end{itemize}
	\end{definition}
	
	\begin{definition}[$\Z^2$--periodic defects -- Figure \ref{fig:ex_dam}\textsc{(c)}]
		\label{def:Z2perD}
		Let $\G$ be a grid with uniformly bounded defects. We say that $\G$ has {\em$\Z^2$--periodic} defects if there exist two linearly independent vectors $\vec{v}_1,\vec{v}_2$ of $\R^2$ such that $\G$ is invariant by translations along integer combinations of $\vec{v}_1$ and $\vec{v}_2$, i.e.
		\[
		 \G\cong\G+k_1\vec{v}_1+k_2\vec{v}_2\quad\text{in}\:\R^2,\qquad\text{for every}\: (k_1,k_2)\in\Z^2.
		\]
	\end{definition}
	
	\begin{remark}
	 In Definitions \ref{def:ZperD}--\ref{def:Z2perD}, for the sake of simplicity, we used the embedding in $\R^2$ and defined periodicity as invariance by suitable integer translations. However, there are several ways to define periodicity of a graph in a more abstract setting, see for instance \cite{BK13, D19, P18}.
	\end{remark}
	
	\begin{remark}
	\label{rem:Dper}
		Following \cite[Chapter 4]{BK13}, if $\G$ is a defected grid with $\Z$--periodic defects, then there exists a subset $W\subset\Q$, with $|W|=+\infty$, such that $\G=\bigcup_{k\in\Z}\widetilde{W}_k$, where $\widetilde{W}_k:=\widetilde{W}+k\vec{v}$ for every $k\in\Z$, $\widetilde{W}:=W\cap\G$ and $\vec{v}$ is as in Definition \ref{def:ZperD} (Figure \ref{fig:ex_dam}\textsc{(b)}). Furthermore, Definition \ref{def:ZperD}(ii) implies that $\left|W\cap\left(\bigcup_{k\in I}\D_k\right)\right|<+\infty$. Similarly, if $\G$ is a defected grid with $\Z^2$--periodic defects, then there exists a subset $W\subset\Q$, with $|W|<+\infty$, such that $\G=\bigcup_{(k_1,k_2)\in\Z^2}\widetilde{W}_{(k_1,k_2)}$, where $\widetilde{W}_{(k_1,k_2)}:=\widetilde{W}+k_1\vec{v}_1+k_2\vec{v}_2$ for every $(k_1,k_2)\in\Z^2$, $\widetilde{W}:=W\cap\G$ and $\vec{v}_1,\vec{v}_2$ are as in Definition \ref{def:Z2perD} (Figure \ref{fig:ex_dam}\textsc{(c)}). In both cases, we call $\widetilde{W}$ a \emph{fundamental domain} of $\G$.
	\end{remark}
	
	\begin{figure}[t]
	\centering
		\subfloat[][Defects contained in a compact subset $K\subset\Q$ (dotted region).]
		{\includegraphics[width=.47\columnwidth]{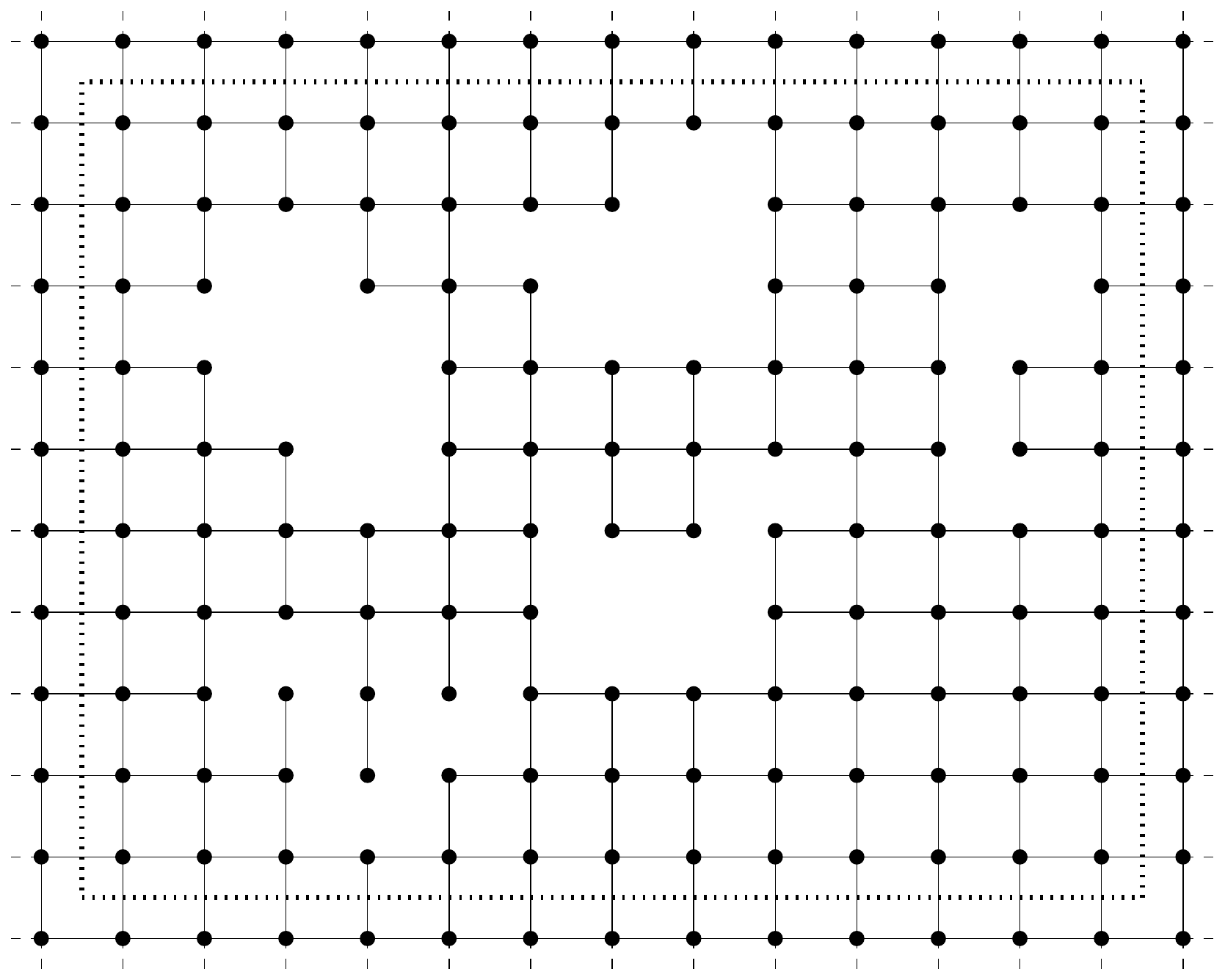}}
		\hspace{.2cm}
		\subfloat[][$\Z$--periodic defects: the vector $\vec{v}$ and the fundamental domain $W$ (dotted region).]
		{\includegraphics[width=.5\columnwidth]{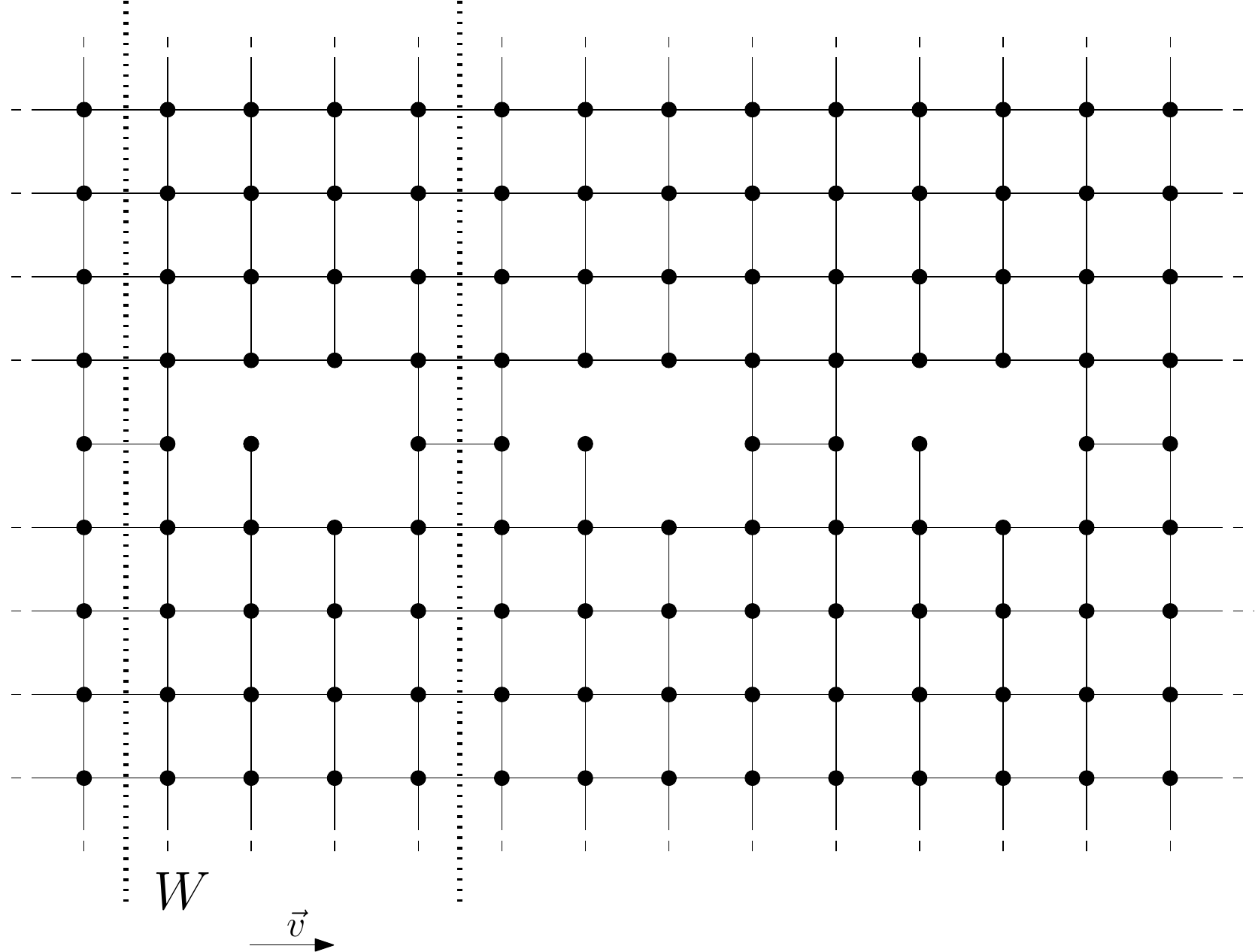}}
		\\
		\subfloat[][$\Z^2$--periodic defects: the vectors $\vec{v}_1,\,\vec{v}_2$ and the fundamental domain $W$ (dotted region).]
		{\includegraphics[width=.52\columnwidth]{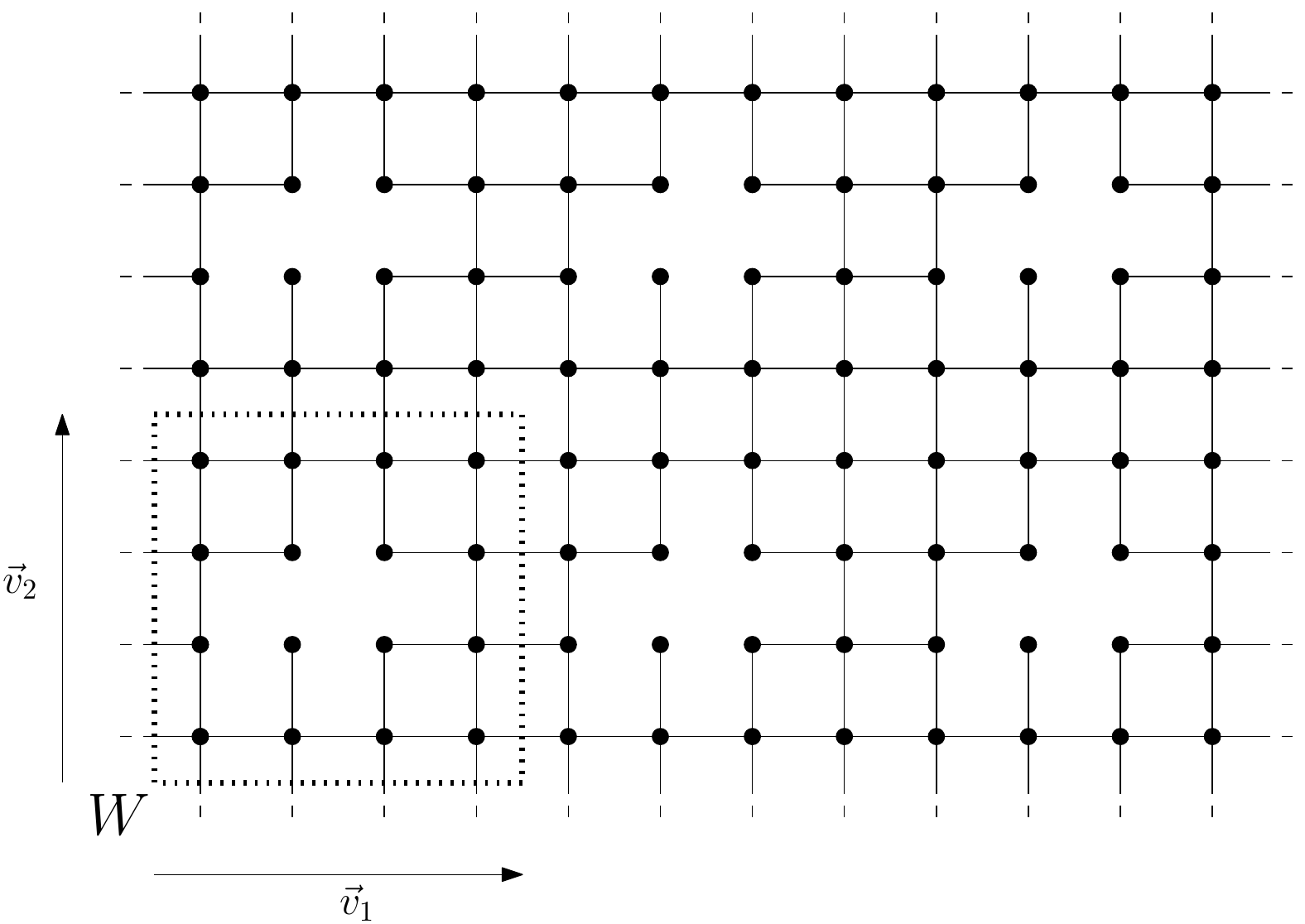}}
		\caption{Examples of defected grids as in Definitions \ref{def:compD}--\ref{def:ZperD}--\ref{def:Z2perD}.}
		\label{fig:ex_dam}
	\end{figure}
	
	The second class of grids we consider concerns defects that are not uniformly bounded.
	
	\begin{definition}[Unbounded defects]
		\label{def:nonbdD}
		Let $\G$ be a defected grid and denote by
        \[
        (\D_k)_{k\in I},\qquad\text{for some}\: I\subseteq\N,
        \]
        the family of the defects of $\G$. We say that $\G$ has \emph{unbounded defects} if
        \begin{itemize}
         \item[(i)] there exists $k\in I$ such that $|\D_k|=+\infty$;
         \item[(ii)] $\displaystyle \sup\{|\D_k|:k\in I, |\D_k|<+\infty\}<+\infty$.
        \end{itemize}
		We denote by $I_\infty$ the set of indices associated with unbounded defects (i.e. $k\in I_\infty$ if and only if $|\D_k|=+\infty$) and by $I_b$ the set of indices associated with bounded defects (i.e. $k\in I_b$ if and only if $|\D_k|<+\infty$). Similarly, we set $(\D_j^\infty)_{j\in I_\infty}$ the family of the unbounded defects of $\G$ and $(\D_k^b)_{k\in I_{b}}$ the family of the bounded defects of $\G$.
	\end{definition}
	
	\begin{remark}
	 According to the previous definition, the family of all the defects of a grid $\G$ with unbounded defects is given by the union of two disjoint families
		\[
		(\D_k^b)_{k\in I_{b}}\cup(\D_j^\infty)_{j\in I_\infty}.
		\]
		Furthermore, it is plainly seen that
		\[
		\sup_{k\in I_{b}}\big|\partial\D_k^b\big|<+\infty\quad\qquad\text{and}\quad\qquad\left|\partial\D_j^\infty\right|=+\infty\,,\quad\forall j\in I_\infty\,.
		\]
    \end{remark}
    
    Among all grids satisfying Definition \ref{def:nonbdD}, we here identify a specific subclass which will be further investigated in the following. Preliminarily, given a vertex $\v\in\VG$, we define
    \[
     P_\v:=\{\gamma_\v\subset\G: \gamma_\v\text{ is a simple path starting at $\v$ with }|\gamma_\v|=+\infty\},
    \]
    namely, the set of all simple paths of infinite length starting at $\v$. Note that $P_\v\neq\emptyset$, for every $\v\in\VG$, by Definition \ref{def:G}.

    \begin{definition}[Assumption (P)]
     \label{def:assP}
     Let $\G$ be a defected grid with unbounded defects. In addition, for every fixed $j\in I_\infty$, let
     
     \begin{itemize}
      \item $ V_j:=\big\{\v\in \VG\cap\partial\D_j^\infty\,:\,\deg(\v)\leq3\big\}$,

      \medskip
      \item $\Gamma_j=(\gamma_\v)_{\v\in V_j}$ be a sequence of paths such that $\gamma_\v\in P_\v$, for every $\v\in V_j$,
    
	  \medskip
	  \item $P_j$ be the family of all possible sequences $\Gamma_j$,

      \medskip
	  \item $\F_j:V_j\times P_j\to[0,+\infty]$ be the function
      \[
       \F_j\left(\v,\Gamma_j\right):=\#\{\w\in V_j\,:\,\gamma_\v\cap\gamma_\w\neq\emptyset\},
      \]
      where $\gamma_\v$ is the unique path in $\Gamma_j$ starting at the vertex $\v$.
     \end{itemize}
     We say that $\G$ \emph{satisfies assumption (P)} if
     \[
       \sup_{j\in I_\infty}\inf_{\Gamma_j\in P_j}\sup_{\v\in V_j} \F_j(\v,\Gamma_j)<+\infty\,.
     \]
    \end{definition}
		
	Roughly speaking, the previous definition states that a defected grid with unbounded defects satisfies assumption (P) if: for every unbounded defect (i.e., $\D_j^\infty$) there exists a ``configuration'' (i.e., $\Gamma_j$) of infinite simple paths starting at the vertices on the boundary of the defect (i.e., $\partial\D_j^\infty$) of degree smaller than or equal to 3, such that a path (i.e., $\gamma_\v$) starting at a fixed vertex (i.e., $\v$) intersects with a number of paths (i.e., $\gamma_\w$) starting at all the other vertices (i.e., $\w$) that is bounded from above by a constant that is independent both of the chosen vertex and of the chosen defect. Note that assumption (P) does not prevent the possible presence of either finitely many paths sharing infinitely many edges (that is, $|\gamma_\v\cap\gamma_\w|=+\infty$ for a finite number of vertices $\w\neq\v$ still yields $\F_j(\v,\Gamma_j)<+\infty$), or infinitely many groups of mutually intersecting paths, provided that the number of paths in each group is uniformly bounded from above (e.g., if there are infinitely many couples of vertices $(\v_i,\w_i)_{i=1}^\infty$ so that $\gamma_{\v_i}\cap\gamma_{\w_i}\neq\emptyset$, but $\gamma_\v\cap\gamma_\w=\emptyset$ if $(\v,\w)\neq(\v_i,\w_i)$ for every $i$, then $\F_j(\v,\Gamma_j)=2$ for every $\v$ and $j$). On the contrary, it guarantees that no edge (or vertex) belongs to infinitely many paths. From a geometrical point of view, one may interpret assumption (P) as a condition that prevents the formation of ``bottlenecks'', i.e. regions of arbitrarily large area confined by curves with uniformly bounded perimeter.
	
	Examples of defected grids satisfying assumption (P) are depicted in Figures \ref{fig:defected}\textsc{(b)} and \ref{fig-area}. In both cases, it is easy to construct a suitable configuration of paths (i.e., the ones consisting of the sole vertical half--lines yield no superposition). Conversely, examples of defected grids that satisfy Definition \ref{def:nonbdD} but violates assumption (P) are given in Figure \ref{fig:noP}.  This time, one can see that, independently of the chosen configuration of paths, there are mandatory passages forcing large superpositions. In particular, in Figure \ref{fig:noP}\textsc{(a)} it is clear that any path starting at a given vertex of the boundary of the unique defect intersects with infinitely many paths starting at other vertices. Similarly, in Figure \ref{fig:noP}\textsc{(b)} any possible path configuration yields a diverging number of superpositions for those vertices on the portions of the boundaries of the two defects that face each other. Note that both the situations correspond to the presence of bottlenecks. In Figure \ref{fig:noP}\textsc{(a)}, if one looks at the graph as an infinite spiral built of consecutive squares, then any unbounded region given by the graph itself except for a finite number of squares in the spiral has infinite area but finite perimeter. The same holds in Figure \ref{fig:noP}\textsc{(b)} taking into account the region enclosed between the two defects.
	
		\begin{remark}
		\label{rem: I<inft_noP}
		An interesting and non--trivial consequence of assumption (P) is that it forces the unbounded defects of $\G$ to be finitely many (see Proposition \ref{prop:fin_ubdD} below). The converse is clearly false, as it is shown, for instance, by the counterexample of Figure \ref{fig:noP}\textsc{(b)}.
	\end{remark}

	\begin{figure}[t]
	\centering
        \subfloat[][Spiraling defect (one defect).]
		{\includegraphics[width=.4\columnwidth]{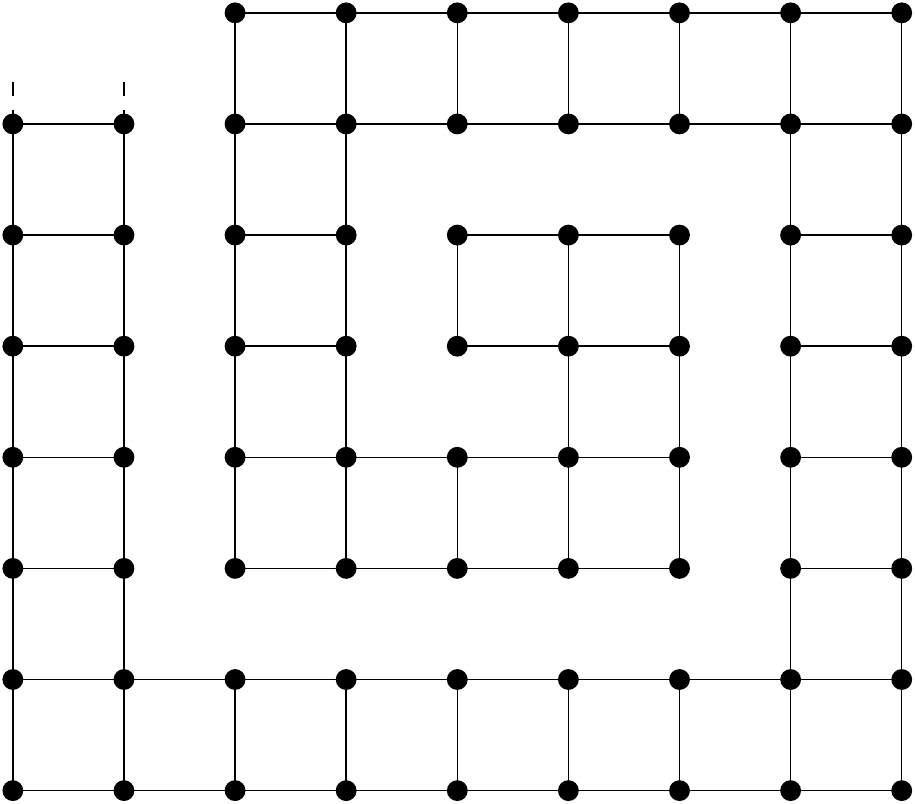}} \qquad
		\subfloat[][Parallel defects (two defects).]
		{\includegraphics[width=.53\columnwidth]{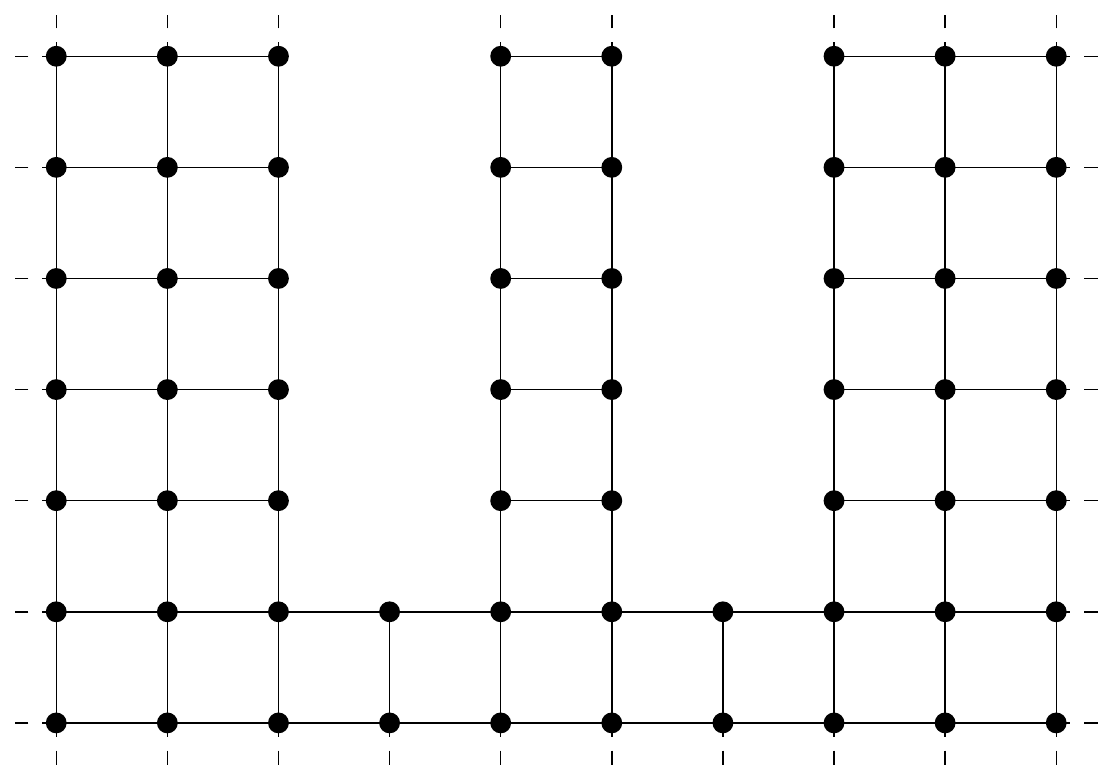}}
		\caption{Examples of defected grids satisfying Definition \ref{def:nonbdD}, but violating assumption (P).}
		\label{fig:noP}
    \end{figure}
    
    \medskip
    Finally, we mention that Definitions \ref{def:unbdD} and \ref{def:nonbdD} do not represent the whole possible phenomenology of defected grids. Indeed, one can consider grids whose bounded defects are not equibounded (e.g., Figure \ref{fig-nounbound}), i.e. grids that satisfy neither Definition \ref{def:unbdD} nor Definition \ref{def:nonbdD}(ii). As mentioned below, we do not address this type of defected grids in the present paper and we postpone their analysis to future investigations.

    

    \subsection{Main results}
    \label{subsec:main}
	
	After this long but necessary introduction we can state the main results of the paper.
	
    

    \subsubsection{Sobolev inequalities}
    \label{subsec:sobineq}

	We aim to detect which types of defects preserve the coexistence of the one and the two--dimensional Sobolev inequalities. One can easily check that the one--dimensional Sobolev inequality 
	\begin{equation}
	 \label{eq:s1d}
	 \|u\|_{L^\infty(\G)}\leq \Gamma_\G\|u'\|_{L^1(\G)},\qquad\forall u\in W^{1,1}(\G),
	\end{equation}
	with $\Gamma_\G$ a positive constant depending only on $\G$, holds on any defected grid like on any other noncompact metric graph, as it is rooted in the local one dimensional nature of metric graphs (see, e.g., \cite{ADST19}). As a consequence, our work focuses on the sole two--dimensional Sobolev inequality
	\begin{equation}
		\label{eq:s2d}
		\|u\|_{L^2(\G)}\leq S_\G\|u'\|_{L^1(\G)}, \qquad\forall u\in W^{1,1}(\G)\,,
	\end{equation}
	where $S_\G$ is again a positive constant depending only on $\G$. The validity of this inequality is a first marker of the persistence of a two--dimensional macroscale in the defected grid.  
	
	\medskip
	In Euclidean spaces, it is well--known that Sobolev inequalities are equivalent to the isoperimetric ones (see for instance \cite[pp.487--488]{FF60} and the review \cite{O78}) and the interplay between these two families of inequalities in general metric spaces has been the object of extensive investigations in the last decade of the previous century (see the comprehensive discussions \cite{BH97,HK00} and references therein). It is not surprising that the same equivalence holds true in the context of defected grids.
	\begin{theorem}
		\label{THM:iso=sob}
		Let $\G$ be a defected grid. Then $\G$ supports \eqref{eq:s2d} if and only if there exists $C_\G>0$ (depending only on $\G$) such that
		\begin{equation}
		\label{intro:iso G}
		\sqrt{A_\G(\Omega)}\leq C_\G P_\G(\Omega)
		\end{equation}
		for every bounded $\Omega\subset\G$, where $A_\G(\Omega), P_\G(\Omega)$ are the area and the perimeter of $\Omega$ as in Definition \ref{def-AP}.
	\end{theorem}
	
	\begin{remark}
	 Note that \eqref{intro:iso G} is clearly valid also on the undefected grid and this is consistent with the fact that the undefected grid supports the two--dimensional Sobolev inequality \eqref{eq:2dGN}.
	\end{remark}

	Although it is the first result of the paper, Theorem \ref{THM:iso=sob} is nothing more than the statement of a general principle in the specific setting we are considering. However, we reported it here explicitly both because it is of actual help for the proofs of some of our next theorems, and because it provides a first condition which is equivalent to the validity of the two--dimensional Sobolev inequality. We also mention that isoperimetric problems on grids and more general planar graphs have been widely considered in the discrete setting (see for instance \cite{BE18,BP,BL91,BL91bis,KL18,WW77}), and investigations on related issues started recently also in the metric framework \cite{KN,Nic20}.  
	
	Despite its great generality, the equivalence between Sobolev and isoperimetric inequalities does not provide any insight on the features of the classes of defects that preserve the validity of \eqref{eq:s2d} and of those that break it. Even from a more operative standpoint, it does not suggest an useful criterion to check whether these inequalities hold true on a given graph, as in general to prove or disprove the isoperimetric inequality computationally is known to be NP--hard \cite{BE18}.
	
	To exploit the nature of the defects, in the following we establish two geometric sufficient conditions on $\G$ ensuring \eqref{eq:s2d}. From a technical point of view, the proof of both theorems relies on a suitable extension argument: given a function on $\G$, one constructs a proper extension to $\Q$ ``filling the holes" caused by the defects. To be of some help, such a procedure has to be performed so that the Sobolev inequality on $\Q$ applied to the extended function yields the desired inequality on the original function on $\G$. Despite the fact that they appear sensibly different, both geometric conditions guarantee that this general strategy works.
	
	\begin{theorem}
		\label{THM:s2d_bound}
		If $\G$ is a grid with uniformly bounded defects (Definition \ref{def:unbdD}), then it supports \eqref{eq:s2d}.
	\end{theorem}
	
	In other words, Theorem \ref{THM:s2d_bound} states that the class of perturbations of the undefected grid given by Definition \ref{def:unbdD} does not affect the two--dimensional Sobolev inequality. The main advantage of the assumption of uniform boundedness on the defects is that it provides a precise geometric idea of the types of defected grids that one is taking into account. Moreover, the family of grids fulfilling this condition is rather wide and varied, as it contains for instance compactly defected grids, grids with $\Z$ and $\Z^2$--periodic defects (see Definitions \ref{def:compD}, \ref{def:ZperD} and \ref{def:Z2perD} and, e.g., Figures \ref{fig:defected}\textsc{(a)} and \ref{fig:ex_dam}). However, it is a much stronger requirement than \eqref{intro:iso G}. The main drawback is that it rules out grids with unbounded defects, for which a finer analysis is needed.
	
	\begin{theorem}
		\label{THM:s2d_P}
		If $\G$ is a grid with unbounded defects (Definition \ref{def:nonbdD}) that satisfies assumption (P) as in Definition \ref{def:assP}, then it supports \eqref{eq:s2d}.
	\end{theorem}
	
	\begin{figure}[t]
	\centering
    \includegraphics[width=.75\columnwidth]{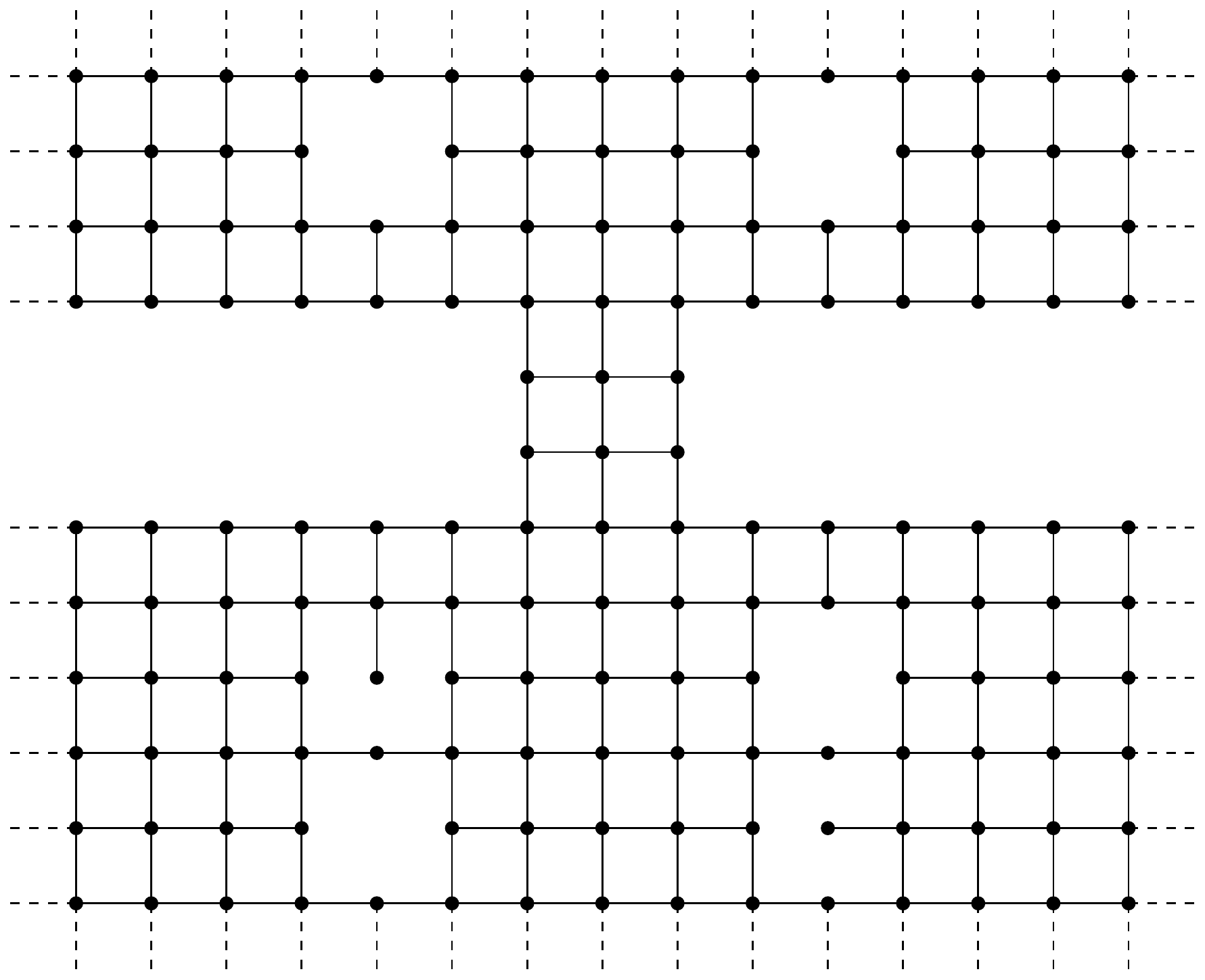}
    \caption{Another example of a defected grid satisfying assumption (P).}
    \label{fig:anotherP}
    \end{figure}
	
	The most remarkable point in Theorem \ref{THM:s2d_P} is that it identifies a topological condition on the defects, and more precisely on the topology of the sole unbounded defects, which is sufficient for the validity of the two--dimensional Sobolev inequality. In this regard, the role of topology is peculiar of unbounded defects, as Theorem \ref{THM:s2d_bound} above shows that nothing similar takes place when we are considering uniformly bounded defects only. Again, the family of defected grids fulfilling assumption (P) is quite variegate, as it includes for instance the graphs in Figures \ref{fig:defected}\textsc{(b)}, \ref{fig-area} and \ref{fig:anotherP}. Conversely, such assumption dismisses topologies as those depicted in Figure \ref{fig:noP}, where (as explained right below Definition \ref{def:assP}) one can easily show that \eqref{intro:iso G} is not satisfied due to the presence of bottlenecks, thus forbidding \eqref{eq:s2d} in view of Theorem \ref{THM:iso=sob}.
	
	Though it proves to be useful, at first glance the rigorous statement of assumption (P) as in Definition \ref{def:assP} may appear a bit mysterious. In particular, it is not straightforward to guess why such an assumption should be connected to the validity of the two--dimensional Sobolev inequality. Let us try to shed some light on this point with the following simplified computation. Let $\D$ be an unbounded defect of a grid $\G$ and let $\partial\D$ be its boundary. Let also $u\in W^{1,1}(\G)$. For every edge $e\in\partial\D$, one has
	\[
	|u(x)|\leq\int_{\Gamma_e}|u'|\qquad\forall x\in e\,,
	\]
	where $\Gamma_e\subset\G$ is any simple path of infinite length whose first edge is $e$. This is nothing more than the one--dimensional Sobolev inequality applied to the restriction of $u$ to the path $\Gamma_e$. Squaring both sides of the previous inequality and integrating over $e$ yields
	\begin{equation}
	\label{easyP_1}
	\|u\|_{L^2(e)}^2\leq\left(\int_{\Gamma_e}|u'|\right)^2.
	\end{equation}
	Namely, the one--dimensional nature of the grid ensures that one can always control the $L^2$--norm of $u$ on a given edge $e$ with the $L^1$--norm of its derivative $u'$ on a suitable path of infinite length starting at that edge. Of course, different edges $e$ require different paths $\Gamma_e$. 
		
	Suppose now that it is possible to construct a family of paths of infinite length in such a way that there is at least one path starting at each vertex of $\partial\D$ and that there is no edge in the grid belonging to more than $M$ of these paths, for some constant $M>0$. Since by definition every edge $e\in\partial\D$ shares at least one vertex of $\partial\D$, this means that there is a family of paths $\Gamma_e$ (indexed by the edges of $\partial\D$) such that at most $M$ of them share any given edge of $\G$. Hence, summing over $e\in\partial\D$ in \eqref{easyP_1} leads to
	\[
	\|u\|_{L^2(\partial\D)}^2=\sum_{e\in\partial\D}\|u\|_{L^2(e)}^2\leq\sum_{e\in\partial\D}\left(\int_{\Gamma_e}|u'|\right)^2\leq M\|u'\|_{L^1(\G)}^2,
	\]
	where the last inequality relies on the fact that in the last sum the contribution of the $L^1$--norm of $u'$ on each path $\Gamma_e$ appears at most $M$ times. The previous estimate shows that one can control the $L^2$--norm of $u$ on the whole boundary of the chosen unbounded defect $\partial\D$ with the square of the $L^1$--norm of  $u'$ on the whole grid. In general, this is a phenomenon one does not expect to arise when the two--dimensional Sobolev inequality does not hold. Indeed, to violate the two--dimensional Sobolev inequality one would like to exploit the essentially one--dimensional nature of the boundary of the defects, constructing functions with arbitrarily large mass close to the defects but uniformly bounded $L^1$--norm of the derivative. 
	
	Clearly, the above calculation is not even close to a proof. However, it seems to suggest that the validity of the two--dimensional Sobolev inequality is somehow connected to the possibility of running arbitrarily far away from any unbounded defect without being forced to visit too often any given region of the grid. This is indeed the idea behind assumption (P), which is nothing but a precise mathematical formulation of such a geometric idea. 
	
	Of course, assumption (P) is stronger than the simplified hypothesis used in the previous illustrative computation, which was dealing with a single unbounded defect only. In particular, as already anticipated in Remark \ref{rem: I<inft_noP} and proved in Proposition \ref{prop:fin_ubdD} below, the kind of uniformity required by assumption (P) is rather deep, as it implies that the number of unbounded defects in the grid is finite (and this is crucial to complete the proof of Theorem \ref{THM:s2d_P}). On the one hand, this naturally raises the question whether one can exhibit a grid with infinitely many unbounded defects that hosts the two--dimensional Sobolev inequality (see also Section \ref{subsec:open}). On the other hand, one could wonder if requiring a finite number of unbounded defects is a sufficient, and way easier than assumption (P), hypothesis for the two--dimensional Sobolev inequality. However, Figure \ref{fig:noP} highlights that, even assuming a priori a finite number of unbounded defects, one has no chance to get the result in general, unless further information on the way the defects are displaced is available.
	
	\begin{figure}[t]
		\centering
		\includegraphics[width=0.7\textwidth]{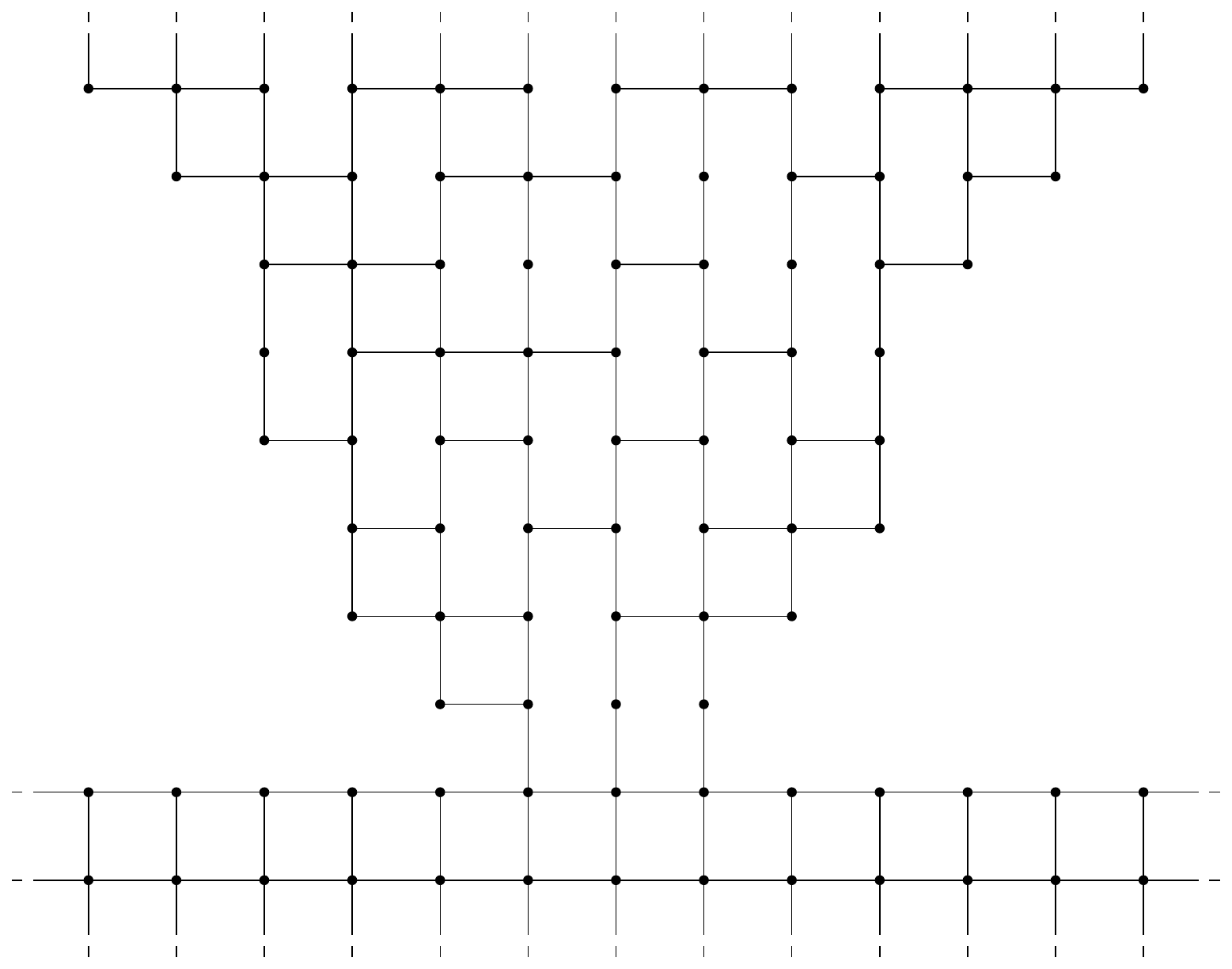}
		\caption{Example of a grid with two unbounded defects and infinitely many bounded ones. The uniformly bounded defects involve the removal of horizontal edges only. It is immediate to see that assumption (P) is fulfilled, as one can simply take the family of all vertical paths starting at the vertices of the boundaries of the two unbounded defects.}
		\label{fig:easyP}
	\end{figure}
	
	Finally, we underline once more that assumption (P) provides a sufficient condition for \eqref{eq:s2d} that involves the unbounded defects only. This may be particularly relevant to rapidly understand whether a grid with few unbounded defects but infinitely many uniformly bounded ones supports the two--dimensional Sobolev inequality. In fact, it is not difficult to think of examples where the validity of assumption (P) is almost evident, whereas it may not be that easy for instance to investigate the isoperimetric inequality (see, e.g., Figure \ref{fig:easyP}).
	
	\medskip
	Even though it starts to unravel the role of the topology of defects, assumption (P) is not sharp, as it can be proved not to be equivalent to \eqref{intro:iso G}, and thus to \eqref{eq:s2d}.
	
	\begin{theorem}
	 \label{THM:SobNoP}
	 There exist grids with unbounded defects that do not satisfy assumption (P) but support \eqref{eq:s2d}.
	\end{theorem}
 The counterexample used to prove Theorem \ref{THM:SobNoP} is the one in Figure \ref{fig:s2d_noP}. The underlying geometric motivation is that, while clearly there is no subset of $\G$ that violates \eqref{intro:iso G}, one can construct a sequence of arbitrarily large regions of $\G$ along which the superpositions among paths diverge.
	
	\begin{figure}[t]
    \centering
    \includegraphics[width=1\columnwidth]{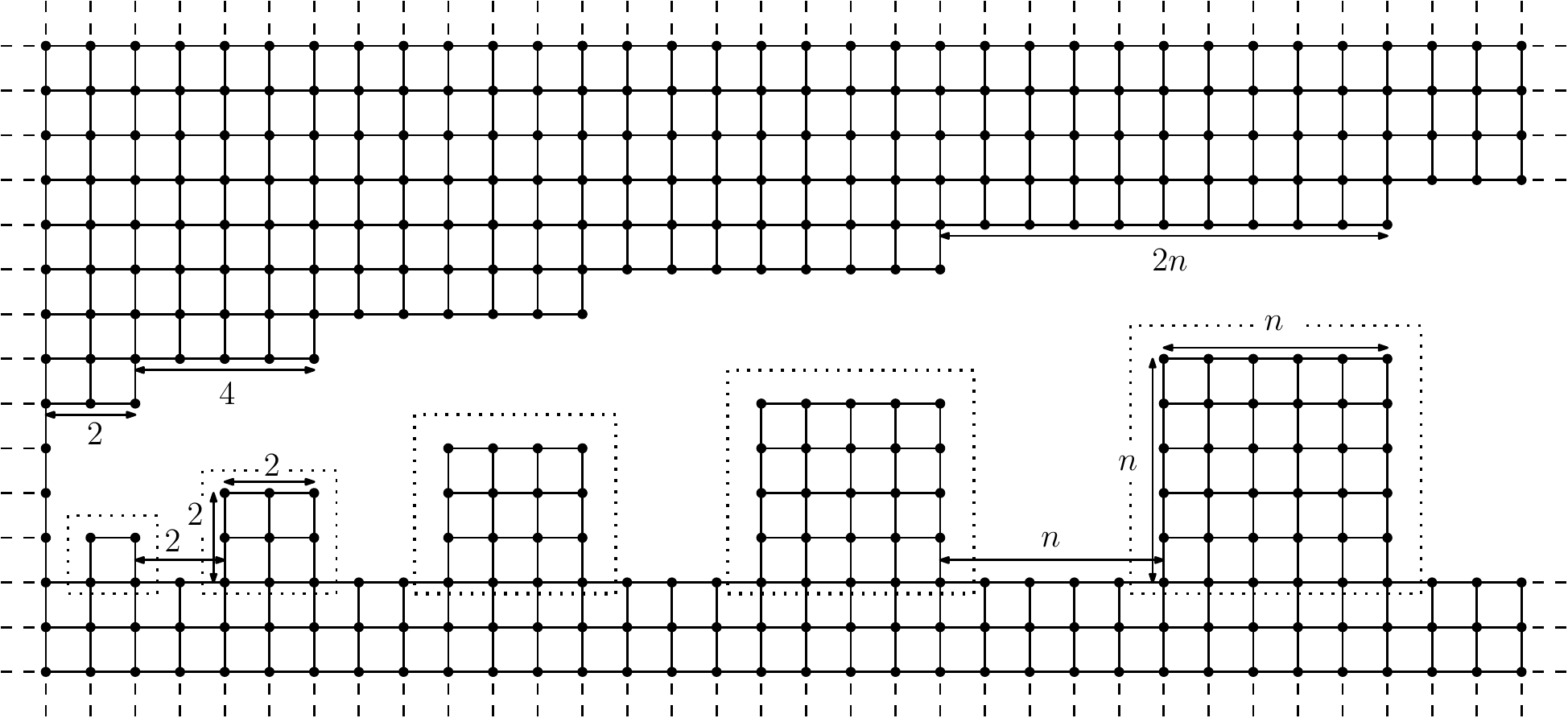}
    \caption{Example of a defected grid supporting the two--dimensional Sobolev inequality but violating assumption (P).}
    \label{fig:s2d_noP}
    \end{figure}

    

    \subsubsection{NLSE ground states}
    \label{subsec:NLSground}
    
    Concerning NLSE ground states, we are interested in detecting which types of defects preserve the phenomenology observed on the undefected grid described in Section \ref{subsec:grid}.
    
    Recall that a \emph{ground state of mass} $\mu>0$ on $\G$ is a function $u\in\Hmu(\G)$ such that
    \[
	E_p(u,\G)=\EE_{p,\G}(\mu):=\inf_{v\in\Hmu(\G)}E_p(v,\G),
	\]
	where
	\begin{equation}
		\label{intro:E}
		E_p(u,\G):=\f12\|u'\|_{L^2(\G)}^2-\f1p\|u\|_{L^p(\G)}^p
	\end{equation}
	and
	\[
	\Hmu(\G):=\{u\in H^1(\G)\,:\,\|u\|_{L^2(\G)}^2=\mu\}.
	\]
	We start with the following existence results. Note that the classes of defected grids managed here have uniformly bounded defects and thus, by Theorem \ref{THM:s2d_bound}, support \eqref{eq:s2d}.
	
	\begin{theorem}
		\label{THM:ex_com&Z}
		Let $\G$ be either a compactly defected grid (Definition \ref{def:compD}) or a grid with $\Z$--periodic defects (Definition \ref{def:ZperD}) or a grid with  $\Z^2$--periodic defects (Definition \ref{def:Z2perD}). It holds that
		\begin{itemize}
			\item[$(i)$] if $p\in(2,4)$, then ground states of mass $\mu$ exist for every $\mu>0$;
			\medskip
			
			\item[$(ii)$] if $p\in[4,6)$, then there exists a critical mass $\mu_{p,\G}$ (depending only on $p$ and $\G$) such that 
			\begin{equation}
				\label{levelTHM1}
				\EE_{p,\G}(\mu)\begin{cases}
				=0, & \text{if }\mu\leq\mu_{p,\G},\\
				<0, & \text{if }\mu>\mu_{p,\G},
				\end{cases}
			\end{equation}
			and
			\medskip
			
			\begin{enumerate}
				\item[$(ii.1)$] if $p\in(4,6)$, then ground states of mass $\mu$ exist if and only if $\mu\geq\mu_{p,\G}$;
				\medskip
				
				\item[$(ii.2)$] if $p=4$, then ground states of mass $\mu$ exist when $\mu>\mu_{4,\G}$, whereas they do not exist when $\mu<\mu_{4,\G}$.
			\end{enumerate}
		\end{itemize}
		Furthermore, if $\G$ is either a compactly defected grid or a grid with $\Z$--periodic defects, then 
		\begin{equation}
		\label{muG<muQ}
		\begin{split}
		\mu_{p,\G}<&\mu_{p,\Q}, \qquad\text{for }p\in(4,6),\\
		\mu_{4,\G}\leq&\mu_{4,\Q}, \qquad\text{for }p=4,
		\end{split}
		\end{equation}
		and
		\begin{equation}
			\label{EG<EQ}
			\EE_{p,\G}(\mu)\begin{cases}
			<\EE_{p,\Q}(\mu) & \text{if }p\in(2,4)\text{ and }\mu>0,\text{ or }p\in[4,6)\text{ and }\mu>\mu_{p,\G}\,,\\
			=\EE_{p,\Q}(\mu) & \text{if }p\in[4,6)\text{ and }\mu\leq\mu_{p,\G}\,.
			\end{cases}
		\end{equation}
	\end{theorem}
	
	On the one hand, these results show that, whenever defects are either confined to a bounded region of the grid or share some periodicity, the defected grid inherits the features of the undefected one, as the dimensional crossover is preserved not only from the point of view of the Sobolev inequalities, but also from the standpoint of the NLSE ground states. We highlight that, according to Definition \ref{def:ZperD}(ii) and Remark \ref{rem:Dper}, here grids with $\Z$--periodic defects are such that all the defects are contained in a bounded strip parallel to the direction along which the grid is periodic. The fact that the first part of Theorem \ref{THM:ex_com&Z} applies also to grids with $\Z^2$--periodic defects shows that imposing periodicity along two directions is enough to guarantee existence of grounds states on grids with uniformly bounded defects that do not satisfy neither Definition \ref{def:compD} nor Definition \ref{def:ZperD}(ii).
	
	On the other hand, Theorem \ref{THM:ex_com&Z} establishes a sort of ``energetic convenience" of compactly defected grids and grids with $\Z$--periodic defects with respect to the undefected grid. More precisely, on these grids ground states of a fixed mass $\mu$ attain strictly lower energy levels with respect to ground states of the same mass on the undefected grid (see \eqref{EG<EQ}). Furthermore, in the critical range of exponents $p\in(4,6)$, ground states exist for a strictly larger interval of masses with respect to the undefected case (see \eqref{muG<muQ}). At the moment we are not able to prove this energetic convenience also for grids with $\Z^2$--periodic defects, although we can exhibit some intuition in this direction (see the next section). 
	
	Periodicity assumptions as in Theorems \ref{THM:ex_com&Z} may seem quite a strict requirement. Moreover, one could guess that it is at least possible to remove condition (ii) in Definition \ref{def:ZperD} and hope to recover the previous results in the case of grids with $\Z$--periodic defects not necessarily confined in a strip parallel to the direction of periodicity. However, this is not the case in general.
	
	\begin{theorem}
	\label{thm:noground}
	 There exist grids with uniformly bounded defects that
	 \begin{itemize}
	 	\item[(i)] satisfy Definition \ref{def:ZperD}(ii) but do not satisfy Definition \ref{def:ZperD}(i), or
	 	\item[(ii)] satisfy Definition \ref{def:ZperD}(i) but satisfy neither Definition \ref{def:ZperD}(ii) nor Definition \ref{def:Z2perD},
	 \end{itemize} and that do not admit ground states of mass $\mu$, for any $\mu>0$ and any $p\in(2,6)$. 
	\end{theorem}

	The counterexamples in the proof of Theorem \ref{thm:noground} are far from being trivial and cannot be easily drawn. The construction is explained in full details in Section \ref{subsec:nonex}. Loosely, the proof of Theorem \ref{thm:noground}(i) involves a non--periodic removal of infinitely many vertical edges in a given horizontal strip that yields a serious loss of compactness at infinity, and Theorem \ref{thm:noground}(ii) follows by periodically repeating along the vertical direction the same configuration of defects.
	
	The main relevance of Theorem \ref{thm:noground} is that it marks a sharp distinction between the two standpoints we used in our discussion. Indeed, by Theorem \ref{THM:s2d_bound} above, the grids in Theorem \ref{thm:noground} support the Sobolev inequality \eqref{eq:s2d}. Hence, there are defected grids which present the same behavior of the undefected grid for what concerns Sobolev inequalities, while exhibit a completely different behavior for what concerns NLSE ground states. From a technical perspective, this is not unexpected, since to observe the dimensional crossover at the level of NLSE ground states one needs not only the simultaneous validity of the one--dimensional  and the two--dimensional Sobolev inequalities, but also a suitable topology preventing the possible loss of compactness for minimizing sequences of the energy.
	


    \subsection{Open problems}
    \label{subsec:open}

	As highlighted in the previous sections, that of the persistence of the dimensional crossover in defected grids is a rich and many--sided issue. Hence, our results do not provide a complete discussion of the topic. On the contrary we hope that, shedding some light on the main and deep--seated features of the problem, they may stimulate further investigations in the coming years. To this aim, we list here some of the open questions we think could be of interest for future research.
	


	\begin{figure}
	\centering
	\begin{tikzpicture}[decoration=brace]
	
	\node  at (-3,0) {$\text{Two--dimensional}$};
	\node at (-3,-0.4) {$\text{Sobolev inequality}$};
	\node  at (-7,-5.3) {$\text{Assumption (P)}$};
	\node  at (3,-5) {$\text{Finite number of}$};
	\node at (3,-5.4) {$\text{unbounded defects}$};
	
	\draw[->, double] (-4.6,-.8)--(-7.1,-4.7) node[pos=0.3,sloped,above] {\text{\small Theorem \ref{THM:SobNoP}}};
	\draw[<-,double] (-4.1,-.8)--(-6.6,-4.7) node[pos=0.7,sloped,below] {\text{\small Theorem \ref{THM:s2d_P}}};
	\node at (-6,-3) {\textbf{\Huge $+$}};
	\draw[->, double] (-5.2,-5.1)--(1,-5.1) node[pos=0.3,sloped,above] {\text{\small Proposition \ref{prop:fin_ubdD}}};
	\draw[<-,double]  (-5.2,-5.4)--(1,-5.4) node[pos=0.7,sloped,below] {\text{\small Remark \ref{rem: I<inft_noP}}};
	\node at (-2,-5.4) {\textbf{\Huge $\times$}};
	\draw[<-,double] (-1.5,-0.8)--(2.6,-4.7) node[pos=0.7, sloped,below] {\text{\small Remark \ref{rem: I<inft_noP}}};
	\node at (0.5,-2.7) {\textbf{\Huge $+$}};
	\draw[->,double] (-1,-0.8)--(3.1,-4.7) node[pos=0.5, sloped,above] {\text{\small $?$}};
	\end{tikzpicture}
	\caption{A scheme of the relations among assumption (P), finite number of unbounded defects and the two--dimensional Sobolev inequality.}
	\label{fig:scheme}
\end{figure}
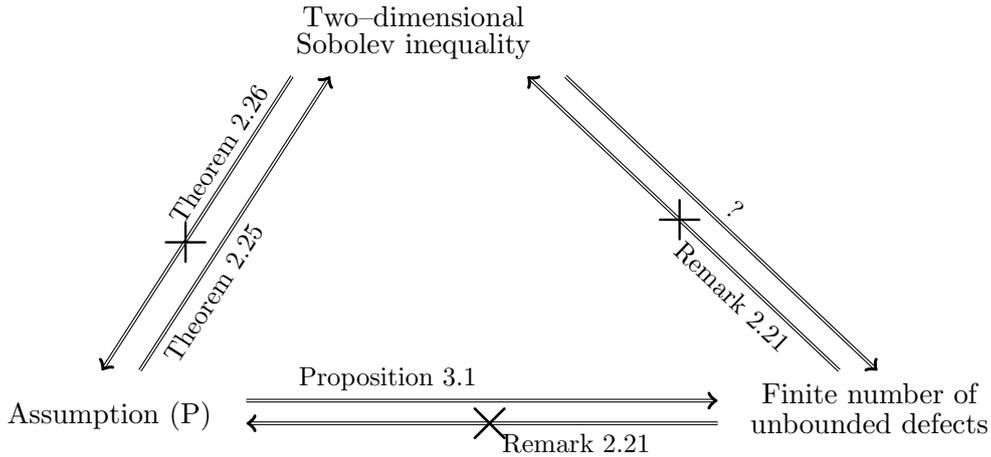

    \subsubsection{Sobolev inequalities}

	Here the picture is completely clear for grids with uniformly bounded defects, while grids with unbounded defects raise some challenging open issues.

	The first one that deserves to be further discussed is the connection among assumption (P), a finite number of unbounded defects and the two--dimensional Sobolev inequality (Figure \ref{fig:scheme}). On the one hand, by Theorem \ref{THM:SobNoP} we know that assumption (P) is not a sharp condition for the validity of \eqref{eq:s2d}. Therefore, it is natural to wonder ``how far'' it is from sharpness. Moreover, given that it is not immediate to check (although easier than \eqref{intro:iso G}), it is also natural to wonder if there exist equivalent conditions with a more direct geometrical explanation. On the other hand, in Remark \ref{rem: I<inft_noP} we explained how assumption (P) implies a finite number of unbounded defects, but also how a finite number of unbounded defects entails neither assumption (P) nor \eqref{eq:s2d}. It is open, on the contrary, if \eqref{eq:s2d} may imply a finite number of unbounded defects or if it is possible to exhibit an explicit example of a grid with infinitely many unbounded defects supporting \eqref{eq:s2d}.

Another point which has not been touched by the paper is the study of those defected grids that do not satisfy Definition \ref{def:nonbdD}(ii), i.e. grids with infinitely many bounded defects which are not uniformly bounded. While it is clear that they cannot support in general the two--dimensional Sobolev inequality (a counterexample is given, for instance, by Figure \ref{fig-nounbound}), it is interesting to investigate the possibility of detecting an extended version of assumption (P) that may apply to this family of defected grids.
\begin{figure}[t]
	\centering
	 \includegraphics[width=.8\columnwidth]{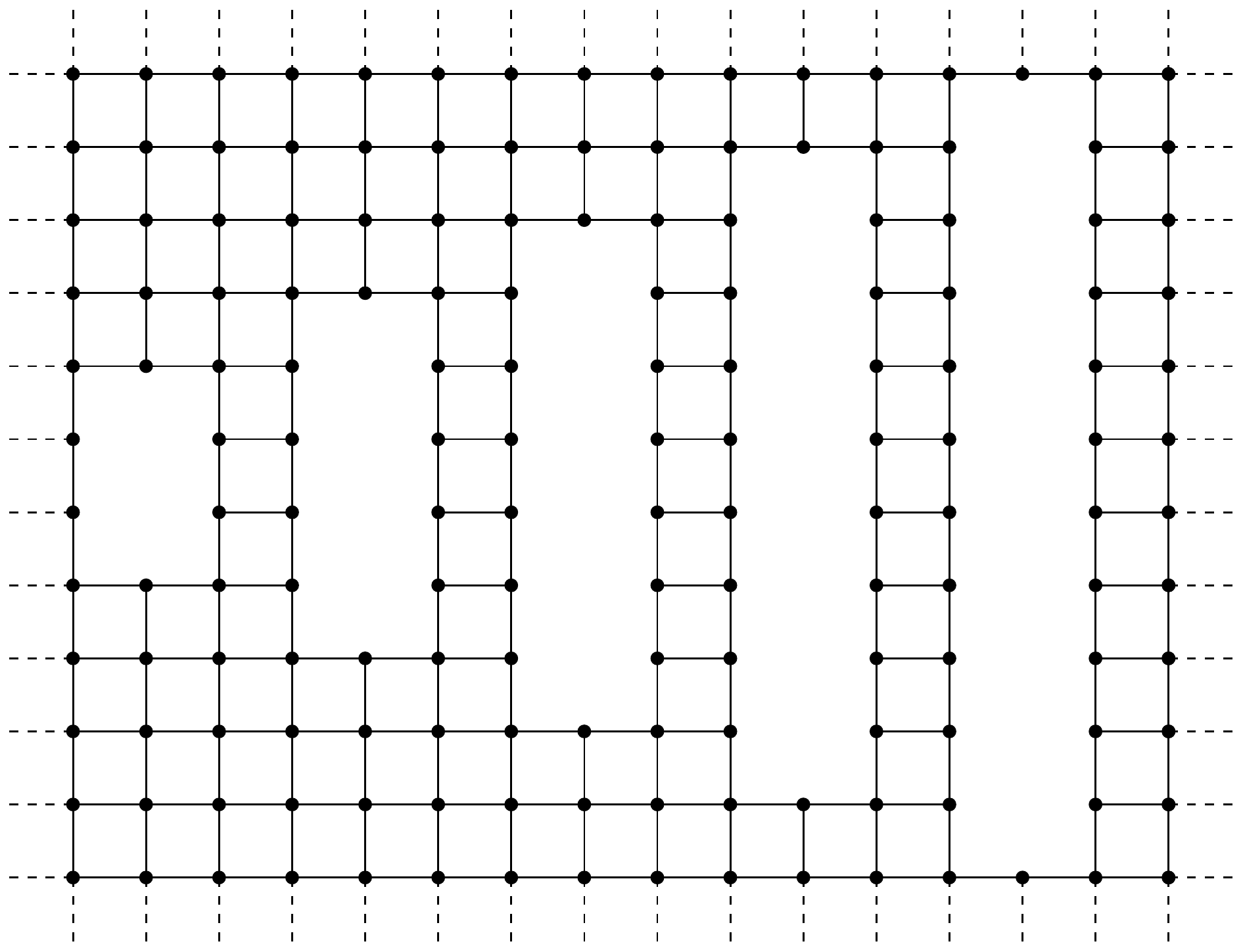}
	\caption{Example of a defected grid with infinitely many bounded defects that are not uniformly bounded: moving rightward, the defects become larger and larger in the vertical direction. As the region between two consecutive defects has always width one, it provides a sequence of subsets of the grid violating the isoperimetric inequality.}
	\label{fig-nounbound}
\end{figure}

A further element that may deserve additional analyses is the role of condition (iii) in Definition \ref{def:G}. As already pointed out, it is for us sort of a ``zero level assumption" that sets the floor of our discussion. However, we do not make use of it explicitly in the proofs of our results. Moreover, it is readily seen that such a condition can be removed as long as we restrict our attention to grids with uniformly bounded defects as in Definition \ref{def:unbdD}. Indeed, one may start by defining a defected grid with uniformly bounded defects as any graph $\G$ that fulfils both conditions (i)--(ii) of Definition \ref{def:G} and \eqref{eq:bound_dD}. If by contradiction one assumes that there exists $\v\in\G$ and $n\to+\infty$ so that 
\begin{equation}
\label{eq:noball}
\lim_{n\to+\infty}\f{|B_n(\v,\G)|}{|B_n(\v,\Q)|}=0\,,
\end{equation}
then by the uniform boundedness of the defects it follows that $B_n(\v,\G)$ intersects a number of different defects proportional to $n^2$. But this implies that the boundary of each of these defects contains at least one edge inside $B_n(\v,\G)$. Since any edge of $\G$ belongs to the boundary of at most two different defects, this means that $B_n(\v,\G)$ contains at least $cn^2$ edges, for some $c>0$ independent of $n$, violating \eqref{eq:noball}. Hence, the exclusive presence of uniformly bounded defects automatically guarantees the validity of condition (iii) in Definition \ref{def:G}. In view of this, it is natural to wonder whether, for grids with unbounded defects, assumption (P) implies Definition \ref{def:G}(iii). More generally, it would be interesting to understand whether the validity of the isoperimetric inequality \eqref{intro:iso G} is enough to ensure condition (iii) in Definition \ref{def:G}. Note that we already know that the converse is false in general, as Figure \ref{fig:noP}\textsc{(b)} provides an example of a defected grid that satisfies Definition \ref{def:G}(iii) but does not support \eqref{intro:iso G}.


    \subsubsection{NLSE ground states}
    \label{sec:openNLS}
	The first open point in this context concerns the energetic convenience of grids with periodic defects with respect to the undefected grid. As we mentioned in the previous section, even though we cannot prove that $\Z^2$-periodic defects present the same energetic convenience of $\Z$--periodic ones, we can construct examples in which this occurs, namely in which \eqref{muG<muQ} and \eqref{EG<EQ} hold. A simple case is depicted in Figure \ref{fig-z2comez1} (more details are provided by Remark \ref{rem:level_Z2per}). It is then interesting to understand whether this can be extended to any grid with $\Z^2$--periodic defects.
	
	Furthermore, as shown by Theorem \ref{thm:noground}, whenever one removes the periodicity (or the compactness) assumption on the set of the defects, the existence of the NLSE ground states is not granted in general. However, at the same time, the construction of a counterexample is far from being immediate. It could be interesting to search for an intermediate condition between uniform boundedness and periodicity of the defects that may guarantee the existence of ground states.
	
	Concerning graphs with unbounded defects, the discussion is completely open. One could wonder, for instance, if assumption (P) yields ground states existence or if some further condition is required. More generally, it appears quite a tough problem to investigate which condition is necessary to combine with the validity of \eqref{eq:s2d} so to have existence of ground states and what happens to ground states on defected grids that do not even support \eqref{eq:s2d}. In this last case one could reasonably expect that NLSE ground states never exist. However, a general proof is still missing and, given the wide and exotic phenomenology that defected grids proved to exhibit, we do not expect it to be straightforward.
	\begin{figure}[t]
		\centering
		\includegraphics[width=.8\columnwidth]{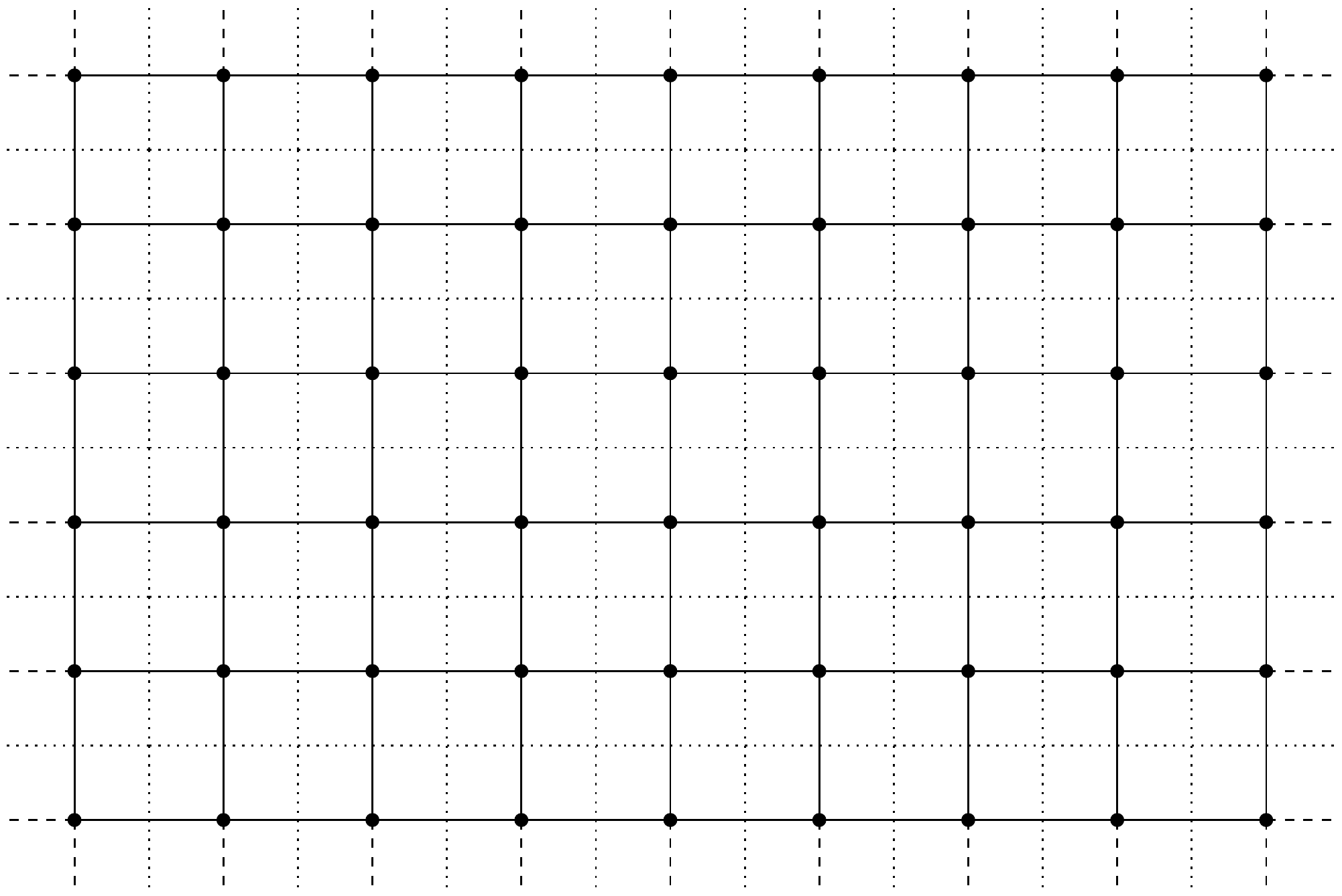}
		\caption{A grid with $\Z^2$--periodic defects that satisfies \eqref{muG<muQ} and \eqref{EG<EQ} (removed edges are dotted).}
		\label{fig-z2comez1}
	\end{figure}



	\section{Two--dimensional Sobolev inequality}
	\label{sec:s2d}
	
	In this section  we present the proofs of the results concerning the validity of the two--dimensional Sobolev inequality. For the sake of simplicity, we organized the section in three parts: the first one addresses the equivalence with the isoperimetric inequality, the second one addresses grids with uniformly bounded defects, the third one addresses grids with unbounded defects.
	
	
	
	\subsection{Equivalence with the isoperimetric inequality}
	\label{subsec:isop}
	
	The aim of this section is to prove the equivalence between the two--dimensional Sobolev inequality \eqref{eq:s2d} and the isoperimetric inequality \eqref{intro:iso G} stated by Theorem \ref{THM:iso=sob}.
	
	\begin{proof}[Proof of Theorem \ref{THM:iso=sob}]
		The proof is completely analogous to that in Euclidean spaces. It is divided in two parts. 
		
		\emph{Part 1: \eqref{eq:s2d} implies \eqref{intro:iso G}.} Let $\Omega$ be a bounded subset of $\G$. If $A_\G(\Omega)=0$ or $P_\G(\Omega)=+\infty$, then \eqref{intro:iso G} is immediate. Thus we assume $A_\G(\Omega)>0$ and $P_\G(\Omega)<+\infty$.
		
		First, denote  $\ov{\Omega}:=\Omega\cup\partial\Omega$ and define $\alpha:=\min\{d_{\G\setminus\Omega}(s,t):s,t\in\partial\Omega,\,s\neq t\}$ and $\beta:=\min\{d_{\G\setminus\Omega}(s,\v):s\in\partial\Omega, \v\text{ is a vertex of }\G\setminus\ov{\Omega}\}$. As $P_\G(\Omega)<+\infty$, the set $\partial\Omega$ is finite and thus we have $\alpha,\,\beta>0$. Now, for every $\eps<\min\{\alpha/2,\beta\}$, set
		\[
		u_\eps(x):=\begin{cases}
		\displaystyle 1 & \text{if }x\in\ov{\Omega}\\
		\displaystyle 1-\f{d_{\G\setminus\Omega}(x,\partial\Omega)}\eps & \text{if }x\in\G\setminus\ov{\Omega}\text{ and }d_{\G\setminus\Omega}(x,\partial\Omega)\leq\eps\\
		\displaystyle 0 & \text{if }x\in\G\setminus\ov{\Omega}\text{ and }d_{\G\setminus\Omega}(x,\partial\Omega)>\eps.
		\end{cases}
		\]
		Clearly $u_\eps\in W^{1,1}(\G)$ and, recalling Definition \ref{def-AP}, 
		\[
		\begin{split}
		\|u_\eps\|_{L^2(\G)}^2=&\int_\Omega dx+\sum_{s\in\partial\Omega}p(s)\int_{0}^{\eps}\left(1-\f{y}{\eps}\right)^2\dy=A_\G(\Omega)+\f\eps 3 P_\G(\Omega)\\
		\|u_\eps'\|_{L^1(\G)}=&\sum_{s\in\partial\Omega}p(s)\int_0^{\eps}\f1\eps\dy=P_\G(\Omega)\,.
		\end{split}
		\]
		As a consequence, in view of \eqref{eq:s2d},
		\[
		\sqrt{A_\G(\Omega)}=\lim_{\eps\to0}\|u_\eps\|_{L^2(\G)}\leq S_\G\lim_{\eps\to0}\|u_\eps'\|_{L^1(\G)}=S_\G P_\G(\Omega)
		\]
		which proves \eqref{intro:iso G}.
		
		\emph{Part 2: \eqref{intro:iso G} implies \eqref{eq:s2d}.} Recall that, without loss of generality, we can always assume $u\geq0$ and by density it is sufficient to prove \eqref{eq:s2d} for nonnegative functions in $C_0^\infty(\G)$. 
		
		Now, let $u\in C_0^\infty(\G)$, $u\geq0$. First, exploiting the coarea formula (see, e.g., \cite[Section 3.2]{F69}), \eqref{eq: P leq H} and \eqref{intro:iso G}, we have
		\begin{multline}
		\label{u'}
		\|u'\|_{L^1(\G)}=\int_0^{+\infty}\mathcal{H}^0\left(\big\{u=t\big\}\right)\dt\\
		\geq\f14\int_0^{+\infty}P_\G\left(\big\{u\geq t\big\}\right)\dt\geq\f{C_\G}4\int_0^{+\infty}\sqrt{A_\G\left(\{u\geq t\}\right)}\dt.
		\end{multline}
		On the other hand, recalling the layer cake representation,
		\begin{equation}
		\label{u^2}
		\|u\|_{L^2(\G)}^2=2\int_{0}^{+\infty}t\,A_\G\left(\big\{u\geq t\big\}\right)\dt.
		\end{equation}
		Furthermore, as $A_\G\left(\big\{u\geq t\big\}\right)$ is a nonincreasing function of $t$,
		\begin{multline*}
		t\,A_\G\left(\big\{u\geq t\big\}\right)\leq\sqrt{A_\G\left(\big\{u\geq t\big\}\right)}\int_0^t\sqrt{A_\G\left(\big\{u\geq s\big\}\right)}\ds\\
		=\f12\f d{dt}\left(\int_0^t\sqrt{A_\G\left(\big\{u\geq s\big\}\right)}\ds\right)^2,\qquad\forall t\in(0,+\infty)
		\end{multline*}
		so that
		\begin{equation}
		\label{eq:A leq A2}
		\int_0^{+\infty}t\,A_\G\left(\big\{u\geq t\big\}\right)\dt\leq C\left(\int_0^{+\infty}\sqrt{A_\G\left(\big\{u\geq t\big\}\right)}\dt\right)^2
		\end{equation}
		for some $C>0$. Then, combining \eqref{u'}, \eqref{u^2} and \eqref{eq:A leq A2}, there results \eqref{eq:s2d}.
	\end{proof} 
	
	
	
	\subsection{Grids with uniformly bounded defects}
	\label{subsec:bound}
	
	The aim of this section is to prove the two--dimensional Sobolev inequality \eqref{eq:s2d} for grids with uniformly bounded defects (Definition \ref{def:unbdD}), as stated in Theorem \ref{THM:s2d_bound}.
	
	\begin{proof}[Proof of Theorem \ref{THM:s2d_bound}]
		Preliminarily, recall that, by Lemma \ref{lem_dDconn}, the boundary $\partial\D_k$ of a defect $\D_k$ is connected. Note also that \eqref{eq:bound_dD} entails both
		\begin{equation}
		\label{eq: bound D}
		\sup_{k\in I}\left|\partial\D_k\right|<+\infty
		\end{equation}
		(in view of Remark \ref{rem:D e dD finiti}) and
		\begin{equation}
		\label{eq:bound_ol}
		\sup_{k\in I} N(k)<+\infty
		\end{equation}
		where $N_k:=\#\{j\in I\,:\,\partial\D_k\cap\partial\D_j\neq\emptyset\}$.
		
		As a first step, for every $k\in I$, fix an edge of $\partial\D_k$ and denote by $t_k$ its middle point. In addition, let
		\begin{equation}
		\label{def:Uk}
		U_k:=\big\{\v\in \VG\cap\partial\D_k\,:\,\deg(\v)\leq3\big\}.
		\end{equation}
		For every $\v\in U_k$, we can choose on any edge $e\in\D_k$ incident at $\v$ a coordinate $x_e$ such that $x_e=0$ corresponds to $\v$ and define
		\begin{equation}
		\label{eq:Lv}
		A_\v:=\bigcup_{\substack{e\in\D_k \\ e\succ\v}}e\cap\left[0,\f12\right]
		\end{equation}
		(namely, the union of the ``first halves'' of every edge in $\D_k$ incident at $\v$). Notice that $\v\in A_\v$. Furthermore, relying on Lemma \ref{lem_dDconn}, for every $\v\in U_k$, let $\gamma_\v\subset\partial\D_k$ be the shortest path in $\partial\D_k$ starting at $\v$ and ending at $t_k$ and let $\lambda_\v:=|\gamma_\v|$.
		
		Consider now a generic function $u\in W^{1,1}(\G)$, $u\geq0$ (as we mentioned before this is not restrictive), and define the function $v:\Q\to\R$ such that
		\[
		v(x):=\begin{cases}
		u(x) & \text{if }x\in\Q\cap\G,\\
		u_{\mid\gamma_\v}(2\lambda_\v x) &\text{if }x\in A_\v,\,\text{for some }\v\in U_k\text{ and some }k\in I,\\
		u(t_k) & \text{if }x\in\D_k\setminus\bigcup_{\v\in U_k}A_\v,\,\text{for some }k\in I.
		\end{cases}
		\]
		By construction, $v$ is continuous on $\Q$ and, for every $p\geq1$,
		\begin{equation}
		\label{eq-efinita?}
		\|v\|_{L^p(\Q)}^p=\|u\|_{L^p(\G)}^p+\sum_{k\in I}a_k u(t_k)^p+\sum_{k\in I}\sum_{\v\in U_k}\left(4-\deg(\v)\right)\int_0^{\f12
		}\left|u_{\mid\gamma_\v}(2\lambda_\v x)\right|^p\,dx\,,
		\end{equation}
		where $a_k:=\left|\D_k\setminus\bigcup_{\v\in U_k}A_\v\right|$. Let us estimate the last two terms in order to prove that $v\in L^p(\Q)$. First, by \eqref{eq:bound_dD},
		\[
		\sup_{k\in I}a_k< +\infty.
		\]
		In addition, for every fixed $j\in I$, as $t_j\in\partial\D_j$, \eqref{eq:bound_ol} entails that the number of possible repetitions of the term $u(t_j)^p$ in the sum $\sum_{k\in I}a_k u(t_k)^p$ is bounded by a constant independent of $j$. Combining these remarks, one sees that there exists $M>0$ such that
		\begin{equation}
		\label{eq:v L2 1}
		\sum_{k\in I}a_k u(t_k)^p\leq M\sum_{k\in I'}u(t_k)^p
		\end{equation}
		where $I'\subset I$ satisfies
		\[
		i,j\in I',\,i\neq j\qquad\Rightarrow\qquad t_i\neq t_j.
		\]
		If $\#I'<+\infty$, then clearly the sum at the right hand side of \eqref{eq:v L2 1} is finite. On the contrary, if $\#I'=+\infty$, then, recalling that for every $k\in I'$ there exists an edge $e_k\in\partial\D_k$ such that $t_k\in e_k$, there results
		\[
		u(t_k)^p\leq\|u\|_{L^\infty(e_k)}^p\leq\|u\|_{W^{1,1}(e_k)}^p.
		\]
		As a consequence,
		\[
		\sum_{k\in I}a_ku(t_k)^p\leq M\sum_{k\in I'}\|u\|_{W^{1,1}(e_k)}^p\leq M \|u\|_{W^{1,1}(\G)}^p<+\infty.
		\]
		Therefore, we are left to estimate the last term of \eqref{eq-efinita?}. To this aim we observe that
		\begin{multline*}
		\sum_{k\in I}\sum_{\v\in U_k}\left(4-\deg(\v)\right)\int_0^{\f12
		}\left|u_{\mid\gamma_\v}(2\lambda_\v x)\right|^p\,dx\\
		\leq \sum_{k\in I}\sum_{\v\in U_k}\f3{2\lambda_\v}\int_{\gamma_\v}|u(x)|^p\,dx
		\leq3\sum_{k\in I}\#U_k\|u\|_{L^p(\partial\D_k)}^p\leq C\sum_{k\in I}\|u\|_{L^p(\partial\D_k)}^p
		\end{multline*}
		(where we used that $\lambda_\v\geq1/2$ and that $\#U_k$ is uniformly bounded by \eqref{eq: bound D}). Then, exploiting again \eqref{eq:bound_ol}, we have that
		\begin{equation}
		\label{eq-auxdopo}
		\sum_{k\in I}\sum_{\v\in U_k}\left(4-\deg(\v)\right)\int_0^{\f12
		}\left|u_{\mid\gamma_\v}(2\lambda_\v x)\right|^p\,dx<+\infty,
		\end{equation}
		and thus $v\in L^p(\Q)$.
		
		On the other hand, relying again on \eqref{eq: bound D} and \eqref{eq:bound_ol} and arguing as before, there results that
		\begin{equation}
		\label{eq:v' L1}
		\begin{split}
		\|v'\|_{L^1(\Q)}=&\|u'\|_{L^1(\G)}+\sum_{k\in I}\sum_{\v\in U_k}\left(4-\deg(\v)\right)\int_0^{\f12}2\lambda_\v|u_{\mid\gamma_\v}'(2\lambda_\v x)|\,dx\\
		\leq&\|u'\|_{L^1(\G)}+3\sum_{k\in I}\#U_k\|u'\|_{L^1(\partial\D_k)}\leq C\|u'\|_{L^1(\G)},
		\end{split}
		\end{equation}
		so that $v\in W^{1,1}(\Q)$. Then, in order to prove \eqref{eq:s2d} it is sufficient to combine \eqref{eq-efinita?}, \eqref{eq:2dGN} and \eqref{eq:v' L1} as follows
		\begin{equation}
		\label{eq:dimsob}
		\|u\|_{L^2(\G)}\leq\|v\|_{L^2(\Q)}\leq\Sq_\Q\|v'\|_{L^1(\Q)}\leq \Sq_\Q C\|u'\|_{L^1(\G)}.
		\end{equation}
	\end{proof}
	
	
	
	\subsection{Grids with unbounded defects}
	\label{subsec:unbd_s2d}
	
	The aims of this section are:
	\begin{itemize}
		\item[--] proving the two--dimensional Sobolev inequality \eqref{eq:s2d} for grids with unbounded defects (Definition \ref{def:nonbdD}) satisfying assumption (P) (Definition \ref{def:assP}), as stated in Theorem \ref{THM:s2d_P};
		\item[--] constructing an example of defected grid that does not satisfy assumption (P), but supports the two--dimensional Sobolev inequality \eqref{eq:s2d}, thus proving Theorem \ref{THM:SobNoP}.
	\end{itemize}
	
	To do this, the first point is showing that the number of unbounded defects of a grid satisfying assumption (P) is finite. 
	
	\begin{proposition}
		\label{prop:fin_ubdD} Let $\G$ be a grid with unbounded defects. If $\#I_\infty=+\infty$, then $\G$ does not satisfy assumption (P).
	\end{proposition}
	
	\begin{proof}
		As $\#I_\infty=+\infty$  we can assume, up to a suitable rigid motion of the coordinate axis of $\R^2$, that there are infinitely many indices $j\in I_\infty$ such that
		\begin{equation}
		\label{eq:y_leq0}
		\inf_{\v\in\partial\D_j^\infty}y_\v=-\infty, \qquad\sup_{\v\in\partial\D_j}y_\v\leq0,
		\end{equation}
		where $y_\v$ denotes the ordinate of the vertex $\v=(x_\v,y_\v)\in\Z^2$. As the argument of the proof involves these unbounded defects only, we can also assume without loss of generality that any unbounded defect $\D_j^\infty$ fulfills \eqref{eq:y_leq0}. Finally, we note that it is sufficient to consider the case $I_{b}=\emptyset$. Indeed, if one proves that a grid with infinitely many unbounded defects only does not satisfy assumption (P), then the addition of bounded defects cannot modify this situation. This follows by the fact that the distance between any couple of points in a graph cannot decrease whenever further edges are removed.
		
		Now, for every $j\in I_\infty$, let $y_{j,0}:=\max_{\v\in\partial\D_j^\infty}y_\v$ and let $\ttt_{j,0}\in\partial\D_j^\infty$ be a fixed point in the interior of some edge on the boundary of $\D_j^\infty$ such that $y_{\ttt_{j,0}}=y_{j,0}$. The existence of such a $\ttt_{j,0}$ is guaranteed by the fact that, if there is a vertex in $\partial\D_j^\infty$ with ordinate $y_{j,0}$ and no other vertex with a higher ordinate, then there exists at least one edge in $\partial\D_j^\infty$ with the same ordinate. Moreover, for every $n\in(-\infty,y_{j,0}]\cap\Z$, let
		\[
		\v_{j,n}^L:=\argmin_{\v\in\partial\D_j^\infty\cap\{y=n\}}x_\v\quad\text{and}\quad\v_{j,n}^R:=\argmax_{\v\in\partial\D_j^\infty\cap\{y=n\}}x_\v\,,
		\]
		which by construction are well--defined for all but at most two unbounded defects.
		We say that a vertex $\v\in\partial\D_j^\infty$ such that $y_\v=n$ belongs to the \emph{left side} of $\partial\D_j^\infty$ if there exists a path $\gamma\subset\partial\D_j^\infty$ from $\v$ to $\v_{j,n}^L$ such that $\ttt_{j,0}\notin\gamma$. Similarly, we say that a vertex $\v\in\partial\D_j^\infty$ belongs to the \emph{right side} of $\partial\D_j^\infty$ if there exists a path $\gamma\subset\partial\D_j^\infty$ from $\v$ to $\v_{j,n}^R$ such that $\ttt_{j,0}\notin\gamma$. Note that the definition of $\ttt_{j,0}$ prevents the possibility for a vertex in $\partial\D_j^\infty$ to belong to both sides. In addition, let $\partial\D_j^{\infty,L}, \partial\D_j^{\infty,R}$ denote the set of all vertices on the left and right side of $\partial\D_j^\infty$, respectively.
		
		Finally, for every $n\leq y_{j,0}$, we introduce the sets
		\[
		\begin{split}
		L_{j,n}:=&\Big\{\v\in\partial\D_j^{\infty,L}\,:\,\deg(\v)\leq3\quad\text{and}\quad y_\v\in\left[n,y_{j,0}\right]\Big\}\\
		R_{j,n}:=&\Big\{\v\in\partial\D_j^{\infty,R}\,:\,\deg(\v)\leq3\quad\text{and}\quad y_\v\in\left[n,y_{j,0}\right]\Big\}
		\end{split}
		\]
		and observe that, by construction (in view of Definition \ref{def:D and dD} and Lemma \ref{lem_dDconn}),
		\begin{equation}
		\label{eq:D_geq_n}
		\begin{split}
		\#L_{j,n}\geq&\max\Big\{\left|n-y_{j,0}\right|,\,\Big|x_{\v_{j,n}^L}-x_{\v_{j,y_{j,0}}^L}\Big|\Big\}\\
		\#R_{j,n}\geq&\max\Big\{\left|n-y_{j,0}\right|,\,\Big|x_{\v_{j,n}^R}-x_{\v_{j,y_{j,0}}^R}\Big|\Big\}.
		\end{split}
		\end{equation}
		Now, for the sake of simplicity, we divide the remainder of the proof in two cases.
		
		\emph{Case 1: there exist two indices $j^*,\,j'\in I_\infty$ and a constant $C>0$ such that}
		\begin{equation}
		\label{eq:2side_bd}
		\left|x_{\v_{j^*,n}^R}-x_{\v_{j',n}^L}\right|\leq C,\qquad\forall n\leq\min\Big\{y_{j^*,0},\,y_{j',0}\Big\}.
		\end{equation}
		
		Here, it is not restrictive to set $y_{j^*,0}\geq y_{j',0}$ and $x_{\v_{j^*,n}^R}\leq x_{\v_{j',n}^L}$, for all $n\leq y_{j',0}$ (the other cases can be managed in the same way). Therefore, letting
		\[
		B_n:=\Big\{\v\in \VG\,:\,x_{\v_{j^*,n}^R}\leq x_\v\leq x_{\v_{j',n}^L},\,y_\v=y_{j',0}\quad\text{or}\quad y_\v=n\Big\}\,,
		\]
		\eqref{eq:2side_bd} yields that $\#B_n\leq 2C$ uniformly on $n$. On the other hand, since by construction every simple path of infinite length starting at a given vertex in $L_{j',n}$ must contain at least one point of $B_n$, by \eqref{eq:D_geq_n} there results that for every family of simple paths $(\gamma_\v)_{\v\in L_{j',n}}$ of infinite length such that $\gamma_\v$ starts at $\v\in L_{j',n}$ there exists one vertex of $B_n$ belonging to at least $\f{\left|n-y_{j',0}\right|}{2C}$ of such paths. In addition, as the degree of this vertex of $B_n$ is at most 4, this entails that at least one edge entering this vertex belongs to at least $\f{\left|n-y_{j',0}\right|}{8C}$ of such paths, so that
		\[
		\sup_{j\in I_\infty}\inf_{\Gamma_j\in P_j}\sup_{\v\in V_j}\F_j(\v,\Gamma_j)\geq\lim_{|n|\to+\infty}\f{\left|n-y_{j',0}\right|}{8C}=+\infty\,,
		\]
		which proves that assumption (P) cannot hold.
		
		\emph{Case 2: there do not exist two indices $j^*,\,j'\in I_\infty$ that satisfy condition \eqref{eq:2side_bd}, for some constant $C>0$.}
		
		As a consequence, there exist infinitely many indices $j\in I_\infty$ such that
		\[
		\text{either}\quad\liminf_{|n|\to+\infty}x_{\v_{j,n}^L}=-\infty\quad\text{or}\quad\limsup_{|n|\to+\infty}x_{\v_{j,n}^R}=+\infty.
		\]
		Let us discuss the second case, i.e.,  there exists $I_\infty^R\subseteq I_\infty$ such that $|I_\infty^R|=+\infty$ and 
		\begin{equation}
		\label{eq:step2}
		\limsup_{|n|\to+\infty}x_{\v_{j,n}^R}=+\infty,\qquad\forall j\in I_\infty^R.
		\end{equation}
		Condition \eqref{eq:step2} entails that, for every $j\in I_\infty^R$, there exists a sequence $(n_{j,k})_{k\in\N}\subset\left(-\infty,y_{j,0}\right]\cap\Z$ such that $n_{j,k}\to-\infty$ as $k\to+\infty$, and
		\begin{equation}
		\label{eq:njk}
		x_{\v_{j,n_{j,k}}^R}=\max_{n\in\left[n_{j,k},y_{j,0}\right]}x_{\v_{j,n}^R},\qquad\forall k\in\N.
		\end{equation}
		Again by \eqref{eq:step2} we have $x_{\v_{j,n_{j,k}}^R}\to+\infty$ as $k\to+\infty$. 
		
		Now, for every fixed $j\in I_\infty^R$, if the horizontal path $\gamma=\{(x,y)\in\Q\,:\,x\geq x_{\ttt_{j,0}},\,y=y_{j,0}\}$ contains points of defects different than $\D_j^\infty$, then let $j'\in I_\infty^R$ such that the point of $\gamma$ closest to $\D_j^\infty$ and belonging to another defect belongs to $\D_{j'}^\infty$. If, on the contrary, $\gamma$ contains no point of other defects than $\D_j^\infty$, then we take $j'\in I_\infty^R$ such that $\partial\D_{j'}^\infty$ verifies $y_{j',0}=\max\Big\{y_{i,0}\,:\,i\in I_\infty^R,\,\exists n_i\text{ such that }x_{\v_{i,n}^L}\geq x_{\v_{j,n}^R}\,\forall n<n_i\Big\}$. In the former case we define $D_j$ as the portion of $\gamma$ between the closest points of $\partial\D_j^\infty$ and $\partial\D_{j'}^\infty$. By construction, $D_j\subset\{(x,y)\in\G\,:\,y=y_{j,0}\}$. Moreover, for every $k\in\N$, we set $B_{j,k}:=\Big\{(x,y)\in\G\,:\,x=x_{\v_{j,n_{j,k}}^R},\,n_{j,k}\leq y\leq y_{j,0}\Big\}$. In the latter case we define $D_j$ as the horizontal path in $\G$ joining the closest points of $\partial\D_j^\infty$ and $\partial\D_{j'}^\infty$ at $y=y_{j',0}<y_{j,0}$. Moreover, for every $k\in\N$, we set $B_{j,k}:=\Big\{(x,y)\in\G\,:\,x=x_{\v_{j,n_{j,k}}^R}, n_{j,k}\leq y\leq y_{j',0}\Big\}$ (see Figure \ref{fig:BDT_jk}). Finally, in both cases, we define $d_j:=\#(\VG\cap D_j)$ and $b_{j,k}:=\#(\VG\cap B_{j,k})$. By construction, $b_{j,k}\leq |n_{j,k}|$.
		\begin{figure}[t]
			\centering
			\includegraphics[width=0.8\textwidth]{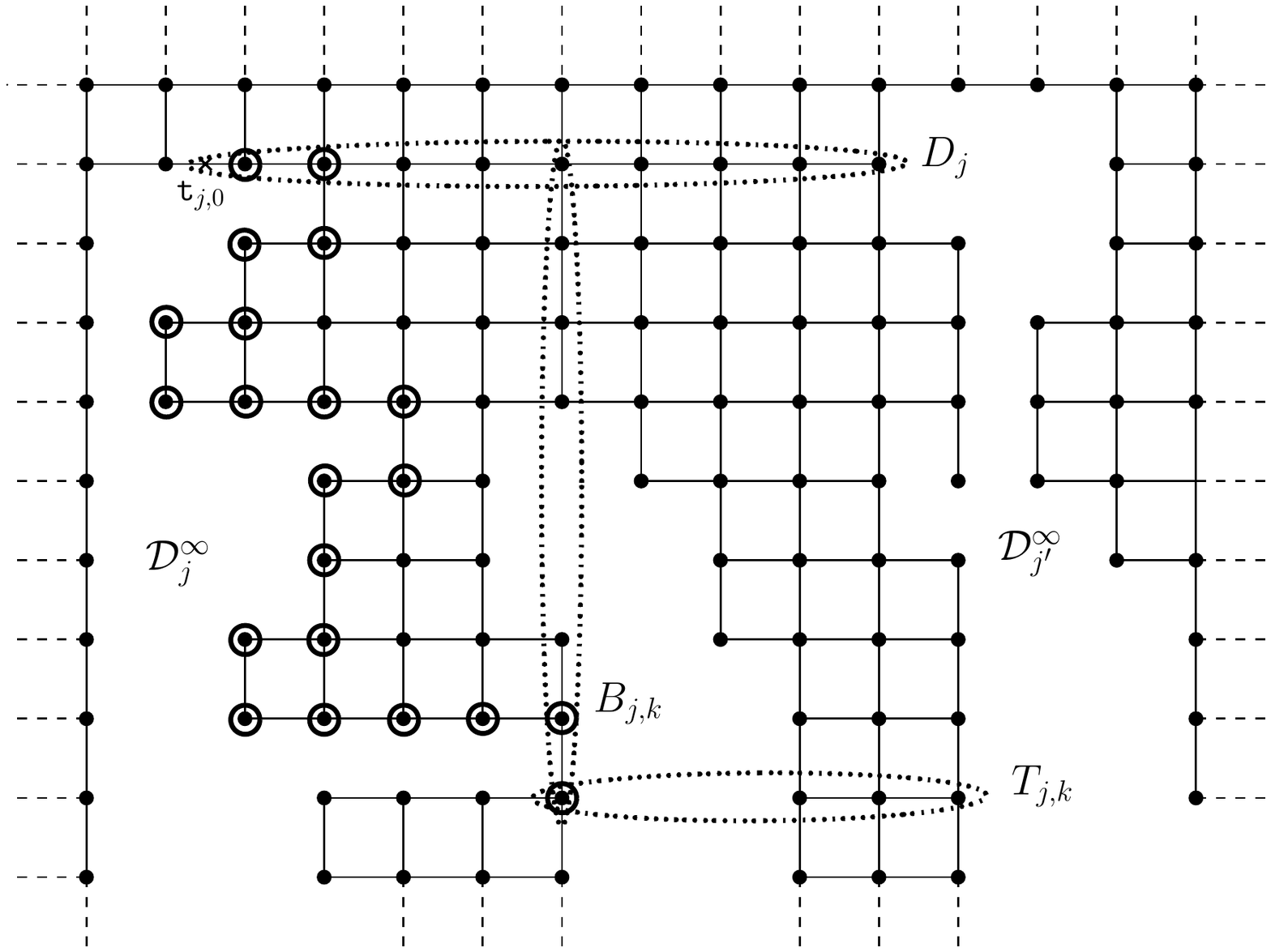}
			\caption{Example of the sets $D_j$, $B_{j,k}$ and $T_{j,k}$ (inside the dotted regions) as in the proof of Proposition \ref{prop:fin_ubdD} (in the case $y_{j,0}\leq y_{j',0}$). The vertices of $R_{j,n_{j,k}}$ are those inside the bold circles. Note that in the picture there is a further unbounded defect between $\D_j^\infty$ and $\D_{j'}^\infty$, but this does not affect the definition of $D_j$, $B_{j,k}$ and $T_{j,k}$.}
			\label{fig:BDT_jk}
		\end{figure}
		
		At this point, one can see that, for every $j\in I_\infty^R$ and $k\in\N$, \eqref{eq:njk} guarantees that any simple path of infinite length in $\G$ starting at any vertex $\v\in R_{j,n_{j,k}}$ (possibly, up to a finite number independent of $k$ and depending only on $j'$) must contain at least one vertex of $D_j\cup B_{j,k}$. Since the degree of each vertex is at most 4, this means that there exists at least one edge of $\G$ that belongs to at least $\f{\#R_{j,n_{j,k}}}{4(d_j+b_{j,k})}$ of such paths. As a consequence,
		\begin{equation}
		\label{eq:bound_P}
		\sup_{j\in I_\infty}\inf_{\Gamma_j\in P_j}\sup_{\v\in V_j} \F_j(\v,\Gamma_j)\geq\sup_{j\in I_\infty^R}\sup_{k\in\N}\f{\#R_{j,n_{j,k}}}{4(d_j+b_{j,k})}.
		\end{equation}
		Now, if there exists $j\in I_\infty^R$ such that $\sup_{k\in\N}b_{j,k}<+\infty$, then \eqref{eq:bound_P} implies that assumption (P) cannot hold, since $\#R_{j,n_{j,k}}\to+\infty$ as $k\to+\infty$ by \eqref{eq:D_geq_n}, whilst $d_j+b_{j,k}$ is uniformly bounded with respect to $k$. If, on the contrary, $b_{j,k}\to+\infty$ as $k\to+\infty$ for every $j\in I_\infty^R$, then $d_j=o(b_{j,k})$ as $k$ is large enough. Hence, if
		\[
		\sup_{j\in I_\infty^R}\sup_{k\in\N}\f{\#R_{j,n_{j,k}}}{b_{j,k}}=+\infty,
		\] 
		then again assumption (P) is false. Suppose by contradiction that there exists $C>0$ independent on $j$ and $k$ such that
		\begin{equation}
		\label{eq:boundRb}
		\#R_{j,n_{j,k}}\leq C b_{j,k},\qquad\forall j\in I_\infty^R,\quad \forall k\in\N.
		\end{equation}
		Consider the set $T_{j,k}:=\Big\{(x,y)\in\G\,:x_{\v_{j,n_{j,k}}^R}\leq x\leq x_{\v_{j',n_{j,k}}^L},\, y=n_{j,k}\Big\}$ (see again Figure \ref{fig:BDT_jk}) and define $t_{j,k}:=\#(\VG\cap T_{j,k})$. Clearly, 
		\begin{equation}
		\label{eq:tjk}
		t_{j,k}\leq x_{\v_{j',n_{j,k}}^L}-x_{\v_{j,n_{j,k}}^R}.
		\end{equation}
		Moreover, observe that every simple path of infinite length in $\G$ starting at a vertex $\v\in R_{j,n_{j,k}}$ must contain at least one vertex of $T_{j,k}\cup D_j$, so that
		\begin{equation}
		\label{eq:bound_Pt}
		\sup_{j\in I_\infty}\inf_{\Gamma_j\in P_j}\sup_{\v\in V_j} \F_j(\v,\Gamma_j)\geq\sup_{j\in I_\infty^R}\sup_{k\in\N}\f{\#R_{j,n_{j,k}}}{4(d_j+t_{j,k})}.
		\end{equation}
		If there exists $j\in I_{\infty}^R$ such that $\sup_{k\in\N}t_{j,k}\leq +\infty$, then \eqref{eq:bound_Pt} implies that assumption (P) does not hold. If, on the contrary, $t_{j,k}\to+\infty$, as $k\to+\infty$, for every $j\in I_\infty^R$, then $d_j=o(t_{j,k})$ as $k$ is large enough. Furthermore, by \eqref{eq:tjk} we have
		\[
		\f{\#R_{j,n_{j,k}}}{t_{j,k}}\geq\f{\#R_{j,n_{j,k}}}{x_{\v_{j',n_{j,k}}^L}-x_{\v_{j,n_{j,k}}^R}}\geq\f{\#R_{j,n_{j,k}}}{x_{\v_{j',n_{j,k}}^R}-x_{\v_{j,n_{j,k}}^R}}\,.
		\]
		Hence, letting $\Delta_{j,k}:=x_{\v_{j',n_{j,k}}^R}-x_{\v_{j,n_{j,k}}^R}$, it is left to prove that
		\begin{equation}
		\label{eq:R/D_infty}
		\sup_{i\in I_\infty^R}\sup_{k\in\N}\f{\#R_{j,n_{j,k}}}{\Delta_{j,k}}=+\infty.
		\end{equation}
		To this aim, observe first that, multiplying and dividing by $|n_{j,k}|$, making use of \eqref{eq:D_geq_n} and \eqref{eq:boundRb} and recalling that $b_{j,k}\leq |n_{j,k}|$, the validity of \eqref{eq:R/D_infty} is granted by the validity of
		\[
		\inf_{j\in I_\infty^R}\inf_{k\in\N}\f{\Delta_{j,k}}{|n_{j,k}|}=0.
		\]
		Hence, assume by contradiction that there exists $\underline{\alpha}>0$ such that
		\begin{equation}
		\label{eq:assurdo}
		\f{\Delta_{j,k}}{|n_{j,k}|}\geq\underline{\alpha},\qquad\forall j\in I_\infty^R,\quad\forall k\in\N,
		\end{equation}
		whence
		\begin{equation}
		\label{eq:alow}
		x_{\v_{j',n_{j,k}}^R}\geq\Delta_{j,k}\geq\underline{\alpha}|n_{j,k}|,\qquad\forall j\in I_\infty^R,\quad\forall k\in\N.
		\end{equation}
		On the other hand, for every $j\in I_\infty^R$, there exists $\overline{k}_j\in\N$ such that
		\begin{equation}
		\label{eq:aup}
		x_{\v_{j',n_{j,k}}^R}\leq \overline{\alpha}|n_{j,k}|,\qquad\forall k\geq\overline{k}_j,
		\end{equation}
		for some $\overline{\alpha}>0$ independent of $j$.
		Indeed, fixing $j\in I_\infty^R$, by \eqref{eq:boundRb} there results that, for every $k, k'\in\N$ such that $n_{j,k}\leq n_{j',k'}\leq0 $,
		\[
		\#R_{j',n_{j',k'}}\leq C b_{j',k'}\leq C|n_{j,k}|,
		\]
		that, coupled with \eqref{eq:D_geq_n} and suitably choosing $n_{j',k'}$, yields
		\[
		\begin{split}
		C|n_{j,k}|\geq\max_{\substack{k'\in\N\,:\, \\ n_{j',k'}\leq n_{j,k}}}\#R_{j',n_{j',k'}}\geq&\max_{\substack{k'\in\N\,:\, \\ n_{j',k'}\leq n_{j,k}}}\left|x_{\v_{j',n_{j',k'}}^R}-x_{\v_{j',y_{j',0}}^R}\right|\\
		\geq&\left|x_{\v_{j',n_{j,k}}^R}-x_{\v_{j',y_{j',0}}^R}\right|= \left|x_{\v_{j',n_{j,k}}^R}\right|\left|1-\f{x_{\v_{j',y_{j',0}}^R}}{x_{\v_{j',n_{j,k}}^R}}\right|.
		\end{split}
		\]
		Since, fixing $j$, by \eqref{eq:njk} we obtain that $\f{x_{\v_{j',y_{j',0}}^R}}{x_{\v_{j',n_{j,k}}^R}}\to0$ as $k\to+\infty$, we find that, for every $j\in I_\infty^R$, there exists $\overline{k}_j\in\N$ such that choosing $\overline{\alpha}=C/2$ there results \eqref{eq:aup}.
		
		Summing up, for every $j\in I_\infty^R$, \eqref{eq:alow}--\eqref{eq:aup} entail that up to subsequence
		\[
		\f{x_{\v_{j',n_{j,k}}^R}}{|n_{j,k}|}\to \alpha_j\qquad\text{as }k\to+\infty,
		\]
		with $\alpha_j\in\left[\underline{\alpha},\overline{\alpha}\right]$ only dependng on $j$. Thus, passing again to a suitable subsequence of $j\in I_\infty^R$, we have $\alpha_j\to\alpha$, as $j\to+\infty$, for some $\alpha\in\left[\underline{\alpha},\overline{\alpha}\right]$. Furthermore, since $j\to+\infty$, for $j$ sufficiently large there exists another index $j_{-}$ such that $x_{\v_{j_{-}',n}^R}\leq x_{\v_{j,n}^R}$ for every $n$ large enough. Also, we can take $j_{-}\to+\infty$ as $j\to+\infty$. Therefore,
		\[
		\f{\Delta_{j,k}}{|n_{j,k}|}\leq\f{x_{\v_{j',n_{j,k}}^R}-x_{\v_{j_{-}',n_{j,k}}^R}}{|n_{j,k}|}=\alpha_j-\alpha_{j_{-}}+o(1)
		\]
		in the limit for $k\to+\infty$. As $\alpha_j-\alpha_{j_{-}}\to0$ when $j\to+\infty$, this contradicts \eqref{eq:assurdo} and proves that assumption (P) cannot be fulfilled whenever \eqref{eq:step2} holds. 
		
		Finally, one should address the other case, in which there exists $I_\infty^L\subseteq I_\infty$ such that $|I_\infty^L|=+\infty$ and 
		\[
		\liminf_{|n|\to+\infty}x_{\v_{j,n}^L}=-\infty,\qquad\forall j\in I_\infty^R.
		\]
		However, this is completely analogous to the previous one.
	\end{proof}
	\begin{remark}
		We point out two further remarks on Proposition \ref{prop:fin_ubdD}. The former is that when estimating the superpositions in the previous proof we avoided using integer parts not to weigh down the notation. The latter is that here condition (ii) of Definition \ref{def:nonbdD} plays no role. As a consequence, the result of Proposition \ref{prop:fin_ubdD} holds whenever the unbounded defects are infinitely many, independently of the bounded ones being uniformly bounded or not.
	\end{remark}
	We can now prove Theorem \ref{THM:s2d_P}. The argument is similar to that of the proof of Theorem \ref{THM:s2d_bound}. 
	\begin{proof}[Proof of Theorem \ref{THM:s2d_P}]
		Note that, as $\G$ satisfies assumption (P), for every $j\in I_\infty$ there exists $\Gamma_j\in P_j$ such that
		\begin{equation}
		\label{eq:GjleqC}
		\sup_{\v\in V_j} \F_j(\v,\Gamma_j)\leq C,
		\end{equation}
		for some $C>0$ independent of $j$. Moreover, by Proposition \ref{prop:fin_ubdD}, $\#I_\infty=N<+\infty$.
		
		Now, for every $k\in I_{b}$, let $U_k$ and $t_k$ be as in the proof of Theorem \ref{THM:s2d_bound} (see \eqref{def:Uk} and the line before). Furthermore, for every $\v\in U_k$, let $A_\v$ be as in \eqref{eq:Lv}, $\gamma_\v$ be the shortest simple path in $\partial\D_k^b$ starting at $\v$ and ending at $t_k$ and set $\lambda_\v:=|\gamma_\v|$. Similarly, choosing on any edge $e\in \D_j^\infty$ incident at a vertex $\v$ a coordinate $x_e$ such that $x_e=0$ corresponds to $\v$, for every $j\in I_\infty$ and every $\v\in V_j$ define
		\[
		\begin{split}
		A_\v^\infty:=&\bigcup_{\substack{e\in\D_j^\infty \\ e\succ\v}}e\cap\left[0,\f12\right],\\
		B_\v^\infty:=&\bigcup_{\substack{e\in\D_j^\infty \\ e\succ\v}}e\setminus A_\v.
		\end{split}
		\]
		Hence, $\v\in A_\v^\infty$ and each edge $e\in\D_j^\infty$ with one endpoint in $V_j$ is such that half of it belongs to $A_\v^\infty$ and the other half is in $B_\v^\infty$.
		
		Let then $u\in W^{1,1}(\G)$, $u\geq0$ (as we mentioned, this is not restrictive), fix a parameter $\lambda>1$ and define the function $v_\lambda:\Q\to\R$ such that
		\[
		v_\lambda(x):=\begin{cases}
		u(x) & \text{if }x\in\Q\cap\G\\
		u_{\mid\gamma_\v}(2\lambda_\v x) &\text{if }x\in A_\v,\,\text{for some }\v\in U_k\text{ and some }k\in I_{b}\\
		u(t_k) & \text{if }x\in\D^b_k\setminus\bigcup_{\v\in U_k}A_\v,\,\text{for some }k\in I_{b}\\
		u_{\mid\gamma_\v}(2\lambda x) &\text{if }x\in A_\v^\infty,\,\text{for some }\v\in V_j\text{ and some }j\in I_\infty\\
		2u_{\mid\gamma_\v}(\lambda)(1-x) &\text{if }x\in B_\v^\infty,\,\text{for some }\v\in V_j\text{ and some }j\in I_\infty\\
		0 & \text{if }x\in\D_j^\infty\setminus\bigcup_{\v\in V_j}\left(A_\v^\infty\cup B_\v^\infty\right),\text{ for some }j\in I_\infty\,.
		\end{cases}
		\]
		By construction, $v_\lambda$ is continuous on $\Q$ and, for every $p\geq1$,
		\[
		\begin{split}
		\|v_\lambda\|_{L^p(\Q)}^p=&\|u\|_{L^p(\G)}^p+\sum_{k\in I_{b}}a_k u(t_k)^p+\sum_{k\in I_{b}}\sum_{\v\in U_k}\left(4-\deg(\v)\right)\int_0^{\f12
		}\left|u_{\mid\gamma_\v}(2\lambda_\v x)\right|^p\,dx\\
		+&\sum_{j\in I_\infty}\sum_{\v\in V_j}\left(4-\deg(\v)\right)\int_{0}^{\f12}|u_{\mid\gamma_\v}(2\lambda x)|^p\,dx+\sum_{j\in I_\infty}\sum_{\v\in V_j}\left(4-\deg(\v)\right)\f{u_{\mid\gamma_\v}^p(\lambda)}{2(p+1)},
		\end{split}
		\]
		with $a_k:=\left|\D^b_k\setminus\bigcup_{\v\in V_k}A_\v\right|$. Recalling the calculations in the proof of Theorem \ref{THM:s2d_bound} (see \eqref{eq-efinita?}--\eqref{eq-auxdopo}), in order to show that $v_\lambda\in L^p(\Q)$ it is sufficient to prove that
		\[
		\sum_{j\in I_\infty}\sum_{\v\in V_j}\left(4-\deg(\v)\right)\int_{0}^{\f12}|u_{\mid\gamma_\v}(2\lambda x)|^p\,dx+\sum_{j\in I_\infty}\sum_{\v\in V_j}\left(4-\deg(\v)\right)\f{u_{\mid\gamma_\v}^p(\lambda)}{2(p+1)}<+\infty.
		\]
		On the one hand, by \eqref{eq:GjleqC} and $\#I_\infty=N$, there results
		\[
		\begin{split}
		\sum_{j\in I_\infty}\sum_{\v\in V_j}\left(4-\deg(\v)\right)&\int_{0}^{\f12}|u_{\mid\gamma_\v}(2\lambda x)|^p\,dx\leq \f{3N}{2\lambda}\sup_{j\in I_\infty}\sum_{\v\in V_j}\int_{\gamma_\v\cap[0,\lambda]}|u(x)|^p\,dx\\
		\leq&\f{3NC}{2\lambda}\sup_{j\in I_\infty}\int_{\Gamma_j\cap \EG}|u(x)|^p\,dx\leq\f{3NC}{2\lambda}\|u\|_{L^p(\G)}^p\,,
		\end{split}
		\]
		(observe that $\Gamma_j\cap \EG$ is the set of all edges of $\G$ visited by at least one path in $\Gamma_j$). On the other hand, for every $\v\in V_j$ there exists an edge $e_\v\in\gamma_\v$ such that $u_{\mid\gamma_\v}(\lambda)$ is attained somewhere in $e_\v$. By \eqref{eq:GjleqC}, the number of couples $\v,\,\w\in V_j$ with $\v\neq\w$ and so that it is not possible to choose $e_\v\neq e_\w$ is bounded uniformly with respect to $j$ and $\lambda$. Hence, there exists $M>0$, independent of $j$ and $\lambda$, such that
		\[
		\sum_{\v\in V_j}u_{\mid\gamma_\v}^p(\lambda)\leq M \sum_{\v\in V_j'}u_{\mid\gamma_\v}^p(\lambda)\,,
		\]
		where a vertex $\v\in V_j$ belongs to $V_j'$ if and only if $e_\v\neq e_\w$, for any $\w\in V_j'$, $\w\neq\v$. As a consequence,
		\[
		\begin{split}
		\sum_{j\in I_\infty}\sum_{\v\in V_j}\left(4-\deg(\v)\right)\f{u_{\mid\gamma_\v}^p(\lambda)}{2(p+1)}\leq&\f{3NM}{2(p+1)}\sup_{j\in I_\infty}\sum_{\v\in V_j'}u_{\mid\gamma_\v}^p(\lambda)\\
		\leq&\f{3NM}{2(p+1)}\sup_{j\in I_\infty}\sum_{\v\in V_j'}\|u\|_{W^{1,1}(e_\v)}^p\leq \f{3NM}{2(p+1)}\|u\|_{W^{1,1}(\G)}^p,
		\end{split}
		\]
		and thus $v_\lambda\in L^p(\Q)$.
		
		We are then left to prove that $v_\lambda'\in L^1(\Q)$. Note that
		\begin{multline}
		\label{eq:vlambda'}
		\|v_\lambda'\|_{L^1(\Q)}=\|u'\|_{L^1(\G)}+\sum_{k\in I}\sum_{\v\in V_k}\left(4-\deg(\v)\right)\int_0^{\f12}2\lambda_\v|u_{\mid\gamma_\v}'(2\lambda_\v x)|\,dx\\
		+\sum_{j\in I_\infty}\sum_{\v\in V_j}\left(4-\deg(\v)\right)\int_0^{\f12}2\lambda|u_{\mid\gamma_\v}'(2\lambda x)|\,dx+2\sum_{j\in I_\infty}\sum_{\v\in V_j}\left(4-\deg(\v)\right)u_{\mid\gamma_\v}(\lambda).
		\end{multline}
		As the second term on the right hand side can be managed again as in the proof of Theorem \ref{THM:s2d_bound} (see \eqref{eq:v' L1}), it is sufficient to estimate the last two terms. On the one hand, using \eqref{eq:GjleqC} and arguing as before, there results
		\begin{multline}
		\label{eq:v'infty1}
		\sum_{j\in I_\infty}\sum_{\v\in V_j}\left(4-\deg(\v)\right)\int_0^{\f12}2\lambda|u_{\mid\gamma_\v}'(2\lambda x)|\,dx\\
		\leq 3NC\sup_{j\in I_\infty}\int_{\Gamma_j\cap \EG}|u'|\,dx\leq 3NC\|u'\|_{L^1(\G)},
		\end{multline}
		for some $C$ independent of $\lambda$. On the other hand, let $B_R:=\{(x,y)\in\G\,:\,|x|\leq R,\,|y|\leq R\}$. It is clear that, for every $R>0$, there exists $\lambda_R>1$ such that one can choose as before $e_\v\notin B_R$, for every $\v\in V_j$, every $j\in I_\infty$ and every $\lambda>\lambda_R$. Since $\|u\|_{W^{1,1}}(\G\setminus B_R)\to0$ as $R\to+\infty$, taking $R$ and $\lambda$ sufficiently large, we obtain
		\begin{multline}
		\label{eq:v'infty1bis}
		2\sum_{j\in I_\infty}\sum_{\v\in V_j}\left(4-\deg(\v)\right)u_{\mid\gamma_\v}(\lambda)\leq 6 N \sup_{j\in I_\infty}\sum_{\v\in V_j}u_{\mid\gamma_\v}(\lambda)\\
		\leq 6NM\sup_{j\in I_\infty} \sum_{\v\in V_j'}\|u\|_{W^{1,1}(e_\v)}\leq 6NM\|u\|_{W^{1,1}(\G\setminus B_R)}\leq 6NM\|u'\|_{L^1(\G)},
		\end{multline}
		which proves that $v_\lambda'\in L^1(\G)$ for any large $\lambda$.
		
		To prove \eqref{eq:s2d} one can then argue as in \eqref{eq:dimsob} exploiting \eqref{eq:vlambda'}, \eqref{eq:v' L1}, \eqref{eq:v'infty1} and \eqref{eq:v'infty1bis} for a large enough choice of the parameter $\lambda$.
	\end{proof}
	
	Finally, we present the proof of Theorem \ref{THM:SobNoP}, which is based on the construction of an explicit counterexample.
	
	\begin{proof}[Proof of Theorem \ref{THM:SobNoP}]
		To prove the statement it is enough to exhibit a defected grid that violates assumption (P), but supports \eqref{eq:s2d}. As we mentioned in Section \ref{subsec:sobineq}, a suitable grid is the one depicted in Figure \ref{fig:s2d_noP}. Let us divide the proof of this fact in three steps.
		
		\emph{Step (i): description of the defected grid $\G$ in Figure \ref{fig:s2d_noP}}. First, let $\Gamma_0:=\bigcup_{i=0}^3\{0\}\times[i,i+1]$, that is the vertical path of length $4$ from $(0,0)$ to $(0,4)$ in $\Q$.
		
		In addition, for every $i\in\N$, let $\Gamma_i^-:=H_i^-\cup V_i\cup H_i^+$, where
		\[
		\begin{split}
		H_i^-:=&\bigcup_{j=0}^i\left[i(i+1)+j,i(i+1)+j+1\right]\times\{0\}\\
		V_i:=&\left(\bigcup_{j=0}^i\{(i+1)^2\}\times[j,j+1]\right)\cup\left(\bigcup_{j=0}^i\{(i+1)(i+2)\}\times[j,j+1]\right)\\
		H_i^+:=&\bigcup_{j=0}^i\left[(i+1)^2+j,(i+1)^2+j+1\right]\times\{i+1\}\,.
		\end{split}
		\]
		One can check that $\Gamma_i^-$ is connected and that, up to the lower horizontal edge, $V_i\cup H_i^+$ is the boundary of a square with edge of length $i+1$. Set then $\Gamma_-:=\bigcup_{i\in\N}\Gamma_i^-$. By construction, it is connected, $(0,0)\in\Gamma_-$ and $x\geq0$ for every $(x,y)\in\Gamma_-$.
		
		On the other hand, let
		\[
		\Gamma_i^+:=\left(\bigcup_{j=0}^{2i+1}[i(i+1)+j,i(i+1)+j+1]\times\{4+i\}\right)\cup\Big(\{(i+1)(i+2)\}\times[4+i,5+i]\Big)
		\] 
		and set $\Gamma_+:=\bigcup_{i\in\N}\Gamma_i^+$. By construction, it is connected, $(0,4)\in\Gamma_+$ and $x\geq0,\,y\geq4$, for every $(x,y)\in\Gamma_+$. Furthermore, for every $x\geq0$, letting $y_-, y_+\geq0$ be such that $(x,y_-)\in\Gamma_-$ and $(x,y_+)\in\Gamma_+$, there results $y_+>y_-$. 
		
		Finally, set $\Gamma:=\Gamma_0\cup\Gamma_-\cup\Gamma_+$. As a consequence, $\Gamma$ is a simple path of infinite length dividing $\Q$ in two separate regions. Between these regions, let $\D$ be the one containing the point $(1,2)$ and define $\G=\Q\setminus\D$. Note that, by construction, $\partial\D=\Gamma$. 
		
		\emph{Step (ii): $\G$ supports \eqref{eq:s2d}.} In view of Theorem \ref{THM:iso=sob}, it is sufficient to prove that $\G$ supports \eqref{intro:iso G}.
		
		First, we observe that this is immediate for every bounded $\Omega\subset\G\setminus\partial\D$, given the validity of the isoperimetric inequality on $\Q$ and the fact that $A_\G(\Omega)=A_\Q(\Omega)$ and $P_\G(\Omega)=P_\Q(\Omega)$, for every $\Omega\subset\G\setminus\partial\D$.
		
		On the other hand, note that, for every couple of distinct points $s,t\in\partial\D$, the ratio between the length of the shortest path from $s$ to $t$ in $\partial\D$ and the length of the shortest path from $s$ to $t$ in $\G$ is bounded from above. Indeed, the two paths coincide for every $s,t\in\Gamma_0\cup\Gamma_+$, whereas the length of the shortest path in $\partial\D$ is at most three times the length of the shortest path in $\G$ whenever at least one between $s$ and $t$ belongs to $\Gamma_-$. Hence, there exists a universal constant $C>0$ such that $P_\Q(\Omega)\leq C P_\G(\Omega)$, for every bounded $\Omega\subset\G$ such that $\Omega\cap\partial\D\neq\emptyset$. Given that $A_\G(\Omega)=A_\Q(\Omega)$, the isoperimetric inequality on $\Q$ gives 
		\[
		\sqrt{A(\Omega)}\leq C_\Q P_\Q(\Omega)\leq C C_\Q P_\G(\Omega)
		\]
		for every bounded $\Omega\subset\G$ such that $\Omega\cap\partial\D\neq\emptyset$.
		
		\emph{Step (iii): $\G$ does not satisfy assumption (P).} Note first that, for every $i\in\N$, any simple path of infinite length starting at a vertex in $V_i\cup H_i^+$ must contain at least one vertex of the horizontal path $B_i:=\bigcup_{j=0}^i\left[(i+1)^2+j,(i+1)^2+j+1\right]\times\{0\}$. Furthermore, to reach $B_i$, any path starting from a vertex $\v=(x_\v,y_\v)\in V_i\cup H_i^+$ must run through at least $y_\v$ vertical edges contained in the region bounded by $V_i\cup H_i^+\cup B_i$. Note also that there are $i+2$ vertices in $H_i^+$, all satisfying $y_\v=i+1$, and there are exactly 2 vertices in $V_i$ such that $y_\v=j$, for every $j\in\{1,\,\dots,\,i\}$.
		
		Now, the total number of vertical edges that are necessary for any given family $\Gamma$ of simple paths of infinite length, each of which starts at a different vertex in $V_i\cup H_i^+$ and contains at least one vertex of $B_i$, is greater than or equal to $(i+1)(i+2)+2\sum_{j=1}^i j=2(i+1)^2$. Conversely, the region in $\G$ bounded by $V_i\cup H_i^+\cup B_i$ contains $(i+1)(i+2)$ distinct vertical edges. Therefore, the total number of repetitions of vertical edges must be at least $2(i+1)^2-(i+1)(i+2)=i(i+1)$. Since the total number of vertices in $V_i\cup H_i^+$ is $3i+2$, it follows that the mean number of repetitions of vertical edges per path is at least $\f{i(i+1)}{3i+2}$. Denoting by $P$ the set of all possible families $\Gamma$, we thus obtain
		\[
		\inf_{\Gamma\in P}\sup_{\substack{\v\in \partial\D \\ \deg(\v)\leq3}} \F(\v,\Gamma)\geq\inf_{\Gamma\in P}\sup_{\v \in V_i\cup H_i^+}\F(\v,\Gamma)\geq \f{i(i+1)}{3i+2}\to+\infty\qquad\text{as }i\to+\infty,
		\]
		which concludes the proof.
	\end{proof}
	
		
		\section{NLSE ground states}
		\label{sec:gs}
		
		In this section we present the proofs of the results concerning the NLSE ground states. For the sake of simplicity, we organized the section in four subsections consisting of
		\begin{itemize}
			\item[(i)] some preliminary results,
			\item[(ii)] the discussion of the ground states on compactly defected grids and grids with $\Z$--periodic defects (i.e. the proof of the corresponding part of Theorem \ref{THM:ex_com&Z}),
			\item[(iii)] the discussion of the ground states on grids with $\Z^2$--periodic defects (i.e. the proof of the corresponding part of Theorem \ref{THM:ex_com&Z}),
			\item[(iv)] the construction of grids with uniformly bounded defects that support no ground state (i.e. the proof of Theorem \ref{thm:noground}).
		\end{itemize}
		
		
		
		\subsection{Preliminary results} 
		\label{subsec:gs_on_Q}
		
		Here we develop some tools necessary to prove Theorems \ref{THM:ex_com&Z}--\ref{thm:noground}. More precisely, first we establish an a priori estimate that must be satisfied by any ground state of mass $\mu$, then we show some consequences of the simultaneous validity of the one and the two--dimensional Sobolev inequalities and finally we report a dichotomy result for the minimizing sequences of the energy functional.
		
		
		
		\subsubsection{A priori estimates on the ground states}
		
		Let us begin with recalling some well known features of ground states, that hold for every $p>2$ and every $\mu>0$.
		
		First, whenever $u=(u_e)_{e\in\EG}$ is a ground state of mass $\mu$, by the Lagrange multipliers Theorem, there exists $\lambda\in\R$ such that
		\begin{equation}
		\label{eq-critical}
		\int_\G u'v'\dx-\int_\G|u|^{p-2}uv\dx=-\lambda\int_\G uv\dx,\qquad\forall v\in H^1(\G).
		\end{equation}
		As a consequence, classical arguments (see, e.g., \cite{AST15,T16}) show that $u$ satisfies
		\begin{gather}
		\label{eq-NLSgrafo}\displaystyle u_e''+|u_e|^{p-2}u_e=\lambda u_e,\qquad\forall e\in\EG,\\[.2cm]
		\displaystyle u_e\in H^2(e)\cap C^2(e),\qquad\forall e\in\EG,\\[.2cm]
		\label{eq-sign}\displaystyle u>0,\quad\text{up to a sign changing},\\[.2cm]
		\label{eq-kirch}\displaystyle \sum_{e\succ\v}\f{du_e}{dx_e}(\v)=0,\qquad\forall\v\in\VG,
		\end{gather}
		where $\f{du_e}{dx_e}(\v)$ denotes the derivative of $u_e$ at $\v$ with the edge $e$ parametrized in such a way that $\v$ is represented by $x_e=0$. On the other hand, as one can check that the bottom of the spectrum of the operator $-\Delta_\G:L^2(\G)\to L^2(\G)$ defined by
		\[
		\begin{array}{c}
		\displaystyle (-\Delta_\G u)_{\mid e}:=-u_e'',\\[.2cm]
		\displaystyle \mathrm{dom}(-\Delta_\G):=\{u\in H^1(\G):u_e\in H^2(e),\forall e \in\EG, \text{ and } u \text{ satisfies \eqref{eq-kirch} },\forall\v\in\VG\},
		\end{array}
		\]
		is zero, there results that
		\begin{equation}
		\label{eq:infnonpos}
		\EE_{p,\G}(\mu)\leq0,\qquad\forall p>2,\quad\forall\mu>0.
		\end{equation}
		As a consequence, combining with \eqref{eq-critical} and $p>2$, if $u$ is a ground state of mass $\mu$, then
		\begin{equation}
		\label{eq-lambdapos}
		\lambda=\lambda(u)=\f{\|u\|_{L^p(\G)}^p-\|u'\|_{L^2(\G)}^2}\mu=-\f{2\EE_{p,\G}(\mu)}\mu+\left(1-\f2p\right)\f{\|u\|_{L^p(\G)}^p}\mu>0.
		\end{equation}
		
		\begin{remark}
			The properties above hold not only on defected grids, but also on the undefected one $\Q$.
		\end{remark}
		
		Now, we can state the main result of the section. In the following we denote by $E_p(u,e)$ the energy functional \eqref{intro:E} restricted to the edge $e\in\EG$. 
		
		\begin{proposition}
			\label{prop:E>0}
			Let $p>2$. If $\G$ is a defected grid and $u\in\Hmu(\G)$ is a ground state of mass $\mu$, then there exists a compact set $K\subset \G$ such that $E_p(u,e)>0$, for every edge $e\in\G\setminus K$. The same holds if $\G$ is replaced by the undefected grid $\Q$.
		\end{proposition}
		
		The proof of Proposition \ref{prop:E>0} relies on the next two auxiliary lemmas.

		\begin{lemma}
			\label{lem:IVP}
			Let $p>2$, $\lambda,a>0$ and $b\in\R$. If $u\in H^1(0,1)$ is a positive solution of the Cauchy problem
			\begin{equation}
			\label{IVP 01}
			\begin{cases}
			u''+|u|^{p-2}u=\lambda u & \text{on }(0,1)\\
			u(0)=a & \\
			u'(0)=b &
			\end{cases}
			\end{equation}
			then
			\begin{equation}
			\label{l eps u l}
			u(x)\leq A_0e^{-\sqrt{\lambda}\,x}+B_0 e^{\sqrt{\lambda}\,x},\qquad\forall x\in(0,1),
			\end{equation}
			where
			\begin{equation}
			\label{A B eps}
			A_0:=\f12\left(a-\f b{\sqrt{\lambda}}\right)\quad\text{and}\quad B_0:=\f12\left(a+\f b{\sqrt{\lambda}}\right).
			\end{equation}
			If, in addition, $a$ and $|b|$ are sufficiently small, then $\|u\|_{L^\infty(0,1)}\ll\lambda^{\f1{p-2}}$ and
			\begin{equation}
			\label{l eps u l bis}
			u(x)\geq A_\delta e^{-\sqrt{\lambda-\delta}\,x}+B_\delta e^{\sqrt{\lambda-\delta}\,x},\qquad\forall x\in(0,1),
			\end{equation}
			where
			\[
			A_\delta:=\f12\left(a-\f b{\sqrt{\lambda-\delta}}\right)\quad\text{and}\quad B_\delta:=\f12\left(a+\f b{\sqrt{\lambda-\delta}}\right)
			\]
			and $\delta:=\|u\|_{L^\infty(0,1)}^{p-2}$.
		\end{lemma}
		\begin{figure}[t]
			\centering
			\includegraphics[width=0.5\textwidth]{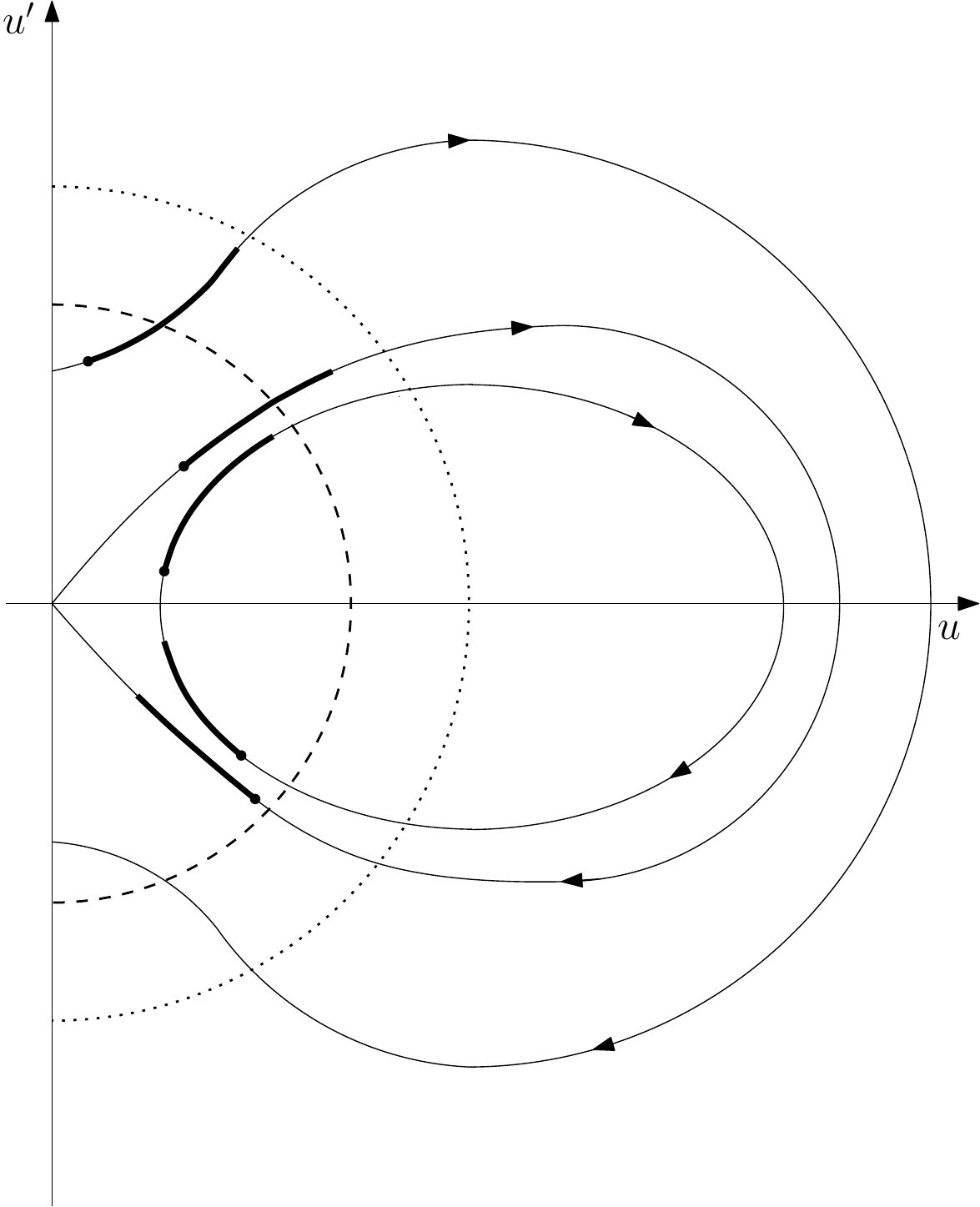}
			\caption{A phase plane description of the situation considered in the second part of the proof of Lemma \ref{lem:IVP}. Taking initial data $(a,b)$ inside the region enclosed by the dashed line and the $u'$ axis, the corresponding positive solution of \eqref{IVP 01} is given by a portion of a corresponding orbit (bold lines in the picture). Since both $\lambda$ and the domain on which the problem is set are fixed a priori, these bold curves will not exit a region enclosed by a suitable dotted line and the $u'$ axis. When the dashed line approaches the origin of the phase plane, i.e. $a,b$ are sufficiently small, the dotted line does the same, i.e. the positive solution of \eqref{IVP 01} is small uniformly on $[0,1]$.}
			\label{fig:phaseplane}
		\end{figure}
		\begin{proof}
			We divide the proof in two parts.
			
			\emph{Part (i): proof of \eqref{l eps u l}.} By \eqref{IVP 01} and $u>0$, we have that $u''=\lambda u-|u|^{p-2}u<\lambda u$. Let then $v\in H^1(0,1)$ be the solution of
			\begin{equation}
			\label{lin sup}
			\begin{cases}
			v''=\lambda v & \text{on }(0,1)\\
			v(0)=a &\\
			v'(0)=b &
			\end{cases}
			\end{equation}
			and set $w(x):=v(x)-u(x)$, for every $x\in [0,1]$. Subtracting \eqref{IVP 01} from \eqref{lin sup}, there results that $w$ satisfies
			\[
			\begin{cases}
			w''>\lambda w & \text{on }(0,1)\\
			w(0)=0 & \\
			w'(0)=0. &
			\end{cases}
			\]
			Since $w''(0)>\lambda w(0)=0$ and $w'(0)=0$, there exists $\eps\in(0,1)$ such that $w'(x)>0$ for every $x\in(0,\eps)$. By $w(0)=0$, it follows also that $w(x)>0$ for every $x\in(0,\eps)$. In particular, $w(\eps)>0$ so that $w''(\eps)>0$ and the argument can be iterated. As one can easily check by a contradiction argument that the iteration cannot get stuck before $x=1$, we obtain $w\geq0$, that is $v\geq u$, on $[0,1]$. Hence, solving explicitly \eqref{lin sup}, one finds
			\[
			v(x)=\f12\left[\left(a-\f b{\sqrt{\lambda}}\right)e^{-\sqrt{\lambda}\,x}+\left(a+\f b{\sqrt{\lambda}} e^{\sqrt{\lambda}\,x}\right)\right],
			\]
			which proves \eqref{l eps u l}.
			
			\emph{Part (ii): proof of \eqref{l eps u l bis}.} Combining \eqref{l eps u l} and \eqref{A B eps}, one sees that (see Figure \ref{fig:phaseplane})
			\[
			\|u\|_{L^\infty(0,1)}\to0,\qquad\text{as}\quad a,\,|b|\to0\,.
			\]
			As a consequence, since $\lambda$ is fixed, if $a$ and $|b|$ are sufficiently small, then $\|u\|_{L^\infty(0,1)}\ll\lambda^{\f1{p-2}}$. Setting $\delta:=\|u\|_{L^\infty(0,1)}^{p-2}$, \eqref{IVP 01} implies that $u''>(\lambda-\delta)u$ on $(0,1)$ and, hence, arguing as in Part (i) (but with reverse inequalities), one obtains \eqref{l eps u l bis}.
		\end{proof}
		
		\begin{remark}
			Note that the second part of the previous proof also shows that, whenever $a$ and $|b|$ are small, the supremum of $u$ cannot be attained in $(0,1)$ (see again Figure \ref{fig:phaseplane}). Indeed, if $u(x_0)=\sup_{x\in[0,1]}u(x)$ with $x_0\in(0,1)$, then $u''(x_0)\leq0$ so that, by \eqref{IVP 01}, $\|u\|_{L^\infty(0,1)}^{p-2}\geq\lambda$.
		\end{remark}
		
		\begin{lemma}
			\label{lem:IVP_E>0}
			Let $p>2$, $\lambda,a>0$, $b\in\R$ and let $u\in H^1(0,1)$ be the positive solution of the Cauchy problem \eqref{IVP 01}. If $a,|b|$ are sufficiently small, then $E_p(u,[0,1])>0$.
		\end{lemma}
		
		\begin{proof}
			First, we note that multiplying the first line of \eqref{IVP 01} by $u'$ and integrating over $[0,x]$ leads to
			\[
			\f12|u'(x)|^2+\f1p|u(x)|^p-\f\lambda2|u(x)|^2=\f{b^2}2+\f{a^p}p-\f{\lambda a^2}2.
			\]
			A further integration over $[0,1]$ yields
			\[
			\f12\|u'\|_{L^2(0,1)}^2+\f1p\|u\|_{L^p(0,1)}^p-\f\lambda2\|u\|_{L^2(0,1)}^2=\f{b^2}2+\f{a^p}p-\f{\lambda a^2}2
			\]
			and, plugging into the energy functional, there results
			\begin{equation}
			\label{e en}
			E_p(u,[0,1])=\f\lambda2\|u\|_{L^2(0,1)}^2-\f2p\|u\|_{L^p(0,1)}^p+\f{b^2}2+\f{a^p}p-\f{\lambda a^2}2\,.
			\end{equation}
			Moreover, setting $\ell:=\f ba$, \eqref{e en} reads
			\begin{equation}
			\label{e en l}
			E_p(u,[0,1])=\f\lambda2\|u\|_{L^2(0,1)}^2-\f2p\|u\|_{L^p(0,1)}^p+\f{\ell^2-\lambda}2 a^2+\f{a^p}p\,.
			\end{equation}
			Now, by \eqref{l eps u l bis}, one sees that
			\begin{equation}
			\label{l eps L2}
			\begin{split}
			\|u\|_{L^2(0,1)}^2\geq&\int_0^1\left(A_\delta e^{-\sqrt{\lambda-\delta}\,x}+B_\delta e^{\sqrt{\lambda-\delta}\,x}\right)^2\,dx\\[.2cm]
			=&\,\f {a^2}4\left(1-\f\ell{\sqrt{\lambda-\delta}}\right)^2\f{1-e^{-2\sqrt{\lambda-\delta}}}{2\sqrt{\lambda-\delta}}\\[.2cm]
			+&\,\f {a^2}4\left(1+\f\ell{\sqrt{\lambda-\delta}}\right)^2\f{e^{2\sqrt{\lambda-\delta}}-1}{2\sqrt{\lambda-\delta}}
			+\f{a^2}2\left(1-\f{\ell^2}{\lambda-\delta}\right)\\[.2cm]
			=&\,\ell^2 a^2\left[\f{e^{4\sqrt{\lambda-\delta}}-1}{8(\lambda-\delta)e^{2\sqrt{\lambda-\delta}}\sqrt{\lambda-\delta}}-\f1{2(\lambda-\delta)}\right]
			+\,\ell a^2\f{(e^{2\sqrt{\lambda-\delta}}-1)^2}{4(\lambda-\delta)e^{2\sqrt{\lambda-\delta}}}\\[.2cm]
			+&\, a^2\left(\f{e^{4\sqrt{\lambda-\delta}}-1}{8e^{2\sqrt{\lambda-\delta}}\sqrt{\lambda-\delta}}+\f12\right)
			\end{split}
			\end{equation}
			where $\delta:=\|u\|_{L^\infty(0,1)}^{p-2}\ll\lambda$. Since $\|u\|_{L^p(0,1)}^p\leq \delta \|u\|_{L^2(0,1)}^2$, by \eqref{e en l} and \eqref{l eps L2} there results
			\[
			E_{p}(u,[0,1])=\f{p\lambda-2\delta}{2p}\|u\|_{L^2(0,1)}^2+\f{\ell^2-\lambda}2 a^2+o(a^2)\geq a^2G_\lambda(\ell,\delta)+\f{a^p}{p}
			\]
			with
			\begin{multline*}
			G_\lambda(\ell,\delta):=\ell^2\left[\f{p\lambda-2\delta}{\lambda-\delta}\f{e^{4\sqrt{\lambda-\delta}}-1}{16pe^{2\sqrt{\lambda-\delta}}\sqrt{\lambda-\delta}}-\f{p\lambda-2\delta}{4p(\lambda-\delta)}+\f12\right]\\[.2cm]
			+\ell \f{(p\lambda-2\delta)(e^{2\sqrt{\lambda-\delta}}-1)^2}{8p(\lambda-\delta)e^{2\sqrt{\lambda-\delta}}}+\left(\f{(p\lambda-2\delta)e^{4\sqrt{\lambda-\delta}}-1}{16pe^{2\sqrt{\lambda-\delta}}\sqrt{\lambda-\delta}}+\f{p\lambda-2\delta}{4p}-\f\lambda2\right).
			\end{multline*}
			As a consequence, recalling that $\delta\to0$ as $a,\,|b|\to0$ by Lemma \ref{lem:IVP}, to conclude the proof we are left to show that, for any fixed $\lambda>0$,
			\[
			G_\lambda(\ell,\delta)>0,\qquad\forall\ell\in\R,
			\]
			provided that $a$ and $|b|$ are sufficiently small. To this aim, define $F_\lambda:[0,+\infty)\to\R$ as
			\[
			F_\lambda(\ell):=\left(\f{e^{4\sqrt{\lambda}}-1}{16e^{2\sqrt{\lambda}}\sqrt{\lambda}}+\f14\right)\ell^2+\f{(e^{2\sqrt{\lambda}}-1)^2}{8e^{2\sqrt{\lambda}}}\ell+\lambda\left(\f{e^{4\sqrt{\lambda}}-1}{16e^{2\sqrt{\lambda}}\sqrt{\lambda}}-\f14\right).
			\]
			Note that $F_\lambda$ is a parabola with strictly positive coefficient of the quadratic term $\ell^2$. Hence, as $G_\lambda(\ell,0)=F_\lambda(\ell)$, if $F_\lambda(\ell)\neq0$ for every $\ell\in\R$, then $G_\lambda(\ell,\delta)>0$ for every $\ell\in\R$, given $a$ and $|b|$ small enough.
			
			It remains to prove that, for every $\lambda>0$,
			\begin{equation}
			\label{vert>0}
			4\lambda\left(\f{e^{4\sqrt{\lambda}}-1}{16e^{2\sqrt{\lambda}}\sqrt{\lambda}}-\f14\right)\left(\f{e^{4\sqrt{\lambda}}-1}{16e^{2\sqrt{\lambda}}\sqrt{\lambda}}+\f14\right)- \left[\f{(e^{2\sqrt{\lambda}}-1)^2}{8e^{2\sqrt{\lambda}}}\right]^2>0\,.
			\end{equation}
			Setting $y:=2\sqrt{\lambda}$, the left--hand side of the above expression reads
			\begin{equation}
			\label{vertice y}
			\begin{split}
			\f{y^2}{16}&\left(\f{e^{2y}-1}{2ye^y}-1\right)\left(\f{e^{2y}-1}{2ye^y}+1\right)-\f{(e^y-1)^4}{64e^{2y}}\\
			=&\f{y^2}{16}\left(\f{(e^{2y}-1)^2}{4y^2e^{2y}}-1\right)-\f{(e^y-1)^4}{64e^{2y}}=\f{(e^y-1)^2(e^y+1)^2}{64e^{2y}}-\f{y^2}{16}-\f{(e^y-1)^4}{64e^{2y}}\\
			=&\f{(e^y-1)^2}{64e^{2y}}\left[(e^y+1)^2-(e^y-1)^2\right]-\f{y^2}{16}=\f{y^2}{16}\left[\f{(e^y-1)^2}{y^2e^y}-1\right]\,.
			\end{split}
			\end{equation}
			On the one hand,
			\begin{equation}
			\label{lim y to 0}
			\lim_{y\to0^+}\f{y^2}{16}\left[\f{(e^y-1)^2}{y^2e^y}-1\right]=0.
			\end{equation}
			On the other hand,
			\begin{equation}
			\label{der y}
			\begin{split}
			\f d{dy}\left[\f{(e^y-1)^2}{y^2e^y}-1\right]=&\f{2(e^y-1)y^2e^{2y}-(2y+y^2)e^y(e^y-1)^2}{y^4e^{2y}}\\
			=&\f{(e^y-1)\left[2y^2e^y-(e^y-1)(2y+y^2)\right]}{y^4e^{y}}\\
			=&\f{e^y-1}{y^3e^{y}}\left[y(e^y+1)+2(1-e^y)\right].
			\end{split}
			\end{equation}
			Now, $\f{e^y-1}{y^3e^{3y}}>0$ for every $y>0$. In addition, $y(e^y+1)+2(1-e^y)=0$ at $y=0$ and 
			\[
			\f d{dy}\left[y(e^y+1)+2(1-e^y)\right]=e^y+1+ye^y-2e^y=1+(y-1)e^y>0
			\] 
			for every $y>0$, since $1+(y-1)e^y=0$ at $y=0$ and $\f d{dy}\left[1+(y-1)e^y\right]=ye^y>0$. Thus, combining with \eqref{der y}, \eqref{lim y to 0} and \eqref{vertice y}, one finally obtains \eqref{vert>0}.
		\end{proof}
		
		We are now in position to prove Proposition \ref{prop:E>0}.
		
		\begin{proof}[Proof of Proposition \ref{prop:E>0}]
			We limit ourselves to the case of a defected grid $\G$, as the proof for the undefected grid $\Q$ is completely analogous.
			
			Let then $u\in\Hmu(\G)$ be a ground state of mass $\mu$ and assume, without loss of generality (see \eqref{eq-sign}), that $u>0$. Let $\overline{x}\in\G$ be a point where $u$ attains its maximum.
			
			First, we note that, for every $\eps>0$, there exists $R_\eps>0$ such that
			\begin{equation}
			\label{eq-linfsmall}
			\|u\|_{L^\infty(\G\setminus B_{R_\eps}(\overline{x},\G))}\leq\eps,
			\end{equation}
			where $B_{R_\eps}(\overline{x},\G)$ denotes the ball in $\G$ of radius $R_\eps$ centered at $\overline{x}$. Indeed, as $u\in H^1(\G)$ and every edge of $\G$ is of the same length, there exists $C>0$, independent of $u$ and $e$, such that
			\[
			\|u\|_{L^\infty(e)}^2\leq C\|u\|_{H^1(e)}^2,\qquad\forall e\in\EG.
			\]
			Therefore, since, for every $R>0$
			\[
			\|u\|_{H^1(\G)}^2=\sum_{e\in B_{R}(\overline{x},\G)}\|u\|_{H^1(e)}^2+\sum_{e\in\G\setminus B_{R}(\overline{x},\G)} \|u\|_{H^1(e)}^2
			\]
			and since
			\[\sum_{e\in B_{R}(\overline{x},\G)}\|u\|_{H^1(e)}^2=\|u\|_{H^1(B_{R}(\overline{x},\G))}^2\to\|u\|_{H^1(\G)}^2,\qquad\text{as}\quad R\to+\infty,
			\]
			we have
			\[
			\|u\|_{L^\infty(\G\setminus B_{R}(\overline{x},\G))}^2\leq C\sum_{e\in\G\setminus B_{R}(\overline{x},\G)} \|u\|_{H^1(e)}^2\to0,\qquad\text{as}\quad R\to+\infty.
			\]
			
			On the other hand, by \eqref{eq-NLSgrafo}, $u$ has to satisfy \eqref{IVP 01} on each edge $e$, for some suitable $a>0,\,b\in\R$ and a Lagrange multiplier $\lambda$ obtained as in \eqref{eq-lambdapos}.
			
			Now, fix $\eps>0$ small. If we take $R_\eps$ in such a way that \eqref{eq-linfsmall} is satisfied and consider any edge $e\in\G\setminus B_{R_\eps}(\overline{x},\G)$, there results
			\begin{equation}
			\label{eq: 0<lu-u^p<leps} 
			0\leq\lambda u-|u|^{p-2}u\leq\lambda\eps\qquad\text{on}\quad e.
			\end{equation}
			Thus \eqref{IVP 01} entails
			\[
			u'(x)=b+\int_0^x u''\,dy=b+\int_0^x(\lambda u-|u|^{p-2}u)\,dy,\qquad\forall x\in e,
			\] 
			whence, further integrating on $[0,1]$,
			\[
			u(1)=a+b+\int_0^1\int_0^x(\lambda u-|u|^{p-2}u)\,dy\,dx\,.
			\]
			Coupling with \eqref{eq: 0<lu-u^p<leps} and recalling that $a=u(0)$ and $u(0),\,u(1)\leq\|u\|_{L^\infty(\G)}\leq\eps$, one obtains
			\[
			|b|\leq|u(1)|+|a|+\left|\int_0^1\int_0^x(\lambda u-|u|^{p-2}u)\,dy\,dx\right|\leq \left(2+\f{\lambda}2\right)\eps
			\]
			for every edge $e\in\G\setminus B_{R_\eps}(\overline{x},\G)$. As a consequence, whenever $\eps$ is sufficiently small, $u_e$ satisfies the assumptions of Lemma \ref{lem:IVP_E>0} on every $e\in\G\setminus B_{R_\eps}(\overline{x},\G)$. Hence, setting $K:=B_{R_\eps}(\overline{x},\G)$, the proof is complete.
		\end{proof}
		
		\begin{remark}
			Note that in the first part of the proof of Proposition \ref{prop:E>0} the notation $e\in B_{R_\eps}(\overline{x},\G)$, as well as $e\in \G\setminus B_{R_\eps}(\overline{x},\G)$, is not rigorous since $B_{R_\eps}(\overline{x},\G),\,\G\setminus B_{R_\eps}(\overline{x},\G)$ may not be subgraphs. However, we kept such notation not to weigh down the text as it does not give rise to misunderstandings.
		\end{remark}
		
		
		
		\subsubsection{Consequences of the Sobolev inequalities}
		\label{subsubsec:Sobolev}
		
		Now, we highlight some consequences of the simultaneous validity of the one and the two--dimensional Sobolev inequality \eqref{eq:s1d} and \eqref{eq:s2d}. Note that, by Definitions \ref{def:compD}, \ref{def:ZperD} and \ref{def:Z2perD} and Theorem \ref{THM:s2d_bound}, the following properties clearly hold under the assumptions of Theorem \ref{THM:ex_com&Z}.
		
		First, as we recalled in Section \ref{subsec:sobineq}, any defected grid supports the one--dimensional Sobolev inequality \eqref{eq:s1d}. Arguing as in \cite[Proof of Theorem 2.1]{ADST19}, one obtains the one--dimensional Gagliardo--Nirenberg inequalities for every $u\in H^1(\G)$, i.e.
		\begin{gather}
		\label{eq:gn1d} \|u\|_{L^p(\G)}^p\leq C_1(p,\G)\|u\|_{L^2(\G)}^{\f p2+1}\|u'\|_{L^2(\G)}^{\f p2-1},\qquad p\in[2,+\infty),\\[.2cm]
		\|u\|_{L^\infty(\G)}\leq C_\infty(\G)\|u\|_{L^2(\G)}^{\f12}\|u'\|_{L^2(\G)}^{\f12}.
		\end{gather}
		On the other hand, arguing as in \cite[Proof of Theorem 2.3]{ADST19}, the validity of \eqref{eq:s2d} yields the two--dimensional Gagliardo--Nirenberg inequality
		\begin{equation}
		\label{eq:gn2d}
		\|u\|_{L^p(\G)}^p\leq C_2(p,\G)\|u\|_{L^2(\G)}^2\|u'\|_{L^2(\G)}^{p-2},\qquad\forall u\in H^1(\G),\quad p\in[2,+\infty).
		\end{equation}
		Finally, arguing as in \cite[Proof of Corollary 2.4]{ADST19}, the simultaneous validity of \eqref{eq:gn1d}--\eqref{eq:gn2d} entails
		\begin{equation}
		\label{eq:gnint}
		\|u\|_{L^p(\G)}^p\leq K(p,\G)\|u\|_{L^2(\G)}^{p-2}\|u'\|_{L^2(\G)}^{2},\qquad\forall u\in H^1(\G),\quad p\in[4,6].
		\end{equation}
		
		\begin{remark}
			In each of the previous expressions, by $C_1(p,\G)$, $C_\infty(\G)$, $C_2(p,\G)$, $K(p,\G)>0$ we denote the sharpest constants for which the corresponding inequalities hold.
		\end{remark}
		
		Let us show how \eqref{eq:gn1d} and \eqref{eq:gnint} immediately entail some of the claims of Theorem \ref{THM:ex_com&Z}. On the one hand, \eqref{eq:gn1d} implies
		\begin{equation}
		\label{eq:coerc}
		E_p(u,\G)\geq\f12\|u'\|_{L^2(\G)}^2-\f{C_1(p,\G)}p\mu^{\f p4+\f12}\|u'\|_{L^2(\G)}^{\f p2-1},\qquad\forall u\in\Hmu(\G),
		\end{equation}
		which provides lower boundedness and coercivity of $E_p(\cdot,\G)$ on $\Hmu(\G)$ for every $p\in(2,6)$ and every $\mu>0$. On the other hand, for every $p\in[4,6)$, plugging \eqref{eq:gnint} into \eqref{intro:E} yields
		\[
		E_p(u,\G)\geq\f12\|u'\|_{L^2(\G)}^2\left(1-\f{2K(p,\G)}p\mu^{\f {p-2}{2}}\right),\qquad\forall u\in\Hmu(\G),
		\]
		so that, setting
		\begin{equation}
		\label{mucrit}
		\mu_{p,\G}:=\left(\f p{2K(p,\G)}\right)^{\f2{p-2}}\,,
		\end{equation}
		there results
		\begin{equation}
		\label{eq-enpos}
		E_p(u,\G)\geq0,\qquad\forall u\in\Hmu(\G),\quad\forall\mu\leq\mu_{p,\G}.
		\end{equation}
		Combining with \eqref{eq:infnonpos}, one obtains
		\begin{equation}
		\label{eq:infzero}
		\EE_{p,\G}(\mu)=0,\qquad\forall\mu\leq\mu_{p,\G},
		\end{equation}
		thus proving the first line of \eqref{levelTHM1}. In addition, as \eqref{eq-enpos} is strict whenever $\mu<\mu_{p,\G}$, also the nonexistence claims of  items (ii.1)--(ii.2) in Theorem \ref{THM:ex_com&Z} are proved.
		
		Finally, let $\mu>\mu_{p,\G}$ and let $(u_n)_n\subset \Hmu(\G)$ be a maximizing sequence for \eqref{eq:gnint}, i.e. 
		\[
		\f{\|u_n\|_{L^p(\G)}^p}{\|u_n'\|_{L^2(\G)}^2}\longrightarrow K(p,\G)\mu^{\f {p-2}{2}}\qquad\text{as}\quad n\to+\infty.
		\]
		Then we have
		\[
		\EE_{p,\G}(\mu)\leq\liminf_{n\to+\infty}E_p(u_n,\G)=\f12\|u_n'\|_{L^2(\G)}^2\left(1-\left(\f\mu{\mu_{p,\G}}\right)^{\f{p-2}{2}}+o(1)\right)<0,
		\]
		which proves the second line of \eqref{levelTHM1}.

		
		
		\subsubsection{Dichotomy alternative for minimizing sequences}
		
		We recall a dichotomy alternative for minimizing sequences of $E_p(\cdot,\G)$ in $\Hmu(\G)$, which is independent of the specific structure of the defects of $\G$. The proof is analogous to that of \cite[Lemma 3.2]{ADST19}. We report it here for the sake of completeness.
		
		\begin{lemma}
			\label{lem dichotomy}
			Let $\G$ be a defected grid, $\mu>0$ and $p\in(2,6)$. Let $(u_n)_n\subset\Hmu(\G)$ be a minimizing sequence for $E_p(\cdot,\G)$. If $u_n\rightharpoonup u$ in $H^1(\G)$ as $n\to+\infty$, then
			\begin{itemize}
				\item[$(i)$] either $u_n\to0$ in $L^\infty_{\text{loc}}(\G)$, or
				
				\item[$(ii)$] $u$ is a ground state of mass $\mu$.
			\end{itemize}
		\end{lemma}
		
		\begin{proof}
			Let $(u_n)_n\subset\Hmu(\G)$ be a minimizing sequence for $E_p(\cdot,\G)$ such that $u_n\rightharpoonup u$ in $H^1(\G)$ as $n\to+\infty$. As a consequence, $u_n\to u$ in $L^\infty_{\text{loc}}(\G)$. Furthermore, by weak lower semicontinuity,
			\begin{equation}
			\label{wls}
			\|u'\|_{L^2(\G)}\leq\liminf_{n\to+\infty}\|u_n'\|_{L^2(\G)}\qquad\text{and}\qquad \|u\|_{L^2(\G)}\leq\liminf_{n\to+\infty}\|u_n\|_{L^2(\G)}.
			\end{equation}
			Set $m:=\|u\|_{L^2(\G)}^2$. If $m=0$, then $u\equiv0$ on $\G$ and $(i)$ holds. On the contrary, if $m=\mu$, then $u\in\Hmu(\G)$ and $u_n\to u$ in $L^2(\G)$, so that, in turn, $u_n\to u$ in $L^p(\G)$. Hence, coupling with \eqref{wls}, we get $E_p(u,\G)\leq\liminf_{n\to+\infty}E_p(u_n,\G)=\EE_{p,\G}(\mu)$, which proves $(ii)$.
			
			Therefore, it is left to exclude the case $m\in(0,\mu)$. Assume by contradiction that $m\in(0,\mu)$. By the Brezis--Lieb Lemma \cite{BL83},
			\begin{equation}
			\label{eq-bl}
			E_p(u_n,\G)=E_p(u,\G)+E_p(u_n-u,\G)+o(1)\qquad\text{as}\quad n\to+\infty.
			\end{equation}
			In addition, by $u_n\rightharpoonup u$ in $L^2(\G)$,
			\[
			\|u_n-u\|_{L^2(\G)}^2=\|u_n\|_{L^2(\G)}^2+\|u\|_{L^2(\G)}^2-2\langle u_n,u\rangle_{L^2(\G)}=\mu-m+o(1)\qquad\text{as}\quad n\to+\infty.
			\]
			On the one hand, since $p>2$ and $\|u_n-u\|_{L^2(\G)}^2<\mu$,
			\[
			\begin{split}
			\EE_{p,\G}(\mu)\leq& E_p\left(\f{\sqrt{\mu}}{\|u_n-u\|_{L^2(\G)}}\,(u_n-u),\G\right)\\
			=&\f12\f\mu{\|u_n-u\|_{L^2(\G)}^2}\|u_n'-u'\|_{L^2(\G)}^2-\f1p\f{\mu^{\f p2}}{\|u_n-u\|_{L^2(\G)}^p}\|u_n-u\|_{L^p(\G)}^p\\
			\leq&\f\mu{\|u_n-u\|_{L^2(\G)}^2}E_p\left(u_n-u,\G\right)
			\end{split}
			\]
			so that, rearranging terms and passing to the liminf as $n\to+\infty$,
			\begin{equation}
			\label{dic1}
			\liminf_{n\to+\infty}E_p(u_n-u,\G)\geq\f{\mu-m}\mu\,\EE_{p,\G}(\mu).
			\end{equation}
			On the other hand, similar calculations yield
			\[
			\EE_{p,\G}(\mu)\leq E_p\left(\sqrt{\f\mu m}\,u,\G\right)<\f\mu m\, E_p(u,\G),
			\]
			that is 
			\begin{equation}
			\label{dic2}
			E_p(u,\G)>\f m\mu\, \EE_{p,\G}(\mu).
			\end{equation}
			Finally, combining \eqref{dic1}--\eqref{dic2} with \eqref{eq-bl}, there results
			\[
			\begin{split}
			\EE_{p,\G}(\mu)=\liminf_{n\to+\infty}E_p(u_n,\G)=&\liminf_{n\to+\infty}E_p(u_n-u,\G)+E_p(u,\G)\\
			>&\f{\mu-m}\mu\EE_{p,\G}(\mu)+\f m\mu\EE_{p,\G}(\mu)=\EE_{p,\G}(\mu),
			\end{split}
			\] 
			which is a contradiction.
		\end{proof}
		
		
		\subsection{Compactly defected grids and grids with $\Z$--periodic defects}
		\label{subsec:proofNLS}
		
		Here we present the proof of the part of Theorem \ref{THM:ex_com&Z} devoted to compactly defected grids and grids with $\Z$--periodic defects. In particular, in view of Section \ref{subsubsec:Sobolev}, it is left to prove
		\begin{itemize}
			\item[(a)] item (i),
			\item[(b)] the existence part of items (ii.1) and (ii.2),
			\item[(c)] estimates \eqref{muG<muQ} and \eqref{EG<EQ}.
		\end{itemize}
		To this aim, we state the following existence criterion.
		
		\begin{lemma}
			\label{lem crit ex}
			Let $\mu>0$, $p\in(2,6)$ and $\G$ be either a compactly defected grid or a grid with $\Z$--periodic defects. If $\EE_{p,\G}(\mu)<\EE_{p,\Q}(\mu)$, then there exists a ground state of mass $\mu$.
		\end{lemma}
		
		\begin{proof}
			We divide the proof in two cases according to the two classes of defected grids.
			
			\emph{Case (i): $\G$ is a compactly defected grid.} Let $(u_n)_n\subset\Hmu(\G)$ be a minimizing sequence for $E_p(\cdot,\G)$. By \eqref{eq:coerc}, $(u_n)_n$ is bounded in $H^1(\G)$, so that $u_n\rightharpoonup u$ in $H^1(\G)$ and $u_n\to u$ in $L^\infty_{\text{loc}}(\G)$ as $n\to+\infty$. By Lemma \ref{lem dichotomy}, either $u\equiv0$ on $\G$ or $u$ is a ground state of mass $\mu$. Hence, to conclude it suffices to show that $u\not\equiv0$.
			
			In view of Remark \ref{rem:compD}, let $K$ be a compact subset of $\Q$ containing the union of all the defects of $\G$, and set $K':=K\cap\G$. Assume by contradiction that $u\equiv0$ on $\G$. As a consequence, $\eps_n:=\|u_n\|_{L^\infty(K')}\to0$ as $n\to+\infty$. Therefore, setting $v_n:=\kappa_n(u_n-\eps_n)_+$, where $\kappa_n>0$ is chosen so that $\|v_n\|_{L^2(\G)}^2=\mu$, we have that $v_n\in\Hmu(\G)$, $v_n\equiv0$ on $K'$ and
			\begin{equation}
			\label{eq-ragionamento}
			\begin{split}
			E_p(v_n,\G)=&\f{\kappa_n^2}2\int_\G\left|\left(u_n-\eps_n\right)_+'\right|^2\,dx-\f{\kappa_n^p}p\int_\G\left|(u_n-\eps_n)_+\right|^p\,dx\\
			\leq&\f12\|u_n'\|_{L^2(\G)}^2-\f1p\|u_n\|_{L^p(\G)}^p+o(1)=\EE_{p,\G}(\mu)+o(1)\qquad\text{as}\quad n\to+\infty
			\end{split}
			\end{equation}
			(where we also used that $\kappa_n\to1$ as $n\to+\infty$).
			Then, $(v_n)_n$ is a minimizing sequence for $E_p(\cdot,\G)$ in $\Hmu(\G)$. Since $v_n$ is identically zero on $K'$, we can think of it as a function supported on the undefected grid $\Q$ which vanishes on the whole $K$. Therefore, with a little abuse of notation, $v_n$ also belongs to $\Hmu(\Q)$ and thus
			\[
			\EE_{p,\G}(\mu)=\lim_{n\to+\infty}E_p(u_n,\G)=\lim_{n\to+\infty}E_p(v_n,\G)=\lim_{n\to+\infty}E_p(v_n,\Q)\geq\EE_{p,\Q}(\mu),
			\]
			contradicting the hypothesis that $\EE_{p,\G}(\mu)<\EE_{p,\Q}(\mu)$.
			
			\emph{Case (ii): $\G$ has $\Z$--periodic defects.} By Remark \ref{rem:Dper}, there exists $W\subset\Q$ such that $\G=\bigcup_{k\in\Z}\widetilde{W}_k$, where $\widetilde{W}_k:=\widetilde{W}+k\vec{v}$, for every $k\in\Z$, with $\widetilde{W}:=W\cap\G$. Moreover, recall that each $\widetilde{W}_k$ intersects finitely many defects of $\G$. 
			
			Since $\G$ is invariant by integer translations along $\vec{v}$, letting $(u_n)_n\subset\Hmu(\G)$ be a minimizing sequence for $E_p(\cdot,\G)$, it is not restrictive to assume that each $u_n$ is such that the maximum value of $|u_n|$ on the union of the boundaries of all the defects of $\G$ is attained somewhere in $\widetilde{W}_0$. As $\widetilde{W}_0$ intersects finitely many defects of $\G$, let $K\subset\widetilde{W}_0$ be a compact subset containing the union of the boundaries of all the defects in $\widetilde{W}_0$. Then, for every $n$, the maximum of $|u_n|$ on the union of the boundaries of all the defects of $\G$ is attained in $K$.
			
			Now, arguing as in Case (i), we have that $u_n\rightharpoonup u$ in $H^1(\G)$ and $u_n\to u$ in $L_{\text{loc}}^\infty(\G)$ as $n\to+\infty$. Assuming by contradiction $u\equiv0$ implies that $u_n\to0$ in $L^\infty(K)$, that is $u_n\to0$ uniformly on the union of the boundaries of all the defects of $\G$. Then, arguing again as in Case (i), one can check that $u\equiv0$ entails that $\EE_{p,\G}(\mu)\geq\EE_{p,\Q}(\mu)$, which contradicts the hypotheses of the lemma. As a consequence, by Lemma \ref{lem dichotomy}, one concludes that $u$ is a ground state of mass $\mu$.
		\end{proof}
		
		\begin{proof}[Proof of Theorem \ref{THM:ex_com&Z}: compactly defected grids and grids with $\Z$--periodic defects]
			Recall that it is sufficient to prove items (a), (b) and (c) listed at the beginning of Section \ref{subsec:proofNLS}.
			
			Preliminarily, recall that if $\G$ is compactly defected, then Remark \ref{rem:compD} guarantees the existence of a compact set $\overline{K}\subset\Q$ such that $\bigcup_{k\in I}\D_k\subset \overline{K}$. On the other hand, if $\G$ has $\Z$--periodic defects, then let $\overline{K}$ be the strip in Definition \ref{def:ZperD}(ii) such that $\bigcup_{k\in I}\D_k\subseteq\overline{K}$. In what follows, $\overline{K}$ will denote the former or the latter set as everything holds in both cases.
			
			First, we focus on the regimes where the problem on $\Q$ admits a ground state, i.e. (see Section \ref{subsec:grid} or \cite[Theorems 1.1--1.2]{ADST19}) $p\in(2,4)$ with $\mu>0$, $p\in(4,6)$ with $\mu\geq\mu_\Q$, $p=4$ with $\mu>\mu_\Q$. Let $u\in\Hmu(\Q)$ be a ground state of mass $\mu$ on $\Q$. By Proposition \ref{prop:E>0}, there exists a compact set $K\subset\Q$ such that $E_{p}(u,\Q\setminus K)>0$. Moreover, exploiting the periodicity of $\Q$, it is possible to choose $u$ (properly translating it on the grid) so to satisfy $K\cap\overline{K}=\emptyset$. 
			
			Set now $m:=\|u\|_{L^2\left(\bigcup_{k\in I}\D_k\right)}^2$, $\kappa:=\sqrt{\f{\mu}{\mu-m}}$ and the function $v:\G\to\R$ given by $v:=\kappa\, u_{\mid\Q\setminus\bigcup_{k\in I}\D_k}$. By construction $v\in\Hmu(\G)$. Moreover, since $K\cap\overline{K}=\emptyset$ and $\bigcup_{k\in I}\D_k\subseteq\overline{K}$, then $E_p\left(u,\bigcup_{k\in I}\D_k\right)>0$, so that 
			\[
			E_p\left(u,\G\right)<E_p\left(u,\G\right)+E_p\left(u,\bigcup_{k\in I}\D_k\right)=E_p(u,\Q)=\EE_{p,\Q}(\mu)\,.
			\]
			Henceforth, making use of $\kappa>1$ and $\EE_{p,\Q}(\mu)\leq0$ (see \eqref{intro:EQ_4p6}),
			\begin{equation}
			\label{E_norm}
			\EE_{p,\G}(\mu)\leq E_p(v,\G)=\f{\kappa^2}2\int_{\G}|u'|^2\,dx-\f{\kappa^p}p\int_{\G}|u|^p\,dx<\kappa^2 E_p(u,\G)<\EE_{p,\Q}(\mu).
			\end{equation}
			Thus, by Lemma \ref{lem crit ex}, the previous inequality entails that if $\EE_{p,\Q}(\mu)$ is attained, then $\EE_{p,\G}(\mu)$ is attained too. 
			
			On the one hand, this proves (a). On the other hand, we get that there exist ground states of mass $\mu$ on $\G$ when $p\in(4,6)$ and $\mu\geq\mu_{p,\Q}$. Furthermore, again by \eqref{E_norm}, $\EE_{p,\G}(\mu_{p,\Q})<\EE_{p,\Q}(\mu_{p,\Q})=0$, so that, combining with \eqref{eq:infzero}, we get $\mu_{p,\G}<\mu_{p,\Q}$. Similarly, we have that there exist ground states on $\G$ when $p=4$ and $\mu>\mu_{4,\Q}$. Since also $\EE_{4,\G}(\mu)<\EE_{4,\Q}(\mu)$, for every $\mu>\mu_{4,\Q}$, exploiting once more \eqref{E_norm}, one obtains $\mu_{4,\G}\leq\mu_{4,\Q}$, which completes the proof of \eqref{muG<muQ}. 
			
			In addition, when either $p\in(4,6)$ and $\mu\in(\mu_{p,\G},\mu_{p,\Q})$, or $p=4$ and $(\mu_{4,\G},\mu_{4,\Q}]$, provided $\mu_{4,\G}<\mu_{4,\Q}$, by \eqref{intro:EQ_4p6} and \eqref{levelTHM1} there results $\EE_{p,\G}(\mu)<0=\EE_{p,\Q}(\mu)$. Hence, \eqref{EG<EQ} is established (concluding the proof of (c)) and, also in this case, Lemma \ref{lem crit ex} entails the existence of ground states of mass $\mu$ on $\G$.
			
			This completes the proof of (ii.2). To get (ii.1) (and thus (b)), it remains to show that, when $p\in(4,6)$, the ground states exist also at $\mu=\mu_{p,\G}$. In this case, $\EE_{p,\G}(\mu_{p,\G})=0=\EE_{p,\Q}(\mu_{p,\G})$ and Lemma \ref{lem crit ex} does not apply. Let then $(u_n)_n\in H_{\mu_{p,\G}}^1(\G)$ be a maximizing sequence for the Gagliardo--Nirenberg inequality \eqref{eq:gnint}, that is
			\begin{equation}
			\label{eq:maxGN}
			\f{\|u_n\|_{L^p(\G)}^p}{\mu_{p,\G}^{\f p2-1}\|u_n'\|_{L^2(\G)}^2}\to K(p,\G)\qquad\text{as }n\to+\infty.
			\end{equation}
			If $\G$ is compacty defected, then trivially the maximum of $|u_n|$ on $\bigcup_{k\in I}\partial\D_k$ is attained in a fixed compact subset of $\G$, independent of $n$. On the other hand, arguing as in the proof of Lemma \ref{lem crit ex}, if $\G$ has $\Z$--periodic defects, then we can assume without loss of generality that the maximum of $|u_n|$ on $\bigcup_{k\in I}\partial\D_k$ is attained in a fixed compact subset of $\widetilde{W}_0$, independent of $n$. 
			
			Now, by \eqref{eq:gn1d}
			\begin{equation*}
			\f{\|u_n\|_{L^p(\G)}^p}{\mu_{p,\G}^{\f p2-1}\|u_n'\|_{L^2(\G)}^2}\leq\f{\mu_{p,\G}^{\f p4+\f12}\|u_n'\|_{L^2(\G)}^{\f p2-1}}{\mu_{p,\G}^{\f p2-1}\|u_n'\|_{L^2(\G)}^2}=\f{\mu_{p,\G}^{\f{(6-p)}4}}{\|u_n'\|_{L^2(\G)}^{\f{(6-p)}2}},
			\end{equation*}
			so that, coupling with \eqref{eq:maxGN}, we have that $(u_n)_n$ is bounded in $H^1(\G)$. Hence, $u_n\rightharpoonup u$ in $H^1(\G)$ and $u_n\to u$ in $L_{\text{loc}}^\infty(\G)$ as $n\to+\infty$. Moreover, again by the boundedness of $(u_n)_n$ in $H^1(\G)$ and by \eqref{mucrit} and \eqref{eq:maxGN}
			\[
			E_p(u_n,\G)=\f12\|u_n'\|_{L^2(\G)}^2\left(1-\f{2K(p,\G)}p\mu_{p,\G}^{\f p2-1}+o(1)\right)=o(1)\qquad\text{as}\quad n\to+\infty.
			\]
			Thus, $(u_n)_n\subset\Hmu(\G)$ is a minimizing sequence for $E_p(\cdot,\G)$ at $\mu=\mu_{p,\G}$ and, hence, by Lemma \ref{lem dichotomy}, either $u\equiv0$ on $\G$ or $u\in H_{\mu_{p,\G}}^1(\G)$ is the required ground state.
			
			Assume by contradiction that $u\equiv0$ on $\G$. Exploiting \eqref{eq:gnint} at $p=4$ gives
			\begin{multline*}
			\f{\|u_n\|_{L^p(\G)}^p}{\mu_{p,\G}^{\f p2-1}\|u_n'\|_{L^2(\G)}^2}\leq\f{\|u\|_{L^\infty(\G)}^{p-4}\|u_n\|_{L^4(\G)}^4}{\mu_{p,\G}^{\f p2-1}\|u_n'\|_{L^2(\G)}^2}\\[.2cm]
			\leq\|u_n\|_{L^\infty(\G)}^{p-4}\f{K(4,\G)\mu_{p,\G}\|u_n'\|_{L^2(\G)}^2}{\mu_{p,\G}^{\f p2-1}\|u_n'\|_{L^2(\G)}^2}=\|u_n\|_{L^\infty(\G)}^{p-4}\f{K(4,\G)}{\mu_{p,\G}^{\f p2-2}},
			\end{multline*}
			which by \eqref{eq:maxGN} ensures that $\|u_n\|_{L^\infty(\G)}$ is bounded away from 0 uniformly on $n$. Hence, $u_n\not\to 0$ in $L^\infty(\G)$. On the contrary, we already know that $u_n\to0$ uniformly on $\bigcup_{k\in I}\partial\D_k$. Hence, setting $v_n:=\kappa_n\left(u_n-\|u\|_{L^\infty\left(\bigcup_{k\in I}\partial\D_k\right)}\right)_+$, with $\kappa_n>0$ so that $\|v_n\|_{L^2(\G)}^2=\mu_{p,\G}$, and retracing the proof of Lemma \ref{lem crit ex}, we have that $v_n$ is a bounded minimizing sequence of $E_p(\cdot,\G)$ in $\Hmu(\G)$ and $v_n\equiv0$ on the boundary of every defect of $\G$. Therefore, we can think again of $v_n$ as functions supported on $\Q$ and vanishing in $\Q\setminus\G$. Furthermore, using the periodicity of $\Q$, we can suitably translate $v_n$ on $\Q$ so that each $v_n$ attains its $L^\infty$ norm in the same compact subset of $\Q$, independent of $n$. Now, the boundedness of $v_n$ in $H^1(\Q)$ guarantees that $v_n\rightharpoonup v$ in $H^1(\Q)$, while by construction $v_n\to v$ in $L^\infty(\Q)$ as $n\to+\infty$. On the other hand, $v\not\equiv0$ on $\Q$, since $\|v_n\|_{L^\infty(\Q)}=\|u_n\|_{L^\infty(\G)}+o(1)$ and $u_n\not\to 0$ in $L^\infty(\G)$, so that by weak lower semicontinuity, $v\in\Hmu(\G)$ for some $0<\mu\leq\mu_{p,\G}<\mu_{p,\Q}$. Coupling with \eqref{intro:EQ_4p6}, \eqref{levelTHM1}, \eqref{muG<muQ}, using the fact that $v_n\equiv0$ on $\Q\setminus\G$ and arguing as in \eqref{eq-ragionamento}, there results
			\[
			0<E_p(v,\Q)\leq\liminf_{n\to+\infty}E_p(v_n,\Q)\leq\liminf_{n\to+\infty}E_p(u_n,\G)=\EE_{p,\G}(\mu_{p,\G})=0,
			\]
			which is a contradiction. Therefore, $u\not\equiv0$ on $\G$, whence $u\in H_{\mu_{p,\G}}^1(\G)$ is a ground state of mass $\mu$.
		\end{proof}
		
		\subsection{Grids with $\Z^2$--periodic defects} 
		\label{sub:z2}
		
		Here we give the proof of the part of Theorem \ref{THM:ex_com&Z} devoted to grids with $\Z^2$--periodic defects. Again, in view of Section \ref{subsubsec:Sobolev} it is left to prove
		\begin{itemize}
			\item[(a)] the existence part of item (i),
			\item[(b)] the existence parts of items (ii.1) and (ii.2).
		\end{itemize}
		
		We preliminarily introduce a level criterion for the existence of ground states, whose proof is analogous to that of \cite[Proposition 3.3]{ADST19}. 
		
		\begin{lemma}
			\label{lem:Z2 crit ex}
			Let $\mu>0$, $p\in(2,6)$ and $\G$ be a grid with $\Z^2$--periodic defects. If $\EE_{p,\G}(\mu)<0$, then there exists a ground state of mass $\mu$.
		\end{lemma}
		
		\begin{proof}
			Let $(u_n)_n\subset\Hmu(\G)$ be a minimizing sequence for $E_p$. By Remark \ref{rem:Dper}, the periodicity of $\G$ allows to further assume, without loss of generality, that each $u_n$ attains its $L^\infty$ norm somewhere in $\widetilde{W}_{(0,0)}$. On the other hand, by \eqref{eq:coerc}, $(u_n)_n$ is bounded in $H^1(\G)$, so that $u_n\rightharpoonup u$ in $H^1(\G)$ and $u_n\to u$ in $L_{\text{loc}}^\infty(\G)$ as $n\to+\infty$. By Lemma \ref{lem dichotomy}, either $u\equiv0$ on $\G$ or $u\in\Hmu(\G)$ is a ground state of mass $\mu$.
			
			Assume by contradiction that $u\equiv0$ on $\G$. Since $\|u_n\|_{L^\infty(\G)}=\|u_n\|_{L^\infty(\widetilde{W}_{(0,0)})}$ for every $n$, then $u_n\to0$ in $L^\infty(\G)$, whence $u_n\to 0$ in $L^p(\G)$ for every $p>2$. Thus, 
			\[
			\EE_{p,\G}(\mu)=\lim_{n\to+\infty}E_p(u_n,\G)\geq\liminf_{n\to+\infty}\f12\|u_n'\|_{L^2(\G)}^2\geq0,
			\]
			which contradicts the hypothesis $\EE_{p,\G}(\mu)<0$.
		\end{proof}
		\begin{proof}[Proof of Theorem \ref{THM:ex_com&Z}: grids with $\Z^2$--periodic defects]
			Recall that it is sufficient to prove items (a) and (b) listed before. In addition, by Lemma \ref{lem:Z2 crit ex}, to prove existence of ground states of mass $\mu$ it is enough to show that $\EE_{p,\G}(\mu)$ is strictly negative.
			
			As a consequence, \eqref{levelTHM1} immediately guarantees existence of ground states when $p\in[4,6)$ and $\mu>\mu_{p,\G}$. On the other hand, the case $p\in(4,6)$ with $\mu=\mu_{p,\G}$ can be managed exactly as already done in the previous section for compactly defected grids and grids with $\Z$--periodic defects. The argument proceeds in the same way, with the unique arrangement of taking $(u_n)_n\subset H_{\mu_{p,\G}}^1(\G)$ fulfilling \eqref{eq:maxGN} so that each $u_n$ attains its $L^\infty$ norm in $\widetilde{W}_{(0,0)}$. Thus (b) is proved.
			
			It remains to prove (a). Relying again on Lemma \ref{lem:Z2 crit ex}, this can be done by exhibiting explicit functions in $H_{\mu}^1(\G)$ such that $E_p(u,\G)<0$. 
			
			To this aim, for any fixed $\eps>0$, let $\varphi_\eps:\R^2\to\R$ be the function defined by $\varphi_\eps(x,y):=\kappa_\eps e^{-\eps(|x|+|y|)}$ with
			\begin{equation}
			\label{kappa eps}
			\kappa_\eps:=\left(\f{\eps\mu}2\f{1-e^{-2\eps}}{1+e^{-2\eps}}\right)^{\f12}.
			\end{equation}
			Let then $u_\eps:\Q\to\R$, $v_\eps:\G\to\R$ be the restrictions of $\varphi_\eps$ to $\Q$ and $\G$, respectively. Direct computations show that
			\[
			\int_\Q|u_\eps|^2\,dx=\mu\,,\qquad\int_\Q|u_\eps'|^2\,dx=\eps^2\mu, 
			\]
			so that
			\begin{equation}
			\label{veps'}
			\int_\G|v_\eps'|^2\,dx\leq\eps^2\mu.
			\end{equation}
			Recalling the definition of $\widetilde{W}_{(k_1,k_2)}$, for any $(k_1,k_2)\in\Z^2$, given in Remark \ref{rem:Dper}, suppose now, without loss of generality, that the vertex of $\G$ that is identified with the origin of $\R^2$ belongs to $\widetilde{W}_{(0,0)}$. Denote this vertex by $\v_{(0,0)}\in\widetilde{W}_{(0,0)}$. Since $\G$ is connected, there exists a path $\gamma$ of minimal length in $\G$ joining $\v_{(0,0)}$ with the vertex $\v_{(1,0)}\in\widetilde{W}_{(1,0)}$ given by $\v_{(1,0)}:=\v_{(0,0)}+\vec{v}_1$, with $\vec{v}_1$ as in Definition \ref{def:Z2perD}. By Definition \ref{def:Z2perD}, for every $k_1\in\Z$, $\gamma+k_1\vec{v}_1$ is a path of minimal length in $\G$ joining $\v_{(k_1,0)}:=\v_{(0,0)}+k_1\vec{v}_1\in\widetilde{W}_{(k_1,0)}$ and $\v_{(k_1+1,0)}:=\v_{(0,0)}+(k_1+1)\vec{v}_1\in\widetilde{W}_{(k_1+1,0)}$. Let $\Gamma_0$ be the union of all such paths over $k_1\in\N$. Then $\Gamma_0$ is connected, it starts at $\v_{(0,0)}$ and it follows that
			\[
			\int_{\Gamma_0}|v_\eps|^p\,dx\geq\kappa_\eps^p\int_0^{+\infty} e^{-\eps p x}\,dx=\f{\kappa_\eps^p}{\eps p}.
			\]
			Note that, since the path $\gamma$ joining $\v_{(0,0)}$ and $\v_{(1,0)}$ is of finite length, there exists $k\in\N$ such that $\gamma+k\vec{v}_2$, with $\vec{v}_2$ as in Definition \ref{def:Z2perD}, does not intersect $\gamma$. Let $\overline{k}\in\N$ be the smallest positive integer with this property. Now, for every $k_2\in\N\setminus\{0\}$, consider the path $\Gamma_{k_2}:=\Gamma_0+k_2\overline{k}\vec{v}_2$. By construction, we have that $\Gamma_{k_2}\cap\Gamma_{k_2'}=\emptyset$, for every $k_2\neq k_2'$. Furthermore, since $\Gamma_{k_2}$ starts at the vertex $\v_{(0,k_2)}:=\v_{(0,0)}+k_2\overline{k}\vec{v}_2$, then $v_\eps(\v_{(0,k_2)})=\kappa_\eps e^{-\eps k_2 R}$, with $R:=\overline{k}(|(\vec{v}_{2})_x|+|(\vec{v}_{2})_y|)$, where we denoted by $(\vec{v}_{2})_x,(\vec{v}_{2})_y$ the components of $\vec{v}_2$ with respect to the standard basis of $\R^2$. Thus, for every $k_2\in\N\setminus\{0\}$,
			\[
			\int_{\Gamma_{k_2}}|v_\eps|^p\,dx\geq\kappa_\eps^pe^{-\eps pk_2R}\int_0^{+\infty}e^{-\eps p x}\,dx=\f{\kappa_\eps^pe^{-\eps pk_2R}}{\eps p}
			\]
			and, given that the paths $\Gamma_{k_2}$ are pairwise disjoint, there results
			\begin{equation*}
			\int_\G|v_\eps|^p\,dx\geq\sum_{k_2=0}^{+\infty}\int_{\Gamma_{k_2}}|v_\eps|^p\,dx\geq\sum_{k_2=0}^{+\infty}\f{\kappa_\eps^pe^{-\eps pk_2R}}{\eps p}=\f{\kappa_\eps^p}{\eps p}\f{e^{\eps p R}}{e^{\eps pR}-1}.
			\end{equation*}
			As a consequence, \eqref{kappa eps} yields
			\begin{equation}
			\label{veps p}
			\int_{\G}|v_\eps|^p\,dx\geq\left(\f{\eps\mu}2\f{1-e^{-2\eps}}{1+e^{-2\eps}}\right)^{\f p2}\f{e^{\eps p R}}{\eps p(e^{\eps pR}-1)}\sim C\mu^{\f p2}\eps^{p-2}\qquad\text{as}\quad\eps\to0,
			\end{equation}
			and thus, combining \eqref{veps'} and \eqref{veps p},
			\[
			E_p(v_\eps,\G)<0\qquad\text{as}\quad\eps\to0,\qquad\forall p\in(2,4).
			\]
			Note that this is not enough to conclude that $\EE_{p,\G}(\mu)<0$ for every $p\in(2,4)$ and $\mu>0$, since $\|v_\eps\|_{L^2(\G)}^2<\|u_\eps\|_{L^2(\Q)}^2=\mu$. However, if one fixes $\eps$ sufficiently small so that $E_p(v_\eps,\G)<0$, then, setting $w:=\sqrt{\f{\mu}{\|v_\eps\|_{L^2(\G)}^2}}v_\eps$, there results that $w\in\Hmu(\G)$ and, arguing as in \eqref{E_norm},
			\[
			E_p(w,\G)<\f{\mu}{\|v_\eps\|_{L^2(\G)}^2}E_p(v_\eps,\G)<0,
			\]
			which concludes the proof.
		\end{proof} 
		\begin{remark}
			\label{rem:level_Z2per}
			We have seen in Section \ref{subsec:proofNLS} that the level estimate \eqref{EG<EQ} (as well as \eqref{muG<muQ}) is a by--product of the argument in the proof of Theorem \ref{THM:ex_com&Z} for compactly defected grids and grids with $\Z$--periodic defects. On the contrary, nothing similar can be immediately established for grids with $\Z^2$--periodic defects. As mentioned in Section \ref{sec:openNLS}, it is actually an open problem whether such an estimate holds in this case. However, one can exhibit grids with $\Z^2$--periodic defects satisfying $\EE_{p,\G}(\mu)<\EE_{p,\Q}(\mu)$, for every $p\in(2,6)$ and $\mu\neq\mu_{p,\G}$ for which $\EE_{p,\G}(\mu)$ is attained. An easy example is given by the grid with edges of length two (Figure \ref{fig-z2comez1}). Clearly, one can think of such a grid as a defected grid $\G$ with $\Z^2$--periodic defects, as it can be obtained by removing from $\Q$ the set
			\[
			\bigcup_{(k_1,k_2)\in\Z^2}\Big([2k_1,2(k_1+1)]\times\{2k_2+1\}\Big)\cup\Big(\{2k_1+1\}\times[2k_2,2(k_2+1)]\Big).
			\] 
			Conversely, note that the grid with edges of length one $\Q$ can be interpreted as the union of two grids with edges of length two, say $\widetilde{\G}$, $\widehat{\G}$, in such a way that $\widetilde{\G}\cap\widehat{\G}$  is a countable number of vertices of $\Q$. Therefore, one can write
			\[
			E_p(u,\Q)=E_p(u,\widetilde{\G})+E_p(u,\widehat{\G}),\qquad\forall u\in\Hmu(\Q),
			\]
			entailing
			\[
			\min\left\{\f{E_p(u,\widetilde{\G})}{\int_{\widetilde{\G}}|u|^2\,dx},\f{E_p(u,\widehat{\G})}{\int_{\widehat{\G}}|u|^2\,dx}\right\}\leq\f{E_p(u,\Q)}\mu.
			\]
			Indeed, if on the contrary we had
			\[
			\f{E_p(u,\widetilde{\G})}{\int_{\widetilde{\G}}|u|^2\,dx}>\f{E_p(u,\Q)}\mu\qquad\text{and}\qquad\f{E_p(u,\widehat{\G})}{\int_{\widehat{\G}}|u|^2\,dx}>\f{E_p(u,\Q)}\mu,
			\]
			then
			\[
			E_p(u,\Q)=E_p(u,\widetilde{\G})+E_p(u,\widehat{\G})>\f{\int_{\widetilde{\G}}|u|^2\,dx}\mu E_p(u,\Q)+\f{\int_{\widehat{\G}}|u|^2\,dx}\mu E_p(u,\Q)=E_p(u,\Q),
			\]
			which is a contradiction. Assume for instance that $\f{E_p(u,\widetilde{\G})}{\int_{\widetilde{\G}}|u|^2\,dx}\leq\f{E_p(u,\Q)}\mu$ and, for the sake of simplicity, denote by $\G$ the grid $\widetilde{\G}$. Hence, we can set $v:=\kappa u_{\mid\G}$, with $\kappa:=\sqrt{\mu/\int_{\G}|u|^2\,dx}$, so that $v\in\Hmu(\G)$ and, as $\kappa>1$,
			\[
			\begin{split}
			E_p(v,\G)=\f{\kappa^2}2\int_{\G}|u'|^2\,dx-\f{\kappa^p}p\int_{\G}|u|^p\,dx<&\,\kappa^2E_p(u,\G)\\
			=&\,\f\mu{\int_{\G}|u|^2\,dx}E_p(u,\G)\leq E_p(u,\Q).
			\end{split}
			\]
			This immediately shows that $\EE_{p,\G}(\mu)<\EE_{p,\Q}(\mu)$ whenever a ground state with mass $\mu$ exists on $\Q$ and, arguing as in Section \ref{subsec:proofNLS}, it shows that this is true for every $p\in(2,6)$ and $\mu\neq\mu_{p,\G}$ for which $\EE_{p,\G}(\mu)$ is attained. 
		\end{remark}
		
		\subsection{Defected grids without periodic structures.}
		\label{subsec:nonex}
		
		Here we prove Theorem \ref{thm:noground}. The proof consists of detecting two suitable counterexamples, one for each claim of the statement. Accordingly, we divide the proof in two parts.
		
		\begin{proof}[Proof of Theorem \ref{thm:noground}: item (i)]
			We need to exhibit a defected grid that satisfies Definition \ref{def:ZperD}(ii), but not Definition \ref{def:ZperD}(i), and that does not admit any ground state. Before stating the technical details of our construction, let us give a heuristic explanation of it. 
			
			The proof is by contradiction, as we assume that a ground state exists and we use it to construct a new function with a strictly lower energy. To this end, we exploit the general property in Proposition \ref{prop:E>0}, ensuring that any ground state realizes strictly positive energy on every edge out of a compact set.
			
			Let $u\in\Hmu(\G)$ be a ground state on a defected grid $\G$ and let $\overline{K}$ be the compact set out of which the energy of $u$ is positive on every edge. Suppose for a moment that the grid we are considering shares the following property: somewhere in the grid, it is possible to find a copy of $\overline{K}$ (that is, there exists $\overline{K}'\subset\G$ which corresponds to $\overline{K}$ up to a suitable translation in $\R^2$) surrounded by an arbitrary number of removed edges. On such a grid, one could then imagine to define a new function $v$ by translating $u$ on $\G$ so to make the restriction of $u$ to $\overline{K}$ to be moved to $\overline{K}'$, and get rid of the restriction of $u$ on the edges around $\overline{K}$ that correspond to removed edges around $\overline{K}'$. Roughly, $v$ is obtained by moving $u$ somewhere in the grid where some of the edges on which such a translation would realize positive energy are missing. 
			
			The main difficulty in making rigorous the above formal argument stems from the fact that Proposition \ref{prop:E>0} does not give any a priori information on how the compact set $\overline{K}$ can be. Therefore, we need to construct a grid in which every compact subset intersecting the defects is repeated infinitely many times in the graph, each time surrounded by an increasing number of defected edges. 
			
			\medskip
			For the sake of clarity, the proof is divided in two steps. The precise construction of a grid with this sort of self--similar structure is given in Step 1, whereas Step 2 deals with the construction of the function $v$.
			
			\emph{Step 1: construction of $\G$.} As by definition $\G:=\Q\setminus\D$, with $\D:=\cup_{k\in I}\D_k$ and $(\D_k)_{k\in I}$ the family of all the defects of $\G$, it is sufficient to describe $\D$. Preliminarily, we say that $\D$ is the union of infinitely many vertical edges of $\Q$ with ordinates between $0$ and $1$, i.e
			\begin{equation}
			\label{eq:defset}
			\D\subset\bigcup_{k\in\Z}\{k\}\times[0,1]\subset\Q.
			\end{equation}  
			As all the defects belong to $\bigcup_{k\in\Z}\{k\}\times[0,1]$, it is clear that $\G$ satisfies Definition \ref{def:ZperD}(ii). We describe $\D$ in a symbolic way. In particular, we define a suitable map $F:\Z\to\{0,1\}$ and set $\D:=F^{-1}(0)\times[0,1]$, in the sense that an edge $\{k\}\times[0,1]$ of $\Q$ belongs to $\D$ if and only if $F(k)=0$. Like this, we can identify $\D$ with a suitable sequence of 0 and 1. For instance, the sequence $1000101$ corresponds to a subset of $\bigcup_{k\in\Z}\{k\}\times[0,1]\subset\Q$ of seven consecutive edges, where the second, the third, the fourth and the sixth ones belong to $\D$ (whereas the first, the fifth and the seventh ones belong to $\G$).
			
			To define $F$, let us first start by introducing several binary sequences that we use as building blocks in our construction. For every $n\geq2$, let $\sigma_n$ be the sequence given by $n$ repetitions of the block 010, i.e.
			\[
			\sigma_n:=\underbrace{010\dots010}_{n\text{ times}}.
			\]
			Then, let
			\[
			\begin{split}
			\Bq_1:=&\sigma_1 111\sigma_1,\\
			\Bq_2:=&\Bq_1\sigma_2\Bq_1\sigma_2\Bq_1,
			\end{split}
			\]
			and, for every $n\geq3$ and $1\leq k\leq 2^{n-2}$, let
			\[
			\Bq_{2^{n-2}+k}=\Bq_k\sigma_n\Bq_{k-1}\sigma_n\dots\sigma_n\Bq_1\sigma_n\Bq_{2^{n-2}}\sigma_n\Bq_1\sigma_n\dots\sigma_n\Bq_{k-1}\sigma_n\Bq_k.
			\]
			Note that $\left(\Bq_i\right)_{i\in\N}$ forms an encapsulated sequence of blocks, in the sense that $\Bq_{i+1}$ is obtained by adding both on the left and on the right of $\Bq_i$ a term $\Bq_k\sigma_{k'}$, for suitable $k,k'$ depending on $i$. Furthermore, given $i\in\N$, for every $n$ so that $2^{n-2}>i+1$, by construction we have that 
			\[
			\Bq_{2^{n-2}+i+1}=\Bq_{i+1}\sigma_n\Bq_i\sigma_n\dots\sigma_n\Bq_i\sigma_n\Bq_{i+1}.
			\]
			This shows that
			\begin{equation}
			\label{eq-noperiod}
			\forall i,\,N\in\N\setminus\{0\},\:\exists j\in\N\setminus\{0\}\quad\text{such that}\:\Bq_j\:\text{contains the block}\:\sigma_N\Bq_i\sigma_N.
			\end{equation}		
			Note also that, denoting by $|\Bq_n|$ the number of digits in the block $\Bq_n$, then $\f{|\Bq_n|-1}2\in\N$ for every $n$, as $\Bq_n$ is by construction the union of an integer number of blocks of three digits each.
			
			Now, we can define $F:\Z\to\{0,1\}$. Roughly speaking, we see any block $\Bq_n$ as a juxtaposition of images of elements of $\Z$ through the function $F$. More precisely, if we denote by $F\mid_a^b$ the sequence of digits given by $F(a)F(a+1)\dots F(b-1)F(b)$, with $a,b\in\Z$, $a<b$, then
			\[
			F\mid_{-\f{|\Bq_n|-1}2}^{\f{|\Bq_n|-1}2}:=\Bq_n,\qquad\forall n\in\N.
			\]
			A first consequence of the definition is that every defect of $\G$ contains either one or two edges. Indeed, by \eqref{eq:defset}, any defect of $\G$ must be of the form
			\[
			\{a,a+1,\dots,b-1,b\}\times[0,1],\qquad a,b\in\Z, \qquad a<b,
			\]
			with $F(a-1)=F(b+1)=1$ and $F(z)=0$ for every $a\leq z\leq b$. Since $F\mid_{a-1}^{b+1}$ is a subsequence of some blocks $\Bq_i$ and, by construction, the only sequences with a given number of 0 enclosed between two 1 are 101 and 1001, there is no defect of $\G$ with more than two edges. 
			An immediate consequence is that the defects of $\G$ are uniformly bounded, so that $\G$ supports \eqref{eq:s2d} by Theorem \ref{THM:s2d_bound}. Furthermore, \eqref{eq-noperiod} guarantees that $\G$ violates Definition \ref{def:ZperD}(i).
			
			\emph{Step 2: $\G$ does not admit any ground state of mass $\mu$, for every $p\in(2,6)$ and every $\mu>0$.} Fix $p\in(2,6)$ and $\mu>0$ and assume by contradiction that there exists $u\in\Hmu(\G)$ such that $E_p(u,\G)=\EE_{p,\G}(\mu)$. By \eqref{eq-sign}, it is not restrictive to take $u>0$ on $\G$. By Proposition \ref{prop:E>0}, there exists a compact subset $K\subset\G$ such that $E_p(u,e)>0$ for every edge $e\in\G\setminus K$. Set
			\[
			A:=K\cap\partial\D\cap\left(\bigcup_{k\in\Z}\left(\{k\}\times[0,1]\right)\right),
			\]     
			with $\partial\D:=\cup_{k\in I}\partial\D_k$ (note that, possibly enlarging $K$, one can assume without loss of generality, that $A\neq\emptyset$). By construction, there exist $k_1,\dots,k_{\overline{N}}\in\Z$ such that $A=\{k_1,\dots,k_{\overline{N}}\}\times[0,1]$. Furthemore, letting $\overline{n}$ be the smallest natural number such that $\{k_1,\dots,k_{\overline{N}}\}\subset\left[-\f{|\Bq_{\overline{n}}|-1}2,\f{|\Bq_{\overline{n}}|-1}2\right]$, we fix a compact set $\overline{K}\subset\G$ such that $\overline{K}\supseteq K$ and such that
			\begin{equation}
			\label{eq:AA}
			\overline{A}:=\overline{K}\cap\bigcup_{k\in\Z}\left(\{k\}\times[0,1]\right)=\left(\bigcup_{k=-\f{|\Bq_{\overline{n}}|-1}2}^{\f{|\Bq_{\overline{n}}|-1}2}\{k\}\times[0,1]\right)\cap (F^{-1}(1)\times[0,1]).
			\end{equation}
			One can see that $E_p(u,e)>0$ for every edge $e\in\left(\G\setminus\overline{K}\right)\cap\bigcup_{k\in\Z}\left(\{k\}\times[0,1]\right)$. Also, by definition of $\D$, there exist infinitely many $n\in\N$ such that $n>\f{|\Bq_{\overline{n}}-1|}2$, $\{n-1,n,n+1\}\times[0,1]\subset\G\setminus\overline{K}$ and $\{-n-1,-n,-n+1\}\times[0,1]\subset\G\setminus\overline{K}$. Indeed, the previous sets are two triples of edges, with middle edges denoted by $\{n\}\times[0,1]$ and $\{-n\}\times[0,1]$, which are consecutive in $\G$. Since each of these triplets corresponds to a sequence 111 contained in a repetition of the block $\Bq_1$, combining \eqref{eq-noperiod} (with $i=1$) with the compactness of $\overline{K}$, one obtains the claim. Denote then by $\Lambda$ the set of all natural numbers $n$ for which this holds, and set 
			\[
			B:=\bigcup_{n\in\Lambda}\{-n-1,-n+1,n-1,n+1\}\times[0,1].
			\]
			According to the definition of $\Lambda$,  $B$ contains the first and the third edge of each triplet of consecutive edges in $\left(\bigcup_{k\in\Z}\{k\}\times[0,1]\right)\cap\left(\G\setminus\overline{A}\right)$. In addition, since $B\cap\overline{K}=\emptyset$, we have that
			\begin{equation}
			\label{EB>0}
			E_p(u,B)>0.
			\end{equation}
			To conclude the proof, we now exploit the previous construction to build a function $v\in H^1(\G)$ such that 
			\[
			\|v\|_{L^2(\G)}^2<\mu\qquad\text{and}\qquad E_p(v,\G)<E_p(u,\G).
			\]
			By \eqref{eq:infnonpos}, this entails $E_p(v,\G)<0$ and therefore, arguing as in the final part of the proof of Theorem \ref{THM:ex_com&Z} for grids with $\Z^2$--periodic defects, one can find another function $w\in\Hmu(\G)$ such that $E_p(w,\G)<E_p(u,\G)=\EE_{p,\G}(\mu)$, which is a contradiction.
			
			To exhibit such a function $v$ we proceed as follows. For every $\eps>0$, set $u_\eps:=(u-\eps)_+$. Clearly, $(u_\eps)_{\eps}\subset H^1(\G)$, each $u_\eps$ has compact support and $u_\eps\to u$ in $H^1(\G)$ as $\eps\to0$. Furthermore, as $u>0$,
			\begin{equation}
			\label{ueps_Abar}
			\text{supp}(u_\eps)\cap\overline{A}\neq\emptyset,
			\end{equation}
			provided $\eps$ is sufficiently small. Set then
			\[
			\begin{split}
			\underline{k}_\eps:=&\min\{k\in\Z\,:\,u_\eps\not\equiv0\text{ on }\{k\}\times[0,1]\},\\
			\overline{k}_\eps:=&\max\{k\in\Z\,:\,u_\eps\not\equiv0\text{ on }\{k\}\times[0,1]\}
			\end{split}
			\]
			and
			\[
			C_\eps:=\left(\{k\in\Z\,:\,\underline{k}_\eps\leq k\leq\overline{k}_\eps\}\times[0,1]\right)\cap\G.
			\]
			Since $\text{supp}(u_\eps)$ is compact, $\underline{k}_\eps,\overline{k}_\eps$ are well--defined and $\overline{A}\subset C_\eps$ by \eqref{ueps_Abar}. On the one hand, let
			\[
			B_\eps:=B\cap C_\eps.
			\]
			As $B_{\eps}\uparrow B\neq\emptyset$ (in the sense of inclusions), $B_\eps\neq\emptyset$ for $\eps$ small. Furthermore, since
			\begin{equation}
			\label{eq-convener}
			E_p(u_\eps,B_\eps)=E_p(u_\eps,B)\qquad\text{and}\qquad\lim_{\eps\to0}E_p(u_\eps,B)= E_p(u,B),
			\end{equation}			
			\eqref{EB>0} yields 
			\begin{equation}
			\label{EBeps>0}
			E_p(u_\eps,B_\eps)>0\qquad\text{for}\:\eps\:\text{small enough}.
			\end{equation}
			On the other hand, let $N_\eps:=\overline{k}_\eps-\underline{k}_\eps$ and recall that, by \eqref{eq-noperiod} and the definition of $F$, there exist $a_\eps,b_\eps\in\Z$ such that $F\mid_{a_\eps}^{b_\eps}=\sigma_{N_\eps}\Bq_{\overline{n}}\sigma_{N_\eps}$. As a consequence, we can take $k_\eps^1,\,k_\eps^2\in\Z$, with $a_\eps<k_\eps^1<k_\eps^2<b_\eps$, such that $F\mid_{k_\eps^1}^{k_\eps^2}=\Bq_{\overline{n}}$ and define
			\[
			\overline{A}_\eps:=\G\cap\left(\{k\in\Z\,:\,k_\eps^1\leq k\leq k_\eps^2\}\times [0,1]\right).
			\]
			Since by construction both $\overline{A}$ and $\overline{A}_\eps$ corresponds through $F$ to the block $\Bq_{\overline{n}}$, then $\overline{A}_\eps$ is nothing but a copy of $\overline{A}$ that presents both on its right and on its left the structure represented by the block $\sigma_{N_\eps}$. Finally, denote by
			\begin{equation}
			\label{eq:deps}
			d_\eps:=\left|k_\eps^2-\f{|\Bq_n|-1}2\right|
			\end{equation}			
			the distance in $\G$ between the edge at the right end of $\overline{A}$ and the edge at the right end of $\overline{A}_\eps$.
			
			Now, let $\widetilde{u}_\eps:=u_{\eps\mid\G\setminus B_\eps}$. By \eqref{EBeps>0}, as soon as $\eps$ is sufficiently small, we have
			\begin{equation}
			\label{eq_ueps_tilde}
			\begin{split}
			\|\widetilde{u}_\eps\|_{L^2(\G\setminus B_\eps)}^2=&\|u_\eps\|_{L^2(\G)}^2-\|u_\eps\|_{L^2(B_\eps)}^2<\mu\\
			E_p(\widetilde{u}_\eps,\G\setminus B_\eps)=&E_p(u_\eps,\G)-E_p(u_\eps,B_\eps)<E_p(u_\eps,\G)\,.
			\end{split}
			\end{equation}
			Note that the unique difference between $\widetilde{u}_\eps$ and $u_\eps$ is that $\widetilde{u}_\eps$ is not defined on both the left and right edges of each triplet of consecutive edges in $C_\eps\setminus\overline{A}$. This means that, in the definition of $\widetilde{u}_\eps$, each triplet of consecutive edges in $C_\eps$ has been replaced by a unique edge (the middle one) enclosed between two ``removed edges''. Exploiting the symbolic representation, it corresponds to replace each subset of $C_\eps$ mapped by $F$ to 111 with a subset mapped by $F$ to 010. As a consequence, the set $C_\eps\setminus\left(\overline{A}\cup B_\eps\right)$ corresponds through the map $F$ to two disjoint sequences of consecutive blocks 010, say $\sigma_{n_\eps^1}, \sigma_{n_\eps^2}$, for some $n_\eps^1,\,n_\eps^2\in\N$. The sequence $\sigma_{n_\eps^1}$ corresponds to the subset of $C_\eps\setminus\left(\overline{A}\cup B_\eps\right)$ on the left of $\overline{A}$, whereas $\sigma_{n_\eps^2}$ describes the one on the right of $\overline{A}$. Since the total number of edges in $C_\eps$ is smaller than $N_\eps$, there results $n_\eps^1<N_\eps$, $n_\eps^2<N_\eps$. 
			
			As a last step, let $v_\eps$ be the translation on $\G$ of $\widetilde{u}_\eps$ by the vector $(d_\eps,0)$. In order to see that it is well defined, one can argue as follows. By definition of $d_\eps$, the restriction of $\widetilde{u}_\eps$ to $\overline{A}$ is moved by this translation to the set $\overline{A}_\eps$. Moreover, recall that both on the left and on the right of $\overline{A}_\eps$, the structure is described through $F$ by the sequence $\sigma_{N_\eps}$. On this subset of $\G$, $v_\eps$ is given by the translation of the restriction of $\widetilde{u}_\eps$ to $C_\eps\setminus\left(\overline{A}\cup B_\eps\right)$. The portion of this set on the left (resp. right) of $\overline{A}$ corresponds through $F$ to $\sigma_{n_\eps^1}$ (resp. $\sigma_{n_\eps^2}$). Since $n_\eps^1<N_\eps$ (resp. $n_\eps^2<N_\eps$), the sequence $\sigma_{N_\eps}$ can be seen as $\sigma_{N_\eps-n_\eps^1}\sigma_{n_\eps^1}$, given by the $N_\eps-n_\eps^1$ blocks 010 of $\sigma_{N_\eps-n_\eps^1}$ followed by the remaining $n_\eps^1$ ones of $\sigma_{n_\eps^1}$ (and the same for $\sigma_{n_\eps^2}\sigma_{N_\eps-n_\eps^2}$). Hence, the subset of $\G$ corresponding to the sequence $\sigma_{N_\eps}$ on the left of $\overline{A}_\eps$ contains a copy of the subset of $C_\eps\setminus\left(\overline{A}\cup B_\eps\right)$ on the left of $\overline{A}$, and the same is true on the right. Since the translation of $\widetilde{u}_\eps$ on $\G\setminus\bigcup_{k\in\Z}\left(\{k\}\times[0,1]\right)$ is straightforward, we have $v_\eps\in H^1(\G)$ and, by \eqref{eq_ueps_tilde},
			\[
			\begin{split}
			\|v_\eps\|_{L^2(\G)}^2=&\|\widetilde{u}_\eps\|_{L^2(\G\setminus B_\eps)}^2<\mu\\[.2cm]
			E_p(v_\eps,\G)=&E_p(\widetilde{u}_\eps,\G\setminus B_\eps)=E_p(u_\eps,\G)-E_p(u_\eps,B_\eps).
			\end{split}
			\]
			Thus, combining \eqref{eq-convener} with \eqref{EB>0} and the fact that $\lim_{\eps\to0}E_p(u_\eps,\G)= E_p(u,\G)$, there results that, whenever $\eps$ is sufficiently small, $\|v_\eps\|_{L^2(\G)}^2<\mu$ and $E_p(v_\eps,\G)<E_p(u,\G)$, which completes the proof.
		\end{proof}
		
		Once proved the former part of Theorem \ref{thm:noground}, the latter follows almost immediately, as the structure of the suitable counterexample strongly relies on that of the previous one.
		
		\begin{proof}[Proof of Theorem \ref{thm:noground}: item (ii)]
			Here it is necessary to exhibit a defected grid that satisfies Definition \ref{def:ZperD}(i), but neither Definition \ref{def:ZperD}(ii) nor Definition \ref{def:Z2perD}, and that does not admit any ground state.
			
			To this aim, consider the set $\D$ defined in Step 1 of the proof of Theorem \ref{thm:noground}(i) and rename it $\D^0$. Then, for every $\ell\in\Z\setminus\{0\}$, let $\D^\ell$ be the translation of $\D^0$  by the vector $\ell(0,2)$ and define
			\[
			\widetilde{\D}:=\bigcup_{\ell\in\Z}\D^\ell\qquad\text{and}\qquad\widetilde{\G}:=\Q\setminus\widetilde{\D}.
			\]
			By construction, $\widetilde{\G}$ fulfills Definition \ref{def:ZperD}(i) with $\vec{v}=(0,2)$, but not Definition \ref{def:ZperD}(ii), as every $\D^\ell$ is unbounded in the direction of $\vec{v}^\bot$. On the other hand, again by construction $\widetilde{\G}$ cannot fulfill Definition \ref{def:Z2perD} too, as every $\D^\ell$ is not periodic in the direction of $\vec{v}^\bot$.
			
			It is then left to prove that $\widetilde{\G}$ does not support any ground state. This is completely analogous to the previos case, if one suitably chooses the compact set $\overline{K}\subset\widetilde{\G}$ such that $E_p(u,e)>0$ for every edge $e\in\widetilde{\G}\setminus\overline{K}$. More precisely, denoting by $\overline{A}^0$ the set defined in \eqref{eq:AA} and, for every $\ell\in\Z\setminus\{0\}$, by $\overline{A}^\ell$ the translation of $\overline{A}^0$  by the vector $\ell(0,2)$, it is sufficient to take $\overline{K}$ so that
			\[
			\overline{K}\cap\bigcup_{\ell\in\Z\cap[\ell_1,\ell_2]}\bigcup_{k\in\Z}\left(\{k\}\times[2\ell,2\ell+1]\right)=\bigcup_{\ell\in\Z\cap[\ell_1,\ell_2]}\overline{A}^\ell,
			\]
			where
			\begin{gather*}
			\ell_1:=\min\left\{\ell\in\Z:\overline{K}\cap\partial\D_\ell\cap\bigg(\bigcup_{k\in\Z}\left(\{k\}\times[2\ell,2\ell+1]\right)\bigg)\neq\emptyset\right\},\\[.2cm]
			\ell_2:=\max\left\{\ell\in\Z:\overline{K}\cap\partial\D_\ell\cap\bigg(\bigcup_{k\in\Z}\left(\{k\}\times[2\ell,2\ell+1]\right)\bigg)\neq\emptyset\right\}.
			\end{gather*}
			Indeed, if one repeats for any of these sets $\overline{A}_\ell$ the same construction in the Step 2 of the proof of Theorem \ref{thm:noground}(i) starting from the set $\overline{A}$, then assuming the existence of a ground states yields the same contradiction. The unique proviso is that, when one defines $v_\eps$, it is necessary to translate by a vector $(d_\eps^*,0)$, where $d_\eps^*:=\max_{\ell\in\Z\cap[\ell_1,\ell_2]}d_\eps^\ell$ and $d_\eps^\ell$ is defined as in \eqref{eq:deps} for every $\overline{A}^\ell$.
		\end{proof}
		 


\section*{Appendix}

We give here the proof of Lemma \ref{lem_dDconn}, which states the connectedness of the boundary of any defect in a defected grid $\G$.

\begin{proof}[Proof of Lemma \ref{lem_dDconn}]
	Preliminarily, note that, if $\D$ consists of a single edge $e$, then $\partial\D=(\Sq_e'\cup\Sq_e'')\setminus\{e\}$ is clearly connected. If on the contrary $\D$ contains at least two edges, let $\v,\w$ be two vertices in $\partial\D$ and let us prove that they are connected by a path in $\partial\D$.
	
	First, whenever they belong to the same edge of $\partial\D$, there clearly exists a path in $\partial\D$ connecting them. Assume then that they do not belong to the same edge of $\partial\D$. On the one hand, if they belong to the same cell $\Sq_0$ of $\Q$, then one easily sees that there exists a path $\gamma\subset\Sq_0$ which connects them. On the other hand, if they do not belong to the same cell of $\Q$, then, by Definition \ref{def:D and dD}, there exists nevertheless a finite sequence of cells $\left(\mathcal{S}_j\right)_{j=0}^N\subset\Q$ such that $\v\in\mathcal{S}_0$, $\w\in\mathcal{S}_N$, and, for every $j=0,\,\dots,\,N-1$, there exists an edge $e_j\in\D$ for which $\mathcal{S}_j\cap\mathcal{S}_{j+1}=\{e_j\}$. Hence, also here there exists a path $\gamma\subset\bigcup_{j=0}^N\mathcal{S}_j$ joining $\v$ and $\w$. Note that it is not restrictive to assume that $\gamma$ is simple.
	
	By construction, $\gamma$ contains edges of $\partial\D$ and $\D$ only. Therefore, there is a sequence of vertices $\v_1,\v_2,\dots,\v_M\in\gamma\cap\partial\D$ (with possibly $\v_1=\v$ or $\v_M=\w$) such that, denoting by $\gamma_{x,y}$ the portion of $\gamma$ from $x$ to $y$,
	\begin{itemize}
		\item[(i)] $\gamma_{\v,\v_1}\subset\partial\D$ and $\gamma_{\v_M,\w}\subset\partial\D$;
		\item[(ii)] for every odd $j\in\{1,\dots,M-1\}$, $\gamma_{\v_j,\v_{j+1}}\setminus\{\v_j,\v_{j+1}\}\subset\D$;
		\item[(iii)] for every even $j\in\{1,\dots,M-1\}$, $\gamma_{\v_j,\v_{j+1}}\subset\partial\D$.
	\end{itemize}
	We can thus think of $\gamma$ as the union of $M+1$ simple paths, each of which contains either edges of $\partial\D$ only or edges of $\D$ only, and such that any path that belongs to $\partial\D$, up to the last one, is followed by a path that belongs to $\D$. Note that, with a little abuse of notation, we denote $\gamma_{\v,\v_1}$ (resp. $\gamma_{\v_M,\w}$) as a path even if $\v_1=\v$ (resp. $\v_M=\w$). Clearly, in this case, such path plays no role and can be neglected in the following arguments.
	
	Let $d(\gamma)$ be the number of these paths $\gamma_{\v_j,\v_{j+1}}$ that belong to $\D$. If $d(\gamma)=0$ the proof if complete. If, on the contrary, $d(\gamma)\neq0$, then it is sufficient to prove that there exists another simple path $\widetilde{\gamma}\subset\partial\D\cup\D$ starting at $\v$ and ending at $\w$ such that $d(\widetilde{\gamma})=d(\gamma)-1$. Indeed, if this property is true, then, by an iterative procedure, one obtains that there exists a path in $\partial\D$ connecting $\v$ and $\w$. 
	
	We prove the claim by induction. Let us start by the inductive step. Suppose that, for a given $n\in\N\setminus\{0,1\}$, if $\gamma'\subset\partial\D\cup\D$ is a simple path joining two vertices of $\partial\D$ such that $d(\gamma')=n$, then there exists another simple path $\gamma''\subset\partial\D\cup\D$, joining the same vertices of $\gamma'$, such that $d(\gamma'')=n-1$. Now, let $\gamma$ be a simple path in $\partial\D\cup\D$, joining two vertices of $\partial\D$, such that $d(\gamma)=n+1$. Then $\gamma$ can be seen as the union of two simple paths $\gamma_n,\gamma_1$, with $d(\gamma_n)=n$ and $d(\gamma_1)=1$, such that by construction
	\begin{itemize}
		\item[--] the starting vertex of $\gamma_1$ coincides with the ending vertex of $\gamma_n$;
		\item[--] the extremal vertices of $\gamma_n$ belong to $\partial\D$.
	\end{itemize}
	Hence, by the inductive assumption, it is possible to replace $\gamma_n$ with a new path $\gamma_{n-1}\subset\partial\D\cup\D$, whose extremal vertices coincide with those of $\gamma_n$, such that $d(\gamma_{n-1})=n-1$. Therefore, $\gamma_{n-1}\cup\gamma_1$ is a simple path, joining the same vertices of $\gamma$, with $d(\gamma_{n-1}\cup\gamma_1)=n$.
	
	It is then left to prove that, if $\gamma\subset\partial\D\cup\D$ is a simple path joining two vertices of $\partial\D$ such that $d(\gamma)=1$, then there exists another simple path $\gamma''\subset\partial\D\cup\D$, joining the same vertices of $\gamma$, such that $d(\gamma'')=0$. Therefore, let $\gamma\subset\partial\D\cup\D$ connect $\v,\w\in\partial\D$, and let $\v_1,\v_2$ be the two vertices in $\gamma$ such that $\gamma_{\v,\v_1}\subset\partial\D$, $\gamma_{\v_1,\v_2}\subset\D$ and $\gamma_{\v_2,\w}\in\partial\D$. Since $\gamma\not\subset\G$ and $\G$ is connected, there is another simple path $\gamma'\subset\G$ joining $\v$ and $\w$. If $\gamma'\subset\partial\D$, then one sets $\gamma''=\gamma'$ and the proof is complete.
	
	Assume on the contrary that $\left(\G\setminus\partial\D\right)\cap\gamma'\neq\emptyset$. By construction, there exists a simple closed curve $\Gamma\subseteq\gamma\cup\gamma'$ such that $\gamma_{\v_1,\v_2}\subset\Gamma$ (note that $\v_1,\v_2\in\Gamma$). Since $\v$ is connected to $\v_1$ in $\partial\D$ along $\gamma_{\v,\v_1}$, whereas $\w$ is connected to $\v_2$ in $\partial\D$ along $\gamma_{\v_2,\w}$, in order to conclude it is sufficient to connect $\v_1$ and $\v_2$ along a new path contained in $\partial\D$. 
	
	To exhibit such a path, we argue as follows. If $\Gamma\cap\gamma'\subset\partial\D$, then again the proof is complete. Hence, suppose $\left(\G\setminus\partial\D\right)\cap\left(\Gamma\cap\gamma'\right)\neq\emptyset$. Since $\Gamma$ is a simple closed curve in $\Q$, it divides $\Q$ in two regions, one of which contains finitely many edges. Let $\Omega\subset\Q$ denote this finite region bounded by $\Gamma$ and set $\overline{\Omega}:=\Omega\cup\Gamma$. As $\Gamma\cap\gamma'$ is not contained in $\partial\D$, for every edge $e\in\left(\Gamma\cap\gamma'\right)\setminus\partial\D$, either the cell $\mathcal{S}_e'$ or the cell $\mathcal{S}_e''$ belongs to $\overline{\Omega}\cap\left(\Q\setminus\D\right)$. Let $\Omega_1$ be the union of all the cells in $\overline{\Omega}$ with at least one edge in $\left(\Gamma\cap\gamma'\right)\setminus\partial\D$. By construction, $\Omega_1$ is given by union of finitely many disjoint bounded subsets of $\overline{\Omega}\setminus\D$, each enclosed by a simple closed curve in $\overline{\Omega}$. Let $\partial\Omega_1\subset\overline{\Omega}$ denote the union of these simple closed curves. Note that, by definition of $\Omega_1$, $\Gamma\setminus\left(\partial\D\cup\D\right)\subset\partial\Omega_1$ and $\left(\Gamma\cap\partial\D\right)\cup\left(\partial\Omega_1\cap\Omega\right)$ is connected. Furthermore, $\Omega_1\cap\partial\D\subset\partial\Omega_1$. Indeed, for every edge $e\in\Omega_1\cap\partial\D$, at least one between $\mathcal{S}_e'$ and $\mathcal{S}_e''$ is not entirely contained in $\Omega_1$. This is a direct consequence of Definition \ref{def:D and dD}, as if $e\in\partial\D$, then at least one between $\mathcal{S}_e'$ and $\mathcal{S}_e''$ contains one edge of $\D$, whilst $\Omega_1\cap\D=\emptyset$. As a consequence $e\not\in\Omega_1\setminus\partial\Omega_1$.
	
	If $\left(\partial\Omega_1\cap\Omega\right)\setminus\partial\D=\emptyset$, then $\left(\Gamma\cap\partial\D\right)\cup\left(\partial\Omega_1\cap\Omega\right)\subset\partial\D$ and, thus, we proved the claim. If on the contrary $\left(\partial\Omega_1\cap\Omega\right)\setminus\partial\D\neq\emptyset$, then let $\Omega_1^2$ be the union of all the cells in $\Omega$ with at least one edge in $\left(\partial\Omega_1\cap\Omega\right)\setminus\partial\D$, and set $\Omega_2:=\Omega_1\cup\Omega_1^2$. Therefore, $\Omega_2$ is given again by the union of finitely many disjoint bounded subsets of $\overline{\Omega}\setminus\D$, each enclosed by a simple closed curve in $\overline{\Omega}$. Letting $\partial\Omega_2$ denote the union of such simple closed curves, we have that $\partial\Omega_1\cap\Gamma\subseteq\partial\Omega_2\cap\Gamma$, $\Omega_2\cap\partial\D\subset\partial\Omega_2$ and $\left(\Gamma\cap\partial\D\right)\cup\left(\partial\Omega_2\cap\Omega\right)$ is connected. Since $\Omega$ contains finitely many edges of $\Q$, possibly iterating the previous construction $K$ times, for some suitable $K\in\N$, we end up with a subset $\Omega_K\subset\overline{\Omega}$ with the following properties:
	\begin{itemize}
		\item[$(a)$] $\Omega_K$ is the union of finitely many disjoint bounded subsets of $\overline{\Omega}$, each enclosed by a simple closed curves;
		\item[$(b)$] letting $\partial\Omega_K$ be the union of all the simple closed curves enclosing $\Omega_K$, then $\Gamma\setminus\left(\partial\D\cup\D\right)\subset \partial\Omega_K$ and $\left(\Gamma\cap\partial\D\right)\cup\left(\partial\Omega_K\cap\Omega\right)$ is connected;
		\item[$(c)$] $\left(\partial\Omega_K\cap\Omega\right)\setminus\partial\D=\emptyset$.
	\end{itemize}
	Then, arguing as before, $\left(\Gamma\cap\partial\D\right)\cup\left(\partial\Omega_K\cap\Omega\right)$ is a subset of $\partial\D$ containing both $\v_1$ and $\v_2$, which concludes the proof.
\end{proof}



\end{document}